%% file: cvww_wave_borne.tex
\documentclass[11pt,reqno]{amsart}
\usepackage[scale=0.8,twoside=false,letterpaper]{geometry}

\usepackage[dvipsnames]{xcolor}

\usepackage[english]{babel}
\usepackage{csquotes} 
\usepackage[bibstyle=ieee,citestyle=numeric-comp,dashed=true,sorting=nty,backref=true,maxnames=7]{biblatex}
\DeclareSourcemap{
    \maps[datatype=bibtex]{
        \map{
            \step[fieldset=issn, null]
        }
    }
}
\DefineBibliographyStrings{english}{
    url = \mkbibacro{URL},
}
\DeclareNumChars*{+} 

\usepackage[inline]{enumitem}
\usepackage[T1]{fontenc}
\usepackage{graphicx}

\usepackage{lmodern}

\usepackage{amssymb}
\usepackage{mathrsfs}
\usepackage{mathtools}
\usepackage{microtype}

\usepackage{siunitx} 

\usepackage{subcaption}
\usepackage[colorlinks=true,urlcolor=NavyBlue,citecolor=ForestGreen,linkcolor=BrickRed,bookmarksdepth=3]{hyperref}
\usepackage[nameinlink,capitalize,noabbrev]{cleveref}
\crefname{enumi}{part}{parts} 

\usepackage{tcolorbox}
\newtcolorbox{todo}[1]{colback=yellow!75!white,colframe=red!75!black,title={#1}}

\usepackage{pgfplots} 
\usepackage{tikz} 
\usetikzlibrary{angles,backgrounds,external, hobby,arrows.meta,calc,decorations,decorations.markings}
\pgfplotsset{compat=1.18}

\tikzsetexternalprefix{tikz/}
\allowdisplaybreaks

\numberwithin{equation}{section}

\theoremstyle{plain}
\newtheorem{theorem}{Theorem}[section]
\newtheorem{proposition}[theorem]{Proposition}
\newtheorem{corollary}[theorem]{Corollary}
\newtheorem{lemma}[theorem]{Lemma}

\theoremstyle{remark}
\newtheorem{remark}[theorem]{Remark}

\theoremstyle{definition}
\newtheorem{definition}[theorem]{Definition}

\DeclareMathOperator{\dist}{dist}

\DeclareMathOperator{\id}{id}
\DeclareMathOperator{\im}{Im}

\DeclareMathOperator{\lspan}{span}

\DeclareMathOperator{\range}{rng}
\DeclareMathOperator{\re}{Re}
\DeclareMathOperator{\sech}{sech}

\MHInternalSyntaxOn
\renewcommand\MT_delim_default_inner_wrappers:n [1]{
    \@namedef{MT_delim_\MH_cs_to_str:N #1 _star_wrapper:nnn}##1##2##3{
        \mathopen{}\mathclose\bgroup ##1 ##2 \aftergroup\egroup ##3
    }
    \@namedef{MT_delim_\MH_cs_to_str:N #1 _nostarscaled_wrapper:nnn}##1##2##3{
        \mathopen{##1}##2\mathclose{##3}
    }
    \@namedef{MT_delim_\MH_cs_to_str:N #1 _nostarnonscaled_wrapper:nnn}##1##2##3{
        \ifx.##1\else\mathopen##1\fi##2\ifx.##3\else\mathclose##3\fi
    }
}
\MHInternalSyntaxOff

\DeclarePairedDelimiter{\abs}{\lvert}{\rvert}
\DeclarePairedDelimiter{\brac}{\lbrace}{\rbrace}
\DeclarePairedDelimiter{\brak}{\lbrack}{\rbrack}
\DeclarePairedDelimiter{\norm}{\lVert}{\rVert}
\DeclarePairedDelimiter{\parn}{\lparen}{\rparen}
\DeclarePairedDelimiter{\rest}{.}{\rvert}
\DeclarePairedDelimiter{\lbrac}{\lbrace}{.}

\newcommand{\mb}{\mathbb} 
\newcommand{\mc}{\mathcal} 
\newcommand{\mf}{\mathfrak} 
\newcommand{\mr}{\mathrm} 
\newcommand{\msc}{\mathscr} 

\newcommand{\mo}{\mathring}
\newcommand{\ol}{\overline}
\newcommand{\wh}{\widehat}

\newcommand{\R}{ \mb{R} } 
\newcommand{\Z}{ \mb{Z} } 
\newcommand{\N}{ \mb{N} } 
\newcommand{\C}{ \mb{C} } 
\newcommand{\T}{ \mb{T} } 

\newcommand{\placeholder}{\makebox[\widthof{$x$}][c]{$\cdot$}} 

\newcommand{\dee}{\mathop{}\!d}

\newcommand{\ceq}{\coloneqq}
\newcommand{\eqc}{\eqqcolon}

\newcommand{\loc}{{\textnormal{loc}} }
\newcommand{\even}{\mr{e}}

\newcommand{\imim}{{\mr{ii}} }
\newcommand{\reim}{{\mr{ri}} }

\newcommand{\F}{\msc{F}} 
\newcommand{\G}{\msc{G}} 

\newcommand{\W}{\msc{W}}
\newcommand{\X}{\msc{X}}
\newcommand{\Y}{\msc{Y}}

\newcommand{\Open}{\msc{O}}

\newcommand{\Curve}{\msc{C}}
\newcommand{\Kurve}{\msc{K}}

\newcommand{\fluidD}{\msc{D}}
\newcommand{\fluidS}{\msc{S}}
\newcommand{\fluidB}{\msc{B}}
\newcommand{\fluidV}{\msc{V}}

\newcommand{\confD}{\Omega}
\newcommand{\confS}{\Gamma}

\newcommand{\confSs}{\Gamma_{\mathrm{s}}}
\newcommand{\confSv}{\Gamma_{\mathrm{v}}}

\newcommand{\wsigma}{\breve{\sigma}}
\newcommand{\wzeta}{\breve{\zeta}}

\newcommand{\proj}{P}

\newcommand{\surf}[1]{#1_\mr{s}}
\newcommand{\vort}[1]{#1_\mr{v}}

\begin{document}

\title[Wave-borne vortices]{Vortex-carrying solitary gravity waves of large amplitude}

\date{\today}

\author[R. M. Chen]{Robin Ming Chen}
\address{Department of Mathematics, University of Pittsburgh, Pittsburgh, PA 15260}
\email{mingchen@pitt.edu}

\author[K. Varholm]{Kristoffer Varholm}
\address{Department of Mathematical Sciences, Norwegian University of Science and Technology, 7491, Trondheim, Norway}
\email{kristoffer.varholm@ntnu.no}

\author[S. Walsh]{Samuel Walsh}
\address{Department of Mathematics, University of Missouri, Columbia, MO 65211}
\email{walshsa@missouri.edu}

\author[M. H. Wheeler]{Miles H. Wheeler}
\address{Department of Mathematical Sciences, University of Bath, Bath BA2 7AY, United Kingdom}
\email{mw2319@bath.ac.uk}

\begin{abstract}
    In this paper, we study two-dimensional traveling waves in finite-depth water that are acted upon solely by gravity. We prove that, for any supercritical Froude number (non-dimensionalized wave speed), there exists a continuous one-parameter family $\Curve$ of solitary waves in equilibrium with a submerged point vortex. This family bifurcates from an irrotational uniform flow, and, at least for large Froude numbers, extends up to the development of a surface singularity. These are the first rigorously constructed gravity wave-borne point vortices without surface tension, and notably our formulation allows the free surface to be overhanging. We also provide a numerical bifurcation study of traveling periodic gravity waves with submerged point vortices, which strongly suggests that some of these waves indeed overturn. Finally, we prove that at generic solutions on $\Curve$ --- including those that are large amplitude or even overhanging --- the point vortex can be desingularized to obtain solitary waves with a submerged hollow vortex. Physically, these can be thought of as traveling waves carrying spinning bubbles of air.
\end{abstract}

\maketitle

\setcounter{tocdepth}{1}
\tableofcontents

\section{Introduction}

Vortices are regions of highly concentrated vorticity within a fluid. They can take many forms, with famous examples including shed vortices trailing in the wake of a submerged body and vortex streets created by Rayleigh--Taylor instabilities. This phenomenon is especially pronounced in two-dimensional incompressible flow, as the transportation of vorticity allows vortices to sustain their coherence over long time scales. Understanding the dynamics and qualitative properties of these structures is among the most classical problems in fluid mechanics.

Our interest here is in wave-borne vortices: traveling water waves carrying a vortex in their bulk. Rotational steady water waves have been a major subject of research for the past 20 years; see, for example, the survey~\cite{haziot2022traveling} and the references therein. The overwhelming majority of the literature treats waves with vorticities that extend throughout the entire flow. Localized vorticity, on the other hand, requires quite different techniques to analyze, and it is only in the last decade or so that rigorous existence results for wave-borne vortices have been obtained. Quite recently, authors have constructed water waves with submerged point vortices~\cite{shatah2013travelling,crowdy2014hollow,varholm2016solitary,le2019existence,cordoba2021existence,crowdy2023exact,keeler2023exact}, vortex patches~\cite{shatah2013travelling,chen2019existence,cordoba2021existence}, and ``vortex spikes''~\cite{ehrnstrom2023smooth}. Some of these solutions have been shown to be (conditionally) orbitally stable~\cite{varholm2020stability} or unstable~\cite{le2019existence}. This body of work can be roughly divided in two: one set of results concerns waves for which the influence of capillarity is much stronger than that of gravity, and a second set assumes that both surface tension and gravity are entirely absent. Neither regime is characteristic of most water waves found in nature, however, as gravitational effects predominate at typical length and velocity scales.

\begin{figure}[htb]
    \centering
    \tikzsetnextfilename{cyclopean_wave}
    \input{figures/cyclopean_wave.tikz}
    \caption{Our framework allows for overhanging waves with submerged point vortices or hollow vortices. Illustrated here is a hollow vortex carried by an overturned solitary wave, which we choose to call a \emph{cyclopean wave}. Its ``eye'' is a bubble of air at constant pressure, surrounded by a vortex sheet. Both water--air interfaces of the wave are free.}
    \label{cyclopean wave figure}
\end{figure}
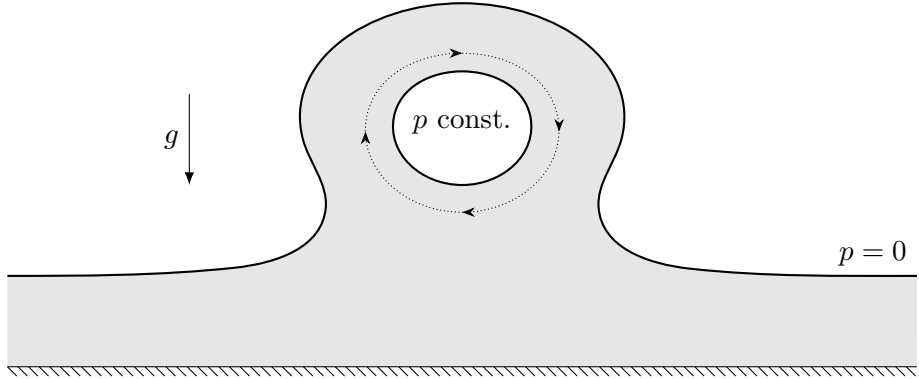

In this paper, we construct both wave-borne point vortices and wave-borne hollow vortices. The main novelties are that we allow for $O(1)$ gravity, do not require surface tension, and prove the existence of large-amplitude waves for which the air--sea interface is far from flat. In fact, these are the first existence results for large-amplitude solitary gravity waves with nontrivial compactly supported vorticity of any kind. Our formulation also allows for an overhanging free surface profile, see \cref{cyclopean wave figure}, and indeed we provide numerical evidence that some of the solutions do have this striking shape.

\subsection{The wave-borne point vortex problem}
\label{intro point vortex problem section}

First consider a two-dimensional gravity water wave carrying a point vortex. We assume the wave is \emph{traveling} or \emph{steady}, meaning it is translating with a fixed speed and without changing shape. Let $(x,y)$ be Cartesian coordinates in a co-moving frame so that the wave propagates to the left along the $x$-axis and gravity acts in the negative $y$-direction. As usual, we identify points $(x,y)$ with $z \ceq x + iy \in \C$. As depicted in \cref{point vortex domains figure}, the water occupies a domain $\fluidD \subset \R^2$ that is bounded below by an impermeable bed at $\fluidB \ceq \brac{ y = 0 }$ and above by a $C^{k+3+\alpha}$ curve $\fluidS$, where $k \geq 0$ and $\alpha \in (0,1)$ are fixed but arbitrary. Note that we expressly do \emph{not} assume that $\fluidS$ is the graph of a (single-valued) function in the horizontal variable.

We are primarily interested in \emph{solitary waves}, meaning that the free surface approaches a horizontal line in the upstream and downstream limits $x \to \pm\infty$. Let $c$ denote the speed of the wave, $d$ the depth of the undisturbed fluid domain, and $g$ the constant gravitational acceleration. Here $c$ is measured in a reference frame where the fluid is at rest at infinity. The relative strength of the gravitational and inertial forces is described by the \emph{Froude number}, which is the non-dimensional quantity $F^2 \ceq c^2/(gd)$. Through a standard rescaling of length and velocity, we can without loss of generality take $c^2 = 1$ and $d = 1$.

Suppose that the velocity field is incompressible and irrotational except for a single point vortex of strength $\gamma$, carried by the wave. The vortex is stationary in the moving frame; we take its location to be $z = ib$ for some $b > 0$. The system is then governed by the free boundary incompressible Euler equations with the so-called Helmholtz--Kirchhoff model for the vortex motion. Concretely, if $(\mf{u},\mf{v})$ is the relative velocity field, then the irrotational incompressible Euler equations become the requirement that
\begin{subequations}
    \label{intro Euler-Kirchhoff-Helmholtz}

    \begin{alignat}{-1}
        \label{intro steady Euler}
         & \mf{u} - i \mf{v} \text{ is holomorphic}                                                       & \qquad & \text{in }\fluidD \setminus \brac{i b}, \\
        \intertext{along with the \emph{kinematic boundary condition}}
        \label{intro kinematic Euler}
         & \mf{u}+i\mf{v} \text{ is purely tangential}                                                    &        & \text{on } \fluidS \cup \fluidB,        \\
        \intertext{and the \emph{dynamic} or \emph{Bernoulli boundary condition}}
        \label{intro dynamic Euler}
         & \frac{1}{2} \parn[\big]{ \mf{u}^2 + \mf{v}^2 } + \frac{1}{F^2} y = \frac{1}{2} + \frac{1}{F^2} &        & \text{on }\fluidS.                      \\
        \intertext{Note that thanks to the non-dimensionalization, the potential energy density is represented by the term $y/F^2$ above. Finally, the Helmholtz--Kirchhoff model is the requirement that}
        \label{intro point vortex advection}
         & \mf{u} - i \mf{v} = \frac{\gamma}{2\pi i } \frac{1}{z-i b} + O\parn*{\abs*{z-i b }}            &        & \text{as } z \to i b,
    \end{alignat}
\end{subequations}
where the absence of a constant term on the right hand side indicates that the point vortex is in equilibrium with the wave. As evidenced by~\eqref{intro point vortex advection}, the vortex strength $\gamma$ can be interpreted as the circulation around the vortex.

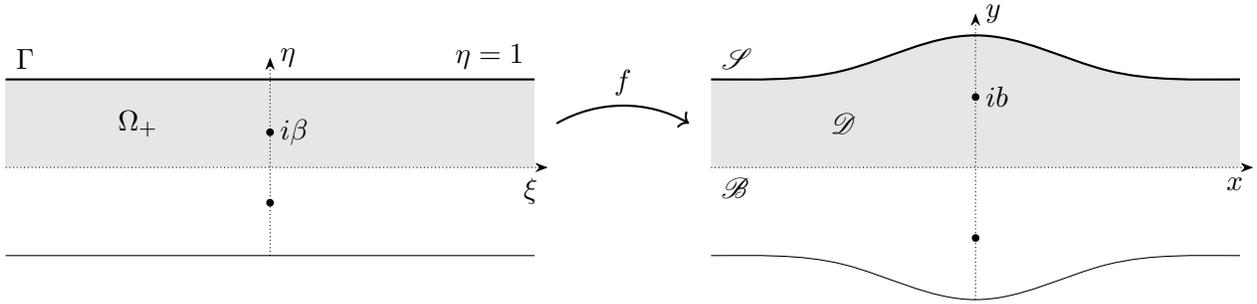
\begin{figure}[htb]
    \centering
    \tikzsetnextfilename{point_vortex_wave}
    \input{figures/point_vortex_wave.tikz}
    \caption{Right: the physical form of the wave in the $z$-plane. Left: the corresponding conformal domain in the $\zeta$-plane. The shaded subset $\confD_+$ is mapped by $f$ to the fluid domain $\fluidD$, while the upper boundary $\confS$ is mapped to the fluid interface~$\fluidS$.}
    \label{point vortex domains figure}
\end{figure}

A major challenge in studying water waves is that the fluid domain is itself one of the unknowns. In order to do functional analysis, we must therefore reformulate the problem in some canonical domain, potentially at the cost of making the governing equations more complicated. In the present paper, we do this by viewing $\fluidD$ as the image of a fixed infinite strip under an unknown conformal mapping. Denote by $\zeta = \xi + i \eta$ the variables in the conformal domain, and define
\begin{equation}
    \label{conformal domain definition}
    \confD \ceq \brac*{ \zeta \in \C : -1 < \eta < 1}, \quad \text{and} \quad \confS \ceq \brac{ \eta = 1 }.
\end{equation}
There exists a conformal mapping $f \in C^{k+3+\alpha}(\ol{\confD})$ such that $\fluidD = f(\confD_+)$, where $\confD_+ \ceq \confD \cap \brac{ \eta > 0 }$ is the upper half of $\confD$,
\begin{equation}
    \label{asymptotics f}
    \partial_\zeta f(\zeta) \to 1 \quad \text{ as } \abs \zeta \to \infty,
\end{equation}
and such that
\begin{equation}
    \label{symmetry f}
    f(-\zeta) = -f(\zeta), \quad f(\ol{\zeta}) = \ol{f(\zeta)}
\end{equation}
for all $\zeta \in \confD$. These properties enforce the symmetry of the domain, and in particular show that $\im{f}$ is odd in $\eta$. Thus $f$ is real-on-real, and so $\fluidB = f(\R)$.

The point vortex at $z = i b$ in the physical domain $\fluidD$ will then be the image of a point $i\beta \in \confD_+$. To maintain the reflection symmetry across the real axis, we may imagine introducing a phantom vortex at $-i \beta$ having the opposite strength. For planar vortices, without the boundary, this would result in a co-translating vortex pair. The physical domain and conformal domain are illustrated in~\cref{point vortex domains figure}. This symmetry allows us to relax the sign condition on $b$, and hence on $\beta$, taking $\beta \in (-1,1)$, which is useful as we are bifurcating from $\beta=0$. Of course it may now be that $-ib = f(-i\beta)$ is the point vortex in the fluid domain $\fluidD$, rather than $ib = f(i\beta)$. The system~\eqref{intro Euler-Kirchhoff-Helmholtz} is modified in the obvious way.

As is well known, there is an explicit relative complex potential
\begin{equation}
    \label{definition complex potential}
    W(\zeta;\gamma,\beta) \ceq \frac{\gamma}{2\pi i} \log\parn*{ \frac{\sinh(\frac{\pi}{2} (\zeta - i \beta))}{\sinh(\frac{\pi}{2} (\zeta+i\beta))} } + \zeta.
\end{equation}
such that any solution of~\eqref{intro Euler-Kirchhoff-Helmholtz} must have complex velocity $\mf{u}-i\mf{v} = (\partial_\zeta W/\partial_\zeta f) \circ f^{-1}$. One can readily confirm that $W$ has constant imaginary part on $\partial \confD$. Moreover, $W_\zeta$ is meromorphic on $\confD$, satisfies $\partial_\zeta W \to 1$ as $\xi \to \pm\infty$, and has simple poles at $\zeta = \pm i \beta$ with residues $\pm \gamma$, representing the contribution of the two counter rotating point vortices. For any conformal mapping $f$, the velocity field obtained in this way satisfies the irrotational, incompressible Euler equation~\eqref{intro steady Euler} and kinematic condition~\eqref{intro kinematic Euler}. It remains to ensure that the dynamic boundary condition~\eqref{intro dynamic Euler} and the Helmholtz--Kirchhoff condition~\eqref{intro point vortex advection} hold. In~\cref{point vortex abstract operator section} we show that the latter has the elegant expression
\begin{equation}
    \label{intro f vortex advection equation}
    \frac{\partial_\zeta^2 f (i \beta)}{\partial_\zeta f (i\beta)} = \pi i \kappa,
\end{equation}
provided that $\kappa \in \R$ is related to $\gamma$ and $\beta$ according to
\begin{equation}
    \label{definition gamma}
    \gamma = \gamma(\kappa,\beta) \ceq - \frac{4 \sin(\pi\beta)}{\cos(\pi\beta) - \kappa \sin(\pi\beta)}.
\end{equation}
Thus, the unknowns $(f,\kappa,\beta)$ can be used to uniquely describe a wave-borne point vortex. For the solution to be physical, we require that $f$ is conformal on $\confD$ with an injective $C^1$ extension to $\ol{\confD}$.

\subsection{The wave-borne hollow vortex problem}
\label{hollow vortex intro section}

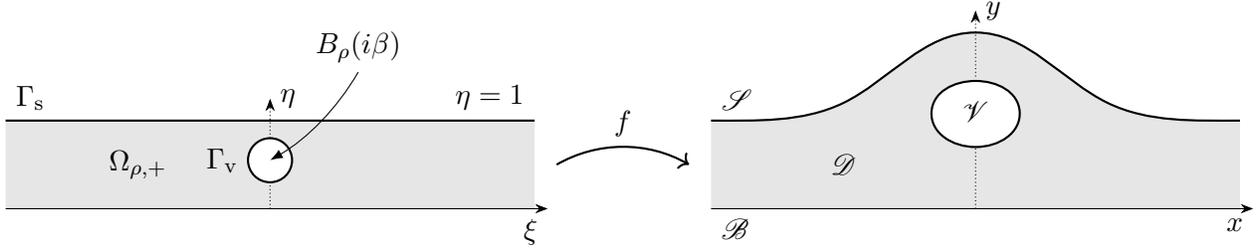
\begin{figure}[htb]
    \centering
    \tikzsetnextfilename{hollow_vortex_wave}
    \input{figures/hollow_vortex_wave.tikz}
    \caption{Right: the physical form of the wave-borne hollow vortex in the $z$-plane. Left: the corresponding conformal domain $\confD_{\rho,+}$ in the $\zeta$-plane. The conformal map $f$ takes the upper boundary $\confSs$ to the top interface $\fluidS$, and $\confSv$ to the boundary of the vortex core $\fluidV$.}
    \label{hollow vortex configuration figure}
\end{figure}

We also consider another classical model of localized vorticity, with a voluminous applied literature: hollow vortices. Roughly speaking, this corresponds to the situation where the vorticity is a measure supported on a collection of Jordan curves. The vortex cores bounded by these curves are taken to be regions of constant pressure; one can imagine them as bubbles of air suspended in the fluid. Recently,~\textcite{chen2023desingularization} systematically desingularized translating, rotating, or stationary planar point vortex configurations into hollow vortices, which could then be continued via global bifurcation theory until the onset of a singularity. In the present paper, we adapt some of those ideas to help construct solitary gravity water waves with a submerged hollow vortex.

To formulate the problem, let us again suppose that we have a solitary wave with fluid domain $\fluidD$. We assume that there is a single hollow vortex being carried by the wave, the vortex core being an open set denoted by $\fluidV$. As it is submerged, the core $\fluidV$ must lie completely below $\fluidS$, above $\fluidB$, and $\fluidV \subset \C \setminus \fluidD$. Thus $\fluidD$ is doubly connected while $\fluidD \cup \ol{\fluidV}$ is simply connected.

Let $(\mf{u}, \mf{v})$ again be the relative velocity in the physical variables. Incompressibility and irrotationality still take the form
\begin{subequations}
    \label{intro hollow vortex problem}
    \begin{alignat}{-1}
        \label{intro steady hollow}
         & \mf{u} - i \mf{v} \text{ is holomorphic}                                                 & \qquad & \text{in } \fluidD,                                     \\
        \intertext{while the kinematic condition must also hold on the vortex boundary:}
        \label{intro kinematic hollow}
         & \mf{u}+i\mf{v} \text{ is purely tangential}                                              &        & \text{on $\fluidS$, $\fluidB$, and $\partial \fluidV$}. \\
        \intertext{Next, the dynamic condition is imposed on both free surfaces. Because the air above $\fluidS$ and inside the vortex core $\fluidV$ are both taken to be at constant pressure (not necessarily equal), we must have that}
        \label{intro dynamic hollow surface}
         & \frac{1}{2}\parn*{ \mf{u}^2 + \mf{v}^2 } + \frac{1}{F^2} y = \frac{1}{2} + \frac{1}{F^2} &        & \text{on } \fluidS                                      \\
        \label{intro dynamic hollow vortex}
         & \frac{1}{2} \parn*{ \mf{u}^2 + \mf{v}^2 } + \frac{1}{F^2} y = q                          &        & \text{on } \partial \fluidV,
    \end{alignat}
    where $q \in \R$ is a Bernoulli constant. Finally, since $\fluidV$ is surrounded by a vortex sheet, we require that
    \begin{equation}
        \label{intro hollow circulation}
        \int_{\partial \fluidV} \parn*{ \mf{u}- i \mf{v} } \dee z = \gamma,
    \end{equation}
    with $\gamma \in \R$ being the vortex strength.
\end{subequations}

As before, we fix the domain through the use of conformal mapping. We will construct the hollow vortices as desingularized point vortices, and so a natural choice is to take
\[
    \confD_\rho \ceq \confD \setminus \ol{B_{\abs\rho}(i\beta) \cup B_{\abs\rho}(-i\beta)},
\]
for $0 < \abs\rho \ll 1$; see~\cref{hollow vortex configuration figure}. Here, the idea is that the vortex boundary $\partial \fluidV$ should be an approximate streamline for the point vortex velocity field, which are asymptotically circular. As before, we have doubled the domain: the line $\brac{\eta = 0}$ corresponds to the bed and $\confSs \ceq \brac{\eta = 1}$ is the pre-image of the free surface. We call $\rho$ the conformal radius of the hollow vortex; in the desingularization procedure, it will serve as a bifurcation parameter. The boundary of the upper vortex will be given by $\partial\fluidV = f(\confSv)$, where $\confSv = \partial B_\rho(i\beta)$.

As in the point vortex case, there is a relative complex potential $W$ depending only on the parameters $\rho$, $\gamma$, and $\beta$, though we lack a simple explicit formula; see~\cref{complex potential lemma}. It follows that the wave-borne hollow vortex configuration can be described entirely by the unknowns $(f, q, \gamma, \beta, \rho)$. However, from the asymptotics for the point vortex velocity field~\eqref{intro point vortex advection}, we expect that $q$ diverges like $1/\rho^2$ as $\rho \searrow 0$. In the actual analysis, we will therefore work with a \emph{normalized Bernoulli constant} that we call $Q$, and turns out to satisfy $Q = O(\rho)$ and $\rho Q = \rho^2 q + O(1)$ as $\rho \searrow 0$; see~\eqref{hollow vortex leading order parameters} and \cref{hollow vortex nonlocal formulation section}.

\subsection{Informal statement of results}

Our first contribution is the following theorem, which establishes the existence of large families of solitary gravity water waves carrying point vortices in their bulk.

\begin{theorem}[Wave-borne point vortices]
    \label{intro point vortex theorem}
    For any supercritical Froude number $F^2 > 1$, there exists a global curve $\Curve$ of solitary gravity water waves with a submerged point vortex. It admits a global $C^0$ parameterization
    \[
        \Curve = \brac*{ (f(s), \kappa(s), \beta(s)) : s \in \R } \ \subset\ C^{k+3+\alpha}(\ol{\confD}) \times \R^2
    \]
    and satisfies the following.
    \begin{enumerate}[label=\textup{(\alph*)}]
        \item \textup{(Local uniqueness)} The curve $\Curve$ bifurcates from the trivial uniform flow
              \[
                  (f(0), \kappa(0), \beta(0)) = (\id, 0, 0),
              \]
              which by~\eqref{definition gamma} is irrotational as $\gamma(0) = \gamma(\kappa(0),\beta(0)) = 0$. In a neighborhood of this flow, $\Curve$ comprises all solutions to both the irrotational water wave problem and the wave-borne point vortex problem~\eqref{intro Euler-Kirchhoff-Helmholtz} with the given Froude number.
        \item \textup{(Limiting behavior)} Either $\Curve$ is a \emph{closed loop} or, as we follow it to either extreme, \emph{blowup} occurs in that
              \begin{equation}
                  \label{intro theorem blowup alternative}
                  \limsup_{s \to \pm\infty} \parn*{ \sup_{\confS} \parn[\bigg]{ \abs{\partial_\zeta f(s) } + \frac{1}{\abs{\partial_\zeta f(s)}} } + \abs{\gamma(s)} } = \infty.
              \end{equation}
              If $\Curve$ is a closed loop, then it must pass through a nontrivial irrotational solitary gravity wave. If $F^2 > 2$, then $\Curve$ cannot be a closed loop and thus the blowup alternative~\eqref{intro theorem blowup alternative} occurs. In either case, $\gamma(s) \ne 0$ for all but a discrete set of parameter values $s \in \R$.
        \item \textup{(Symmetry and monotonicity)} For each solution on $\Curve$, the streamlines and horizontal velocity $\mf{u}(s)$ are even across the imaginary axis, while the vertical velocity $\mf{v}(s)$ is odd. Moreover,
              \begin{equation}
                  \label{intro monotonicity}
                  \mf{v}(s) < 0 \qquad \text{in } \parn*{ \fluidD(s) \cup \fluidS(s) } \cap \brac{ x > 0},
              \end{equation}
              where $\fluidD(s) \ceq f(s)(\confD_+)$ and $\fluidS(s) \ceq f(s)(\confS)$.
    \end{enumerate}
\end{theorem}

This result is proved through a global bifurcation-theoretic argument that is outlined in the next subsection and carried out in~\cref{local bifurcation section,global point vortex section}. The limiting behavior along $\Curve$ described in~\eqref{intro theorem blowup alternative} can be understood as follows: if the curve is not a closed loop, then as we follow it either the conformal description of the domain degenerates --- corresponding to the blowup of $\abs{\partial_\zeta f(s)}$ or $1/\abs{\partial_\zeta f(s)}$ --- or the vortex strength is unbounded. Conformal degeneracy is expected to be accompanied by a loss of regularity for the free boundary, such as the development of a corner. Of course, both of these limiting scenarios can occur simultaneously. Note also that, a priori, it is possible to have $\abs{\beta(s)} \to 1$, meaning that the vortex approaches the surface in the conformal domain. In fact, this situation is not observed numerically (for finite $F^2$), but we show that it would nevertheless coincide with the blowup in~\eqref{intro theorem blowup alternative}. The leading-order form of waves close to the point of bifurcation are given in~\eqref{small amplitude asymptotics}.

It is intriguing to ask what shape the free surface limits to as one traverses $\Curve$. In~\cref{numerics section}, we address this question numerically. For simplicity, we consider periodic waves so as to have a compact computational domain. The solutions we obtain have quite localized free boundary deflections, and thus it is reasonable to expect that they are qualitatively the same as solitary waves when the period is taken sufficiently large. For supercritical $F$ close to $1$, the limiting behavior is similar to that for irrotational gravity waves: the solutions appear to approach a wave of greatest height with a corner at its crest. On the other hand, with moderate or large Froude number, we find that the waves overturn with the surface becoming nearly circular. This limiting form agrees with the exact solutions for the gravity-less case~\cite{crowdy2023exact} and earlier computations of weak gravity and point vortex forced waves~\cite{doak2017solitary}.

The existence of overhanging gravity waves was first predicted by a series of numerical papers in the 1980s, where they were found to occur in the presence of constant vorticity~\cite{dasilva1988steep} and in two-fluid systems with a vortex sheet~\cite{pullin1988finite,turner1988broadening}. Rigorously constructing a global branch of solutions connecting trivial shear flows to overhanging waves remains among the largest open problems in the field. Up to this point, authors have succeeded in proving the existence of families that potentially overturn for the case of periodic waves with constant vorticity~\cite{constantin2016global}, general vorticity~\cite{wahlen2022critical}, or linear density stratification~\cite{haziot2021stratified}; for solitary waves with constant vorticity~\cite{haziot2023large}; and for fronts in two-layer systems~\cite{chen2023global}. These are on par with our result in the sense that numerical evidence indicates that overhanging waves are reached, but no analytical argument has yet been given. Very recently, global branches of periodic waves with constant vorticity, definitively including overhanging waves, have been constructed for large Froude numbers~\cite{carvalho2023gravity}, or equivalently for small gravity. These solutions are obtained by perturbing an explicit family of solutions to the zero-gravity problem with $F=\infty$~\cite{hur2020exact,hur2022overhanging}.

The closest antecedents to~\cref{intro point vortex theorem} are the constructions of solitary capillary-gravity waves with a submerged point vortex by~\textcite{shatah2013travelling}, who considered the infinite-depth case, and by~\textcite{varholm2016solitary}, who studied the analogous class of waves in finite-depth water. Both of these papers rely crucially on surface tension, however, and the solitary waves they obtain are small-amplitude and have vortex strength and wave speed close to $0$. Essentially, they use an implicit function theorem argument where the starting trivial solution is a stationary wave with a zero-strength point vortex at an $O(1)$ distance from the boundary, and a mirror vortex in the air region. With surface tension, one finds that the linearized problem at such a configuration is an isomorphism. By contrast, for gravity solitary waves, the linearized problem is invertible if and only if the waves are fast moving in the sense that $F^2 > 1$. Based on the situation in the planar case, it seems difficult to arrange for this to occur if the vortex pair is far separated and weak. For that reason, in proving~\cref{intro point vortex theorem}, we arrange that the vortex and mirror vortex are superposed, but still have zero strength, at the point of bifurcation. While the same methodology could be generalized to treat both small- and large-amplitude capillary-gravity waves, the monotonicity property~\eqref{intro monotonicity} is not expected to be preserved along the curve, which would lead to a larger set of alternatives than what we find in~\eqref{intro theorem blowup alternative} for gravity waves.

While there are no known explicit solutions to wave-borne vortex problems such as~\eqref{intro Euler-Kirchhoff-Helmholtz} when the effect of gravity is included, things are dramatically different in the absence of gravity. Formally, this amounts to setting $F=\infty$, eliminating the inhomogeneous term $y/F^2$ in~\eqref{intro dynamic Euler} so that the magnitude of the fluid velocity is constant along the surface $\fluidS$. \citeauthor{crowdy2010steady}~\cite{crowdy2010steady} discovered a family of periodic waves with one point vortex per period and a nonzero constant background vorticity, whose fluid velocity vanishes identically along the surface. Later, \citeauthor{crowdy2014hollow}~\cite{crowdy2014hollow} found a second family of periodic waves, this time with nonzero velocity along the surface and no constant background vorticity. More recently, \citeauthor{crowdy2023exact}~\cite{crowdy2023exact} put the above solutions into a unified framework, which also encompasses the solutions~\cite{hur2020exact} with constant vorticity and no point vortices. The central idea is that, in the absence of any gravity or surface tension, the complex velocity field $\mf{u} - i \mf{v}$ must be given explicitly in terms of the so-called Schwarz function associated to the fluid domain $\fluidD$. If we assume that $\fluidD$ is given in terms of an explicit conformal mapping, then this Schwarz function is similarly explicit. To illustrate the power of the method,~\cite{crowdy2023exact} introduces several completely new solution families. In particular, these includes solutions to the $F=\infty$ limit of our solitary wave problem~\eqref{intro Euler-Kirchhoff-Helmholtz}; see~\cite[Section 6]{crowdy2023exact} and \cref{exact zero gravity wave proposition} below. It would be interesting to see if these waves could be continued to large but finite $F$ using the methods of~\cite{carvalho2023gravity,hur2022overhanging}. Subsequent work building on~\cite{crowdy2023exact} includes \textcite{keeler2023exact}, which generalizes~\cite{crowdy2010steady} to the case of two point vortices per period. Unfortunately, the Schwarz function machinery does not obviously extend to the $O(1)$ gravity wave setting that we study here.

Finally, it is important to mention the early works of~\textcite{terkrikorov1958vortex,filippov1960vortex,filippov1961motion}, who studied the related problem of traveling gravity waves with point vortex \emph{forcing}. That is, they look at the somewhat simpler problem where one does not impose the condition~\eqref{intro point vortex advection} that the vortex is at equilibrium. \textcite{gurevich1964vortex,shaw1972note} investigated the gravity-less version of the same system, which admits explicit overhanging solutions. Physically, point vortex forcing may for instance model waves interacting with an immersed body being dragged through the bulk. In the present paper, we use the term \emph{wave-borne vortices} to make clear that the vortices we are considering are instead carried along by the wave. Interestingly, in~\cite{terkrikorov1958vortex}~\citeauthor{terkrikorov1958vortex} looked specifically at the possibility of a bifurcation curve of point vortex forced waves connecting two irrotational waves, which is analogous to the closed loop alternative in~\cref{intro point vortex theorem}.

Our second main result concerns solutions to the wave-borne hollow vortex problem~\eqref{intro hollow vortex problem}. It states that for a generic subset of the waves on $\Curve$, it is possible to desingularize the point vortex into a hollow vortex. More precisely, we prove the following.

\begin{theorem}[Wave-borne hollow vortices]
    \label{intro hollow vortex theorem}
    Assume $F^2 > 1$ and let $\Curve$ be the curve of wave-borne point vortices furnished by~\cref{intro point vortex theorem}. For all parameter values $s_0 \in \R$ outside of a discrete set, there is a real-analytic curve $\Kurve = \Kurve_{s_0}$ of solutions to the solitary gravity wave-borne hollow vortex problem~\eqref{intro hollow vortex problem} admitting the parameterization
    \[
        \Kurve_{s_0} \ceq \brac[\big]{ (f^\rho, \gamma^\rho, Q^\rho, \rho ) : \abs{\rho} < \rho_1 } ,
    \]
    where $\rho_1 =\rho_1(s_0)>0$.

    The curve $\Kurve_{s_0}$ bifurcates from the point vortex solution on $\Curve$ at parameter value $s = s_0$ in that
    \[
        \parn[\big]{f^0,\gamma^0} = \parn[\big]{ f(s_0), \gamma\parn[\big]{\kappa^0, \beta^0}},
    \]
    where $\beta^0 \ceq \beta(s_0)$, $\kappa^0 \ceq \kappa(s_0)$. Moreover, the conformal mapping has the leading-order form
    \begin{subequations}
        \label{hollow vortex leading order}
        \begin{equation}
            \label{hollow vortex leading order f}
            \begin{aligned}
                f^\rho                                           & = f^0+ O(\rho^2) \quad \text{in } C^{k+3+\alpha}(\ol{\confD_\rho}; \C)                                                                                                           \\
                \frac{1}{2\pi i} \int_{\confSv} f^\rho \dee\zeta & = \rho^3 \parn*{ \pi^2 f_\zeta^0(i\beta^0) \parn[\bigg]{ \frac{(\kappa^0)^2}{8} - \frac{1-3 \csc^2{(\pi \beta^0)}}{12} } +\frac{f_{\zeta\zeta\zeta}^0(i\beta^0)}{4} }+O(\rho^4),
            \end{aligned}
        \end{equation}
        while the circulation and corresponding Bernoulli constant $q^\rho$ satisfy
        \begin{equation}
            \label{hollow vortex leading order parameters}
            \gamma^\rho = \gamma^0+O(\rho^2) \qquad
            \rho^2 q^\rho =
            \frac{(\gamma^0)^2}{4f_\zeta^0(i\beta^0)^2} \parn[\bigg]{ \frac 1{\pi^2} -\rho^2 \frac{(\kappa^0)^2}{4} +O(\rho^4)}.
        \end{equation}
    \end{subequations}
\end{theorem}

We emphasize that, because this result can be applied at a generic solution on $\Curve$, the resulting waves can be of large amplitude. Indeed, if $\Curve$ contains waves with an overturned free surface, as suggested by the numerics, then \cref{intro hollow vortex theorem} implies the existence of cyclopean waves, as depicted in \cref{cyclopean wave figure}. Note that the curve $\Kurve_{s_0}$ contains waves with small hollow vortices. While we do not pursue it here, it is possible to combine the argument used to prove~\cref{intro point vortex theorem} with that for planar hollow vortices in~\cite{chen2023desingularization} to extend the curve $\Kurve_{s_0}$, thereby obtaining solitary waves with \emph{large} submerged hollow vortices. As in~\eqref{intro theorem blowup alternative}, the expected limiting behavior along this global curve would be either unboundedness of $\gamma$, $Q$, or the development of a surface singularity or self-intersection on either $\fluidS$ or $\partial\fluidV$.

\cref{intro hollow vortex theorem} represents the first construction of wave-borne hollow vortices. In fact, to the best of our knowledge, it is the first existence result for hollow vortices in the presence of gravity. Translating hollow vortex pairs in the plane were first constructed by~\textcite{pocklington1894configuration}. A modern treatment of the same system based on Schottky--Klein prime functions was given by~\textcite{crowdy2013translating}, and a similar methodology was used by~\textcite{green2015analytical} to study hollow vortex pairs in a channel. However, as mentioned above, this complex function theoretic approach cannot be applied directly when gravity is present. By contrast, the implicit function theorem argument we use, like that in~\cite{chen2023desingularization}, is completely indifferent to even $O(1)$ gravity. Indeed, looking at the dynamic condition on the vortex boundary~\eqref{intro dynamic hollow vortex}, we see that the kinetic energy term $(\mf{u}^2+\mf{v}^2)/2$ is $O(1/\rho^2)$; for the linearized problem, it will totally overpower the potential energy term.

With some effort, a similar desingularization technique can also be used to prove the existence of solitary gravity waves with a submerged vortex patch, that is, a bounded open region of nonzero vorticity. The wave-borne vortex patch problem is simpler than the wave-borne hollow vortex problem, in the sense that the boundary of the patch is a streamline but not at constant pressure. The fully nonlinear dynamic boundary condition~\eqref{intro dynamic hollow vortex} need therefore not be imposed there. On the other hand, the interior of the patch is part of the fluid domain, and not simply air, so we must determine the streamlines there as well. For general vorticity, this task amounts to solving a semilinear elliptic free boundary problem that is coupled to the flow in the exterior of the patch. Waves of this type are the subject of a forthcoming work.

\subsection{Outline of the proof and plan of the article}

In~\cref{local bifurcation section}, we begin the argument leading to~\cref{intro point vortex theorem} by first constructing a curve of small-amplitude solitary gravity waves with a submerged point vortex. Somewhat whimsically, one can imagine this being done by injecting a point vortex into the bulk of the fluid through the bed, and an accompanying counter-rotating mirror vortex below the bed. The pair of vortices will then tend to translate together, and so it remains only to couple their motion to that of the wave. We find that when the Froude number is supercritical, the linearization of~\eqref{intro Euler-Kirchhoff-Helmholtz} at the trivial, irrotational uniform flow is an isomorphism between the appropriate spaces. The implicit function theorem therefore furnishes a local curve of solutions $\Curve_\loc$, which will contain waves whose free surface is nearly flat, and carrying a point vortex that is very close to the bed and has very weak vortex strength. In~\cref{global point vortex section}, we use techniques from analytic global bifurcation theory to continue $\Curve_\loc$ into the large-amplitude regime, obtaining waves with potentially strong vortices that may be $O(1)$ removed from the bed. Specifically, we make use of a version of this general theory developed in~\cite{chen2018existence} to treat problems set on unbounded domains. By establishing a certain monotonicity property, we are able to rule out several undesirable ``loss of compactness'' alternatives for the limiting behavior of the solution curve. Finally, through nonlinear a priori estimates, we are able to show that either $\Curve$ is a closed loop, or the blowup scenario~\eqref{intro theorem blowup alternative} must occur.

In~\cref{numerics section}, we complement this analysis with numerical computations for periodic waves. In fact, the proof of~\cref{intro point vortex theorem} can be adapted to the periodic regime, though we leave that to future work in the interest of brevity. For a sufficiently large period, our computed waves show close agreement with the explicit solitary solutions available in the limit $F=\infty$~\cite{crowdy2023exact}. We also show that, after a suitable rescaling, the surfaces of these explicit solutions converge to a circle centered at the point vortex, sitting on top of a line, as $\beta \to 1$.

\cref{hollow vortex section} is devoted to the proof of~\cref{intro hollow vortex theorem} concerning the existence of solitary wave-borne hollow vortices. As described above, this is done by taking a wave-borne point vortex on $\Curve$, then looking for nearby water waves having a submerged hollow vortex such that $\partial\fluidV$ is approximately a streamline of the starting velocity field. The main challenge in adapting the desingularization machinery from~\cite{chen2023desingularization} to the present setting is naturally to account for the wave-vortex interaction. Through careful asymptotic analysis, we find a formulation of the system as an abstract operator equation $\G(u,\rho) = 0$ for a new unknown $u$. Crucially, we prove that $\G$ is real analytic, and that this equation agrees with the wave-borne hollow vortex problem~\eqref{intro hollow vortex problem} when $\rho > 0$ and recovers the wave-borne point vortex problem~\eqref{intro Euler-Kirchhoff-Helmholtz} at $\rho = 0$. The linearized operator $D_u\G(u^0,0)$ at a given wave-borne point vortex solution $u^0$ can be manipulated into $2 \times 2$ block-diagonal form, with one block corresponding to the linearization of the wave-borne point vortex problem, and the other relating to the linearized planar hollow vortex problem. This simple structure enables us to confirm that generically $D_u\G(u^0,0)$ is an isomorphism, so that the curve $\Kurve$ can be constructed using the implicit function theorem.

Finally, for the convenience of the reader,~\cref{quoted results appendix} collects some background results that are used at various stages throughout the paper.

\subsection{Notation}
Here we lay out some notational conventions for the remainder of the paper. We will often make use of the Wirtinger derivative operator
\[
    \partial_z \ceq \tfrac{1}{2} \parn*{ \partial_x - i \partial_y },
\]
which may also be denoted as $f_z = \partial_z f$ when there is no risk of confusion.
Primes are reserved for complex (total) derivatives of functions with domain $\T$, the unit circle in the complex plane; if $f$ is a function of $\tau = e^{i\theta} \in \T$, then $f' = \partial_\tau f = -ie^{-i\theta} df/d\theta$.

Let $D$ be a subset of $\R^n$ or $\C^n$. For each integer $\ell \geq 0$ and $\alpha \in (0,1)$, we denote by $C^{\ell+\alpha}(D)$ the usual space of real-valued Hölder continuous functions of order $\ell$, exponent $\alpha$, and having domain $D$. When $D$ is unbounded, define $C_0^{\ell+\alpha}(D)$ to be the subspace of $u \in C^{\ell+\alpha}(D)$ all of whose partials $\partial^\mu u$ of order $\abs\mu \le \ell$ vanish uniformly at infinity. In the specific case where $D \subset \C$, we denote by $C_\even^{\ell+\alpha}(D)$ those functions that are even across the imaginary axis and odd across the real axis. Likewise, $C_{0,\even}^{\ell+\alpha}(D) \ceq C_0^{\ell+\alpha}(D) \cap C_\even^{\ell+\alpha}(D)$.

Every $\varphi \in C^{\ell+\alpha}(\T)$ admits a unique power series representation
\[
    \varphi(\tau) = \sum_{m \in \Z} \wh{\varphi}_m \tau^m, \qquad \text{with}\qquad \wh{\varphi}_m \ceq \frac{1}{2\pi} \int_{\T} \varphi(\tau) \tau^{-m} \dee\theta,
\]
and we write $\tau = e^{i\theta}$ when $\T$ is parameterized by arc length. Note that because these are real-valued functions, the coefficients must necessarily obey $\wh{\varphi}_{-m} = \ol{\wh{\varphi}_m}$. For $m \geq 0$, let $\proj_m \colon C^{\ell+\alpha}(\T) \to C^{\ell+\alpha}(\T)$ denote the projection
\[
    (\proj_m\varphi)(\tau) \ceq
    \begin{cases}
        \wh{\varphi}_m \tau^m + \wh{\varphi}_{-m} \tau^{-m} & \text{if } m \neq 0 \\
        \wh{\varphi}_0                                      & \text{if } m = 0,
    \end{cases}
\]
and set $\proj_{\leq m} \ceq \proj_0 + \cdots + \proj_m$ and $\proj_{> m} \ceq 1-\proj_{\leq m}$. We will often work with the space $\mo{C}^{\ell+\alpha}(\T) \ceq \proj_{> 0} C^{\ell+\alpha}(\T)$ of mean $0$ elements of $C^{\ell+\alpha}(\T)$.

\section{Small solitary wave-borne point vortices}
\label{local bifurcation section}

The main purpose of this section is to establish the existence of small-amplitude solitary gravity waves carrying point vortices. We will use an implicit function theorem argument that takes advantage of the fact that the linearized dynamic condition on the free surface $\confS$ is invertible provided the Froude number is supercritical.

\subsection{Abstract operator equation}
\label{point vortex abstract operator section}

We begin by fixing the functional analytic setting, rewriting the wave-borne point vortex system described in~\cref{intro point vortex problem section} as an abstract operator equation to which the implicit function can eventually be applied. Because the full conformal map can be reconstructed from its imaginary part, it will be convenient to recast the problem in terms of the unknown $w$ defined by
\[
    w(\zeta) \ceq \im\parn*{ f(\zeta)- \zeta} = \im{f(\zeta)} - \eta.
\]
Thus, $w \in C^{k+3+\alpha}(\ol{\confD})$ is a real-valued harmonic function that vanishes as $\xi \to \pm\infty$. Naturally, it inherits symmetry properties from $f$ due to~\eqref{symmetry f}. With that in mind, we introduce the closed subspace
\begin{equation}
    \label{definition W space}
    \W \ceq \brac*{ w \in C_{0,\even}^{k+3+\alpha}(\ol{\confD}) : \text{$\Delta w = 0$ in $\confD$} },
\end{equation}
which is the natural class of $w$ for us to consider.

Now, we must rewrite the Euler equations~\eqref{intro Euler-Kirchhoff-Helmholtz} in terms of $w$. Using $\mf{u}-i\mf{v} = (\partial_\zeta W/\partial_\zeta f) \circ f^{-1}$, the Helmholtz--Kirchhoff condition~\eqref{intro point vortex advection} can be written
\begin{equation}
    \label{conformal helmholtz kirchhoff}
    \frac{W_\zeta}{f_\zeta} - \frac{\gamma}{2\pi i} \frac{1}{f - i b} = O(\zeta - i \beta) \qquad \text{as } \zeta \to i \beta.
\end{equation}
From~\eqref{definition complex potential} we compute directly that
\begin{equation}
    \label{solitary W_zeta expansion}
    W_\zeta(\zeta) = \frac{\gamma}{2\pi i} \frac{1}{\zeta-i \beta} + \parn[\bigg]{1 + \frac{\gamma}{4} \cot{(\pi \beta)}} - \frac{\pi \gamma}{24i} \parn[\big]{ 2 + 3 \cot^2{(\pi \beta)} } (\zeta-i\beta) + O\parn[\big]{ (\zeta-i\beta)^2 },
\end{equation}
whence performing a Laurent expansion of the left-hand side of~\eqref{conformal helmholtz kirchhoff} yields
\begin{equation}
    \label{solitary laurent expansion}
    \frac{W_\zeta(\zeta)}{f_\zeta(\zeta)} - \frac{\gamma}{2\pi i} \frac{1}{f(\zeta) - ib} = \frac{1}{f_\zeta(i\beta)} \parn[\bigg]{ 1+ \frac{\gamma}{4} \cot{(\pi \beta)} - \frac{\gamma}{4\pi i} \frac{f_{\zeta\zeta}(i\beta)}{f_\zeta(i\beta)} } + O( \zeta-i\beta )
\end{equation}
as $\zeta \to i\beta$. From this, it is evident that~\eqref{conformal helmholtz kirchhoff} is satisfied if and only if the constant term on the right-hand side of~\eqref{solitary laurent expansion} vanishes. This occurs precisely when the point vortex advection equation~\eqref{intro f vortex advection equation} holds and $\gamma$ is given by~\eqref{definition gamma}. Via the Cauchy--Riemann equations,~\eqref{intro f vortex advection equation} can equivalently be stated as
\begin{subequations}
    \label{w equation}
    \begin{equation}
        \label{vortex equation}
        \rest[\Big]{\frac{w_{\xi\xi}}{1+ w_\eta}}_{\zeta = i\beta} = \pi \kappa
    \end{equation}
    in terms of $w$. Likewise, the formula~\eqref{definition complex potential} for $W$ allows us to rewrite the Bernoulli condition~\eqref{intro dynamic Euler} quite concisely as
    \begin{equation}
        \label{Bernoulli condition}
        \frac{1}{2} \frac{a^2}{w_\xi^2 + (1+w_\eta)^2} + \frac{1}{F^2} w = \frac{1}{2} \qquad \text{on } \confS,
    \end{equation}
\end{subequations}
where $a = a(\xi;\kappa,\beta)$ is given explicitly by
\begin{equation}
    \label{definition a}
    a(\xi;\kappa,\beta) \ceq W_\zeta(\xi+i; \gamma(\kappa,\beta),\beta) = 1-\frac{\gamma(\kappa,\beta)}{2} \frac{\sin(\pi \beta)}{\cosh(\pi\xi)+\cos(\pi\beta)}.
\end{equation}

In view of the above discussion, in the conformal reformulation~\eqref{w equation},
\begin{equation}
    \label{definition u}
    u \ceq (w,\kappa)
\end{equation}
will serve as the unknown with the conformal altitude $\beta$ as the parameter. We therefore introduce the space $\X \ceq \W \times \R$ and open set
\begin{equation}
    \label{definition of O nbhd}
    \Open \ceq \brac[\bigg]{ (w,\kappa,\beta) \in \W \times \R^2 : \quad
        \begin{gathered}
            w_\xi^2+(1+w_\eta)^2 > 0 \quad \text{in }\ol{\confD},\\
            \abs{\beta} < 1, \quad \text{and } \quad \cos(\pi\beta) > \kappa \sin(\pi\beta)
        \end{gathered}
    },
\end{equation}
preventing degeneracy of the formulation, and ensuring that it is equivalent to the physical problem. In particular, the final requirement on $\kappa$ and $\beta$ makes sure that $\gamma(\kappa,\beta)$ in~\eqref{definition gamma} is well-defined, and that we remain in the same connected component of the domain of $\gamma$ as the trivial solution. Moreover, a consequence is that
\begin{equation}
    \label{gamma inequality}
    \gamma(\kappa,\beta) \sin(\pi \beta) \le 0, \quad
    \text{ with equality if and only } \beta = 0,
\end{equation}
and therefore
\begin{equation}
    \label{a inequality}
    a(\placeholder; \kappa,\beta) \ge 1,
\end{equation}
for all $(w,\kappa,\beta) \in \Open$.

In summary, the water wave problem with a submerged vortex can be rewritten as the abstract operator equation
\begin{equation}
    \label{abstract operator equation}
    \F(u,\beta) = 0,
\end{equation}
where $\F = (\F_1,\F_2) : \Open \subset (\X\times \R) \to \Y$ is the real-analytic map defined by
\begin{equation}
    \label{definition F}
    \begin{aligned}
        \F_1(w,\kappa,\beta) & \ceq
        \rest*{\parn*{\frac{1}{2} \frac{a^2}{w_\xi^2 + (1+w_\eta)^2} + \frac{1}{F^2} w - \frac{1}{2}}}_\confS          \\
        \F_2(w,\kappa,\beta) & \ceq \frac{1}{\pi} \rest[\Big]{\frac{w_{\xi\xi}}{1+ w_\eta}}_{\zeta = i\beta} - \kappa,
    \end{aligned}
\end{equation}
and having codomain
\[
    \Y = \Y_1 \times \Y_2 \ceq C_{0,\even}^{k+2+\alpha}(\confS) \times \R.
\]
Recall that the subscript ``e'' indicates that the elements of $\Y_1$ are even in $\xi$. Notice also from~\eqref{definition gamma} that $\gamma(-\kappa,-\beta) = -\gamma(\kappa,\beta)$, and hence $a(\placeholder; -\kappa,-\beta) = a(\placeholder; \kappa,\beta)$. The nonlinear operator therefore exhibits the symmetries
\begin{equation}
    \label{symmetry F}
    \begin{aligned}
        \F_1(w,\kappa,\beta) & = \F_1(w,-\kappa,-\beta)   \\
        \F_2(w,\kappa,\beta) & = -\F_2(w,-\kappa,-\beta),
    \end{aligned}
\end{equation}
where the second identity follows from the fact that $w_{\xi\xi}/(1+w_\eta)$ is odd in $\eta$ when $w \in \W$.

\subsection{Local bifurcation}

We can now state and prove the main result of the section, which furnishes small-amplitude solitary waves with submerged point vortices near the bed and with small vortex strength.

\begin{theorem}[Small-amplitude waves]
    \label{small-amplitude theorem}
    For any supercritical Froude number $F^2 > 1$, there exists a curve $\Curve_\loc$ of solitary gravity waves with a submerged point vortex having the following properties:
    \begin{enumerate}[label=\textup{(\alph*)}]
        \item \label{local uniqueness part} There is a neighborhood of $(0,0) \in \X \times \R$ in which $\Curve_\loc$ comprises the entire zero-set of $\F$.
        \item \label{local parameterization part} The curve admits the real-analytic parameterization
              \[
                  \Curve_\loc = \brac*{ (w^\beta,\kappa^\beta,\beta) : \abs{\beta} < \beta_0 } \subset \Open \subset \X \times \R,
              \]
              for some $\beta_0 > 0$, where $(w^0,\kappa^0) = (0,0)$. Moreover, it enjoys the symmetry
              \begin{equation}
                  \label{small amplitude symmetry}
                  (w^{-\beta}, \kappa^{-\beta}) = (w^\beta, -\kappa^\beta)
              \end{equation}
              for all $\abs{\beta} < \beta_0$.
        \item \label{leading-order part} The solutions along $\Curve_\loc$ have the leading-order form
              \begin{equation}
                  \label{small amplitude asymptotics}
                  \begin{aligned}
                      w^\beta      & = \ddot{w}\beta^2 + O(\beta^4)                                   \\
                      \kappa^\beta & = -\frac{1}{\pi} \partial_\eta^3\ddot{w}(0)\beta^3 + O(\beta^5),
                  \end{aligned}
              \end{equation}
              in $\W$ and $\R$, respectively, where $\ddot{w} \in \W$ is the unique element satisfying
              \begin{equation}
                  \label{ddot w characterization}
                  \lbrac*{
                      \begin{aligned}
                          \Delta \ddot{w}                        & = 0                                       &  & \text{in } \confD  \\
                          \ddot{w}_\eta - \frac{1}{F^2} \ddot{w} & = \pi^2 \sech^2\parn*{\frac{\pi}{2} \xi } &  & \text{on } \confS.
                      \end{aligned}
                  }
              \end{equation}
              Consequently, the vortex strength and altitude associated with $u^\beta$ satisfy
              \begin{align*}
                  \gamma^\beta & = -4 \tan(\pi \beta) + O(\beta^5)                        \\
                  b^\beta      & = \beta + \partial_\eta \ddot{w}(0)\beta^3 + O(\beta^5).
              \end{align*}
    \end{enumerate}
\end{theorem}
\begin{proof}
    Recalling~\eqref{definition u}, a quick calculation reveals that
    \[
        \F_{u}(0,0) =
        \begin{pmatrix}
            \rest*{\parn[\Big]{-\partial_\eta + \frac{1}{F^2} }}_{\confS} & 0  \\
            \frac{1}{\pi} \rest[\big]{\partial_\xi^2}_{\zeta = 0}         & -1
        \end{pmatrix}
        \eqc \mc{L} \vcentcolon \X \to \Y
    \]
    at the trivial solution. The upper left entry of this operator matrix is invertible precisely when $F^2 > 1$. This follows, for instance, from the existence of the strict supersolution $\dot w(x,y) = y$; see~\cite[Theorem~6.1]{lopezgomez2003classifying} and~\cite[Corollary~A.11]{wheeler2015pressure}. As $\mc{L}$ is bounded and lower triangular, and the other diagonal entry is also invertible, this immediately implies that $\mc{L}$ is an isomorphism $\X \to \Y$, at which point the existence and local uniqueness of the solution curve $\Curve_\loc$ is an immediate consequence of the analytic implicit function theorem. In particular, uniqueness and the symmetry properties of $\F$ in~\eqref{symmetry F} imply that $w^\beta = w^{-\beta}$ and $\kappa^{-\beta} = -\kappa^\beta$. This proves the statements in~\cref{local uniqueness part,local parameterization part}. Moreover, as a result of the symmetry~\eqref{small amplitude symmetry}, only even powers of $\beta$ will appear in the expansion for $w^\beta$, and only odd powers in the expansion for $\kappa^\beta$.

    Consider next the leading-order form of $u^\beta \ceq (w^\beta,\kappa^\beta)$ asserted in~\cref{leading-order part}. Differentiating the equation $\F(u^\beta,\beta) = 0$ with respect to $\beta$ gives
    \[
        \F_u(u^\beta,\beta)\partial_\beta u^\beta +\F_\beta(u^\beta,\beta) = 0,
    \]
    where we note that $\F_\beta(0,0) = 0$, and therefore $\partial_\beta u^0= 0$. Taking a second derivative then yields
    \[
        \mc{L} \partial_\beta^2 u^0 + \F_{\beta \beta}(0,0) =0
    \]
    after evaluating at $\beta = 0$. A direct calculation using the definition of $a$ in~\eqref{definition a} now shows that
    \[
        \F_{\beta\beta}(0,0) = \parn*{2\pi^2 \sech^2\parn*{\frac{\pi \xi}{2}},0},
    \]
    which gives the leading order term for $w^\beta$, with the claimed characterization of $\ddot{w}$. Similarly, the leading order term in $\kappa^\beta$ is obtained by differentiating yet again.
\end{proof}

While we will not need such representations in the present paper, it is interesting to note that the function $\ddot{w}$ appearing in the leading-order asymptotics \cref{leading-order part} can be expressed as an integral.

\begin{proposition}[Fourier representation of $\ddot{w}$]
    The element $\ddot{w}$ in \cref{leading-order part} can be written as
    \begin{equation}
        \label{ddot w as fourier transform}
        \ddot{w}(\zeta) = 2 \pi \frac{\sin(\eta \pi / 2)}{\cosh(\pi \xi /2) + \cos(\eta \pi / 2)} + 4\pi \ddot{v},
    \end{equation}
    where
    \begin{equation}
        \label{ddot v definition}
        \ddot{v}(\zeta) \ceq \frac{2}{\pi} \int_0^\infty \frac{1}{F^2 t \coth(t) - 1} \frac{\sinh(\eta t)}{\sinh(2t)}\cos(\xi t)\dee t.
    \end{equation}
\end{proposition}
\begin{proof}
    Deducing that the solution to~\eqref{ddot w characterization} satisfies
    \[
        \ddot{w}(\zeta) = 8 \int_0^\infty \frac{F^2 t}{F^2 t - \tanh(t)} \frac{\sinh(\eta t)}{\sinh(2t)}\cos(\xi t)\dee t
    \]
    is relatively straightforward, since the right-hand side of the boundary condition has an explicit Fourier transform~\cite[\nopp I.7.2]{oberhettinger1990tables}. From this,~\eqref{ddot w as fourier transform} immediately follows by using~\cite[\nopp I.7.20]{oberhettinger1990tables}.
\end{proof}

\begin{remark}
    The integral in~\eqref{ddot v definition} can further be developed as an infinite series, elucidating its exact asymptotics. Inserting the partial fraction decomposition
    \begin{equation}
        \label{partial fraction decomposition}
        \frac{1}{F^2 t \coth(t) - 1} = \frac{2}{F^2}\sum_{k=0}^\infty \frac{T_k}{t^2 + t_k^2}
    \end{equation}
    into~\eqref{ddot v definition}, and using~\cite[\nopp I.7.46]{oberhettinger1990tables} combined with~\cite[\nopp 1.91]{oberhettinger1973fourier} eventually yields
    \[
        \ddot{v}(\zeta) = \sum_{k=0}^\infty \frac{T_k}{\sin^2(t_k)}e^{-t_k \abs{\xi}} \sin(t_k \eta) - \frac{\sin(\eta \pi /2) \cosh(\xi \pi/2)}{\cosh(\pi \xi) - \cos(\pi \eta)}
    \]
    for all $\zeta \in \confD \setminus \brac{\xi = 0}$. Concretely, $\brac{t_k}$ is the sequence of positive solutions to $F^2 t\cot(t)=1$, with $t_k - k\pi \in (0,\pi/2)$, $t_k \nearrow \pi / 2$, and
    \[
        T_k \ceq \parn*{1 - \frac{1}{t_k^2}\frac{1}{F^2}\parn*{1-\frac{1}{F^2}}}^{-1}.
    \]
\end{remark}
\begin{remark}
    Using~\eqref{partial fraction decomposition} together with~\cite[\nopp 2.4.3.1, 2.4.3.13]{prudnikov1986integrals}, it is also possible to show from~\eqref{ddot v definition} that
    \[
        \partial_\eta^{2m+1} \ddot{v}(0) = \frac{4}{\pi F^2}(-1)^{m-1} \sum_{k=0}^\infty T_k t_k^{2m}\parn[\bigg]{\Psi\parn*{\frac{2 t_k}{\pi}} + \sum_{j=1}^{2m+1} \frac{2^j-1}{j} \parn*{\frac{2 t_k}{\pi}}^{-j} B_j}
    \]
    for all $m \geq 0$, where~\cite[\nopp 2.11.9]{luke1969special}
    \[
        \Psi(x) \ceq \frac{1}{2}\parn*{\psi\parn*{\frac{x}{2}+1}-\psi\parn*{\frac{x}{2}+\frac{1}{2}}} \sim - \sum_{j=1}^\infty \frac{2^j-1}{j} x^{-j} B_j
        \quad \text{ as } x \to \infty.
    \]
    Here $\psi$ is the usual digamma function, while $B_j$ is the $j$th Bernoulli number. These derivatives show up in the asymptotics of \cref{small-amplitude theorem}\ref{leading-order part}.
\end{remark}

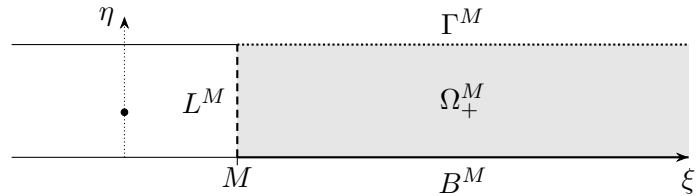
\begin{figure}[htb]
    \centering
    \tikzsetnextfilename{conformal_monotonicity}
    \input{figures/conformal_monotonicity.tikz}
    \caption{The half-strip $\confD_+^M$ in~\eqref{definition confDplus} along with its boundary components in~\eqref{definition confDplus boundary}.}
    \label{conformal monotonicity figure}
\end{figure}

\subsection{Monotonicity}
\label{local monotonicity section}

The next two results prove that near the point of bifurcation, $w$ has a certain monotonicity property related to~\eqref{intro monotonicity} in the statement of \cref{intro point vortex theorem}. Before stating this property precisely, we introduce some additional notation, illustrated in \cref{conformal monotonicity figure}. For any $M \ge 0$, we define the half strip
\begin{equation}
    \label{definition confDplus}
    \confD_+^M \ceq (M,\infty) \times (0,1),
\end{equation}
and its boundary components
\begin{equation}
    \label{definition confDplus boundary}
    L^M \ceq \brac{M} \times [0,1], \quad \confS^M \ceq (M,\infty) \times \brac{1}, \quad B^M \ceq (M,\infty) \times \brac{0}.
\end{equation}
These will be useful in the proofs.
\begin{definition}
    \label{definition conformal monotonicity}
    We call $(w,\kappa,\beta) \in \Open$ \emph{strictly conformally monotone} provided $w$ satisfies
    \begin{equation}
        \label{conformal monotonicity inequality}
        w_\xi < 0 \qquad \text{on } \confD_+^0 \cup \confS^0.
    \end{equation}
\end{definition}
Clearly this can be viewed as a condition on the conformal mapping $f$. Among other things, this condition implies that $y$-coordinate along the surface $\fluidS$ is strictly increasing (as a function of arc length, say) for $x \leq 0$ and strictly decreasing for $x \geq 0$. Note that this can happen not only for surfaces that are graphs like the one depicted in \cref{point vortex domains figure}, but also overhanging surfaces such as the one shown in \cref{cyclopean wave figure}.

\begin{lemma}[Asymptotic conformal monotonicity]
    \label{asymptotic conformal monotonicity lemma}
    For any $F^2 > 1$ and $M >0$, there exists $\delta > 0$ such that, if $(u,\beta) \in \Open \setminus \brac{(0,0)}$ satisfies~\eqref{Bernoulli condition} on $\confS^M$, obeys the bound $\norm{ w }_{C^1(\confS^M)} < \delta$, and
    \begin{alignat*}{-1}
        w_\xi & \leq 0 & \qquad & \text{on } L^M,                      \\
        \intertext{then necessarily we have that}
        w_\xi & < 0    &        & \text{on } \confD_+^M \cup \confS^M.
    \end{alignat*}
\end{lemma}

\begin{proof}
    By definition of the space $\W$, we have that $w_\xi$ is harmonic and it vanishes identically on $B^M$, as well as in the limit $\xi \to +\infty$. Fix now $ 0 < \varepsilon < F^2 - 1$, and define $v$ through
    \begin{equation}
        \label{v definition}
        v \ceq \frac{1 + \varepsilon}{\eta + \varepsilon} w_\xi,
    \end{equation}
    so that
    \[
        \Delta v + \frac{2}{\eta + \varepsilon} v_\eta = 0 \quad \text{in } \confD_+^M, \quad \text{with} \quad
        \begin{alignedat}{-1}
            v & = 0    &  & \text{on } B^M             \\
            v & \leq 0 &  & \text{on } L^M             \\
            v & \to 0  &  & \text{as } \xi \to \infty,
        \end{alignedat}
    \]
    and for the sake of contradiction, suppose that $v$ attains a positive maximum. By the maximum principle, this will happen at some point $\zeta_* \in \confS^M$, as $v \leq 0$ on $L^M \cup B^M$. There,
    \[
        v_\xi(\zeta_*) = 0 \quad \text{and} \quad v_{\eta}(\zeta_*) > 0,
    \]
    where the strict inequality follows from the Hopf boundary-point lemma~\cref{max principle}\ref{hopf lemma}.

    On the other hand, differentiating the Bernoulli condition~\eqref{Bernoulli condition}, we find that
    \[
        \brak*{\frac{1}{1+\varepsilon} A^2(1+w_\eta)-\frac{1}{F^2}}v + A^2 (1+w_\eta)v_\eta = A a_\xi < 0, \quad \text{where } A \ceq \frac{a}{v^2 + (1+w_\eta)^2},
    \]
    at $\zeta = \zeta^*$. In view of~\eqref{a inequality}, and our choice of $\varepsilon$, we see that the left-hand side is positive when $\delta$ is sufficiently small. This is a contradiction.
\end{proof}

\begin{lemma}[Local conformal monotonicity]
    \label{local conformal monotonicity lemma}
    By possibly shrinking $\beta_0$, we may ensure that all solutions on $(u^\beta,\beta) \in \Curve_\loc \setminus\brac{(0,0)}$ are strictly conformally monotone in the sense of \cref{definition conformal monotonicity}
\end{lemma}
\begin{proof}
    First, we claim that the leading-order part of the waves on $\Curve_\loc$ is strictly conformally monotone: Differentiating~\eqref{ddot w characterization} with respect to $\xi$, and defining $\ddot{v}$ and $\varepsilon$ like in~\eqref{v definition}, we find that
    \[
        \Delta \ddot{v} + \frac{2}{\eta + \varepsilon} \ddot{v}_\eta = 0 \quad \text{in } \confD_+^0, \quad \text{with} \quad
        \begin{alignedat}{-1}
            \brak*{\tfrac{1}{1+\varepsilon} - \tfrac{1}{F^2}}\ddot{v} + \ddot{v}_\eta & < 0   &  & \text{on } \confS^0        \\
            \ddot{v}                                                                  & = 0   &  & \text{on } B^0 \cup L^0    \\
            \ddot{v}                                                                  & \to 0 &  & \text{as } \xi \to \infty,
        \end{alignedat}
    \]
    Again seeking a contradiction, we suppose that $v$ attains a positive maximum, which must occur at some point $\zeta_* \in \confS^0$. Since $v_\eta(\zeta_*) > 0$ by the Hopf boundary-point lemma, we obtain our contradiction in the boundary condition on $\confS^0$.

    From~\cref{asymptotic conformal monotonicity lemma}, we know there is a $\delta > 0$ such that if $\norm{w^\beta}_{C^1(\confS^1)} < \delta$ and $w_\xi^\beta \leq 0$ on $L^1$, then $w_\xi^\beta < 0$ on $\confD_+^1 \cup \confS^1$. Possibly shrinking $\beta_0$, we may arrange both that $\norm{w^\beta}_{C^1(\confS^1)} < \delta$ and that
    \[
        w_\xi^\beta < 0 \qquad \text{on } (\confD_+^0 \cup \confS^0) \setminus (\confD_+^2 \cup \confS^2)
    \]
    for all $0 <\abs{\beta} < \beta_0$. The latter due to the asymptotics in~\eqref{small amplitude asymptotics}. Thus we have strict conformal monotonicity for all $0 < \abs{\beta} < \beta_0$.
\end{proof}

\section{Large solitary wave-borne point vortices}
\label{global point vortex section}

In this section, we extend the local solution curve $\Curve_\loc$ from \cref{small-amplitude theorem} to the large-amplitude regime using a \emph{global} implicit function theorem argument; ultimately leading us to~\cref{intro point vortex theorem}. As a first step in this direction, we have the following unrefined continuation result, which is the immediate product of applying the general theory from~\cref{homoclinic global ift}.

\begin{theorem}[Global bifurcation]
    \label{global ift}
    There exist a curve $\Curve \supset \Curve_\loc$ that admits the global $C^0$ parameterization
    \[
        \Curve \ceq \brac*{ (u(s), \beta(s)) : s \in \R } \subset \F^{-1}(0) \cap \Open
    \]
    with $(u(0),\beta(0)) = (0,0)$, and which satisfies the following:
    \begin{enumerate}[label=\textup{(\alph*)}]
        \item \label{symmetry} The curve enjoys the symmetry properties
              \[
                  w(-s) = w(s), \quad \kappa(-s) = -\kappa(s), \quad \text{and} \quad \beta(-s) =-\beta(s)
              \]
              for all $s \in \R$.
        \item \label{well behaved} The linearized operator $\F_u(u, \beta) \colon \X \to \Y$ is Fredholm of index $0$ for all $(u,\beta) \in \Curve$.
        \item \label{alternatives} One of the following limiting alternatives hold.
              \begin{enumerate}[label=\textup{(A\arabic*)}]
                  \item \label{blowup alternative}
                        \textup{(Blowup)} The quantity
                        \begin{equation}
                            \label{blowup}
                            N(s)\ceq \norm{w(s)}_{\W} + \frac{1}{\dist((u(s),\beta(s)), \partial \Open)}
                        \end{equation}
                        is unbounded as $s \to \infty$.
                  \item \label{loss of compactness alternative} \textup{(Loss of compactness)} There exists a sequence $s_n \to \infty$ with $\sup_n{N(s_n)} < \infty$, but for which $\{( u(s_n), \beta(s_n) )\}$ has no convergent subsequence in $\X \times \R$.
                  \item \label{loss of fredholmness alternative} \textup{(Loss of Fredholmness)} There exists a sequence $s_n \to \infty$ with $\sup_{n}{N(s_n)} < \infty$, and for which $\brac{(u(s_n), \beta(s_n))}$ converges to some $(u_*, \beta_*)$ in $\X \times \R$, but the operator $\F_u(u_*, \beta_*)$ is not Fredholm of index $0$.
                  \item \label{loop alternative} \textup{(Closed loop)} There exists $T > 0$ such that $(u(s+T), \beta(s+T)) = (u(s), \beta(s))$ for all $s \in \R$.
              \end{enumerate}
        \item \label{global reparam} Near each point $(u(s_0),\beta(s_0)) \in \Curve$, we can locally reparameterize $\Curve$ so that $s\mapsto (u(s),\beta(s))$ is real analytic.
        \item \label{maximal part} The curve $\Curve$ is maximal, in the sense that if $\Kurve \subset \F^{-1}(0) \cap \Open$ is a locally real-analytic curve containing $(u(0),\beta(0))$, and along which $\F_u$ is Fredholm index $0$, then $\Kurve \subset \Curve$.
    \end{enumerate}
\end{theorem}
\begin{proof}
    Since we have already verified that $\F_u(0,0)$ is an isomorphism $\X \to \Y$, this follows directly from the analytic implicit function theorem as stated in~\cref{homoclinic global ift}. The symmetry in \cref{symmetry} is a consequence of~\eqref{symmetry F}.
\end{proof}

The next several subsections address the realizability of each of the scenarios~\ref{blowup alternative}--\ref{loop alternative}, with the eventual outcome being the much improved set of alternatives claimed in~\cref{intro point vortex theorem}.

\subsection{Monotonicity}
\label{global monotonicity section}

In the previous section, we proved that each solution $(u,\beta)$ on the local curve $\Curve_\loc$ is strictly conformally monotone, which we recall means that~\eqref{conformal monotonicity inequality} holds. On the other hand, we wish to ultimately show that the vertical velocity $\mf{v}$ exhibits the sign~\eqref{intro monotonicity}, which also implies a sort of monotonicity. Indeed, if~\eqref{intro monotonicity} holds, then starting at the crest line $\brac{ x = 0 }$ in the $z$-plane and moving along any streamline above the bed, the vertical coordinate will be strictly decreasing. Let us now prove that these two notions of monotonicity are actually equivalent for solutions to~\eqref{abstract operator equation}: the vertical velocity obeys~\eqref{intro monotonicity} precisely when the corresponding $w$ is conformally monotone~\eqref{conformal monotonicity inequality}.

These types of arguments are commonplace in global bifurcation-theoretic studies of water waves. However, there is a novel element here, because the velocity field has a singularity at the point vortex, and thus special attention and some new ideas are needed to understand the behavior there. Rather than work with the vertical velocity $\mf{v}$ in the $z$-plane, it will be more convenient to consider the conformal vertical velocity $V = \mf{v} \circ f$. It can be found from $(u,\beta)$ via
\begin{equation}
    \label{formula V}
    V = -\im{\frac{W_\zeta}{f_\zeta}} = -\im{\frac{W_\zeta}{1+w_\eta+i w_\xi}}
\end{equation}
in $\confD \setminus \brac{\pm i\beta}$, and the sign condition~\eqref{intro monotonicity} is clearly equivalent to
\begin{equation}
    \label{definition streamline monotonicity}
    V < 0 \qquad \text{on } \confD_+^0 \cup \confS^0.
\end{equation}

As a preparatory lemma, let us first establish this sign near the point vortex.

\begin{lemma}[Sign near vortex]
    \label{sign V near vortex lemma}
    Let $(u,\beta) \in \Open \cap \F^{-1}(0)$ be given with $\beta \neq 0$. Then
    \begin{equation}
        \label{sign of V near vortex}
        V < 0 \qquad \text{on } \partial B_r(i\abs{\beta}) \cap \confD_+^0.
    \end{equation}
    for some $r = r(\beta,\kappa) > 0$.
\end{lemma}
\begin{proof}
    We assume without loss of generality that $\beta > 0$. As $(u,\beta) \in \Open$, the function $f_\zeta = 1+w_\eta+i w_\xi$ is real on both axes, converges to $1$ at infinity, and is non-vanishing. In fact, a continuity argument shows that $f_\zeta$ must be positive on both axes, and therefore $f_\zeta(i\beta) > 0$. Next we observe that
    \[
        \frac{W_\zeta}{f_\zeta} - \frac\gamma{2\pi i f_\zeta(i\beta)} \frac 1{\zeta - i\beta}
    \]
    is holomorphic at $i\beta$ due to Helmholtz--Kirchhoff~\eqref{conformal helmholtz kirchhoff}, and necessarily also real on imaginary. Thus taking imaginary parts yields
    \[
        V = -\im\frac{W_\zeta}{f_\zeta}
        =\xi\parn*{\frac{\gamma}{2\pi f_\zeta(i\beta)} \frac{1}{\abs{\zeta - i\beta}^2}
            + O(1)}
    \]
    as $\zeta \to i\beta$. From this,~\eqref{sign of V near vortex} follows since $\gamma < 0$ by~\eqref{gamma inequality}.
\end{proof}

\begin{lemma}[Monotonicity equivalence]
    \label{monotonicity equivalent lemma}
    A solution $(u,\beta) \in \Open \cap \F^{-1}(0)$ to the wave-borne point vortex problem is strictly conformally monotone in the sense of \Cref{definition conformal monotonicity} if and only if the corresponding conformal vertical velocity $V$ satisfies~\eqref{definition streamline monotonicity}
\end{lemma}
\begin{proof}
    From~\eqref{definition complex potential}, we can see that $W_\zeta$ is real on the axes, and thus by~\eqref{formula V} we have that $V$ vanishes identically there, except for at $\zeta = \pm i\beta$. Likewise, we have via an elementary computation that
    \begin{equation}
        \label{V on Gamma}
        V = \frac{ aw_\xi}{w_\xi^2 + (1+w_\eta)^2}
    \end{equation}
    on $\confS$. Thus by~\eqref{a inequality}, the signs of $V$ and $w_\xi$ coincide on $\confS$.

    Suppose now that $(u,\beta)$ is strictly conformally monotone. By the above paragraph and \cref{sign V near vortex lemma}, we know that for all $0 < r \ll 1$, $V$ is harmonic on $S_r \ceq \confD_+^0 \setminus B_r(i\abs{\beta})$, vanishes at infinity, and that $V \leq 0$ on $\partial S_r$. As $V < 0$ on $\confS^0$, the strong maximum principle therefore implies that $V$ satisfies~\eqref{definition streamline monotonicity}. Conversely, if $V$ satisfies~\eqref{definition streamline monotonicity}, then $w_\xi \leq 0$ on $\partial \confD_+^0$, vanishes at infinity, and is harmonic on $\confD_+^0$. Since $w_\xi < 0$ on $\confS^0$, the strong maximum principle shows that the solution is strictly conformally monotone.
\end{proof}

Since these various notions of monotonicity are all equivalent for solutions to~\eqref{abstract operator equation}, we may simply refer to solutions as strictly monotone without any ambiguity. Our next objective is to prove that strict monotonicity persists globally along $\Curve$, by showing that it is both an open and closed property in $\Open \cap \F^{-1}(0)$.

\begin{lemma}[Open property]
    \label{open property lemma}
    Let $(u_*,\beta_*) \in \Open \cap \F^{-1}(0)$ be given and suppose that it is strictly monotone. Then there exists $\varepsilon > 0$ such that if $(u,\beta) \in \Open \cap \F^{-1}(0)$ and
    \[
        \norm{ w - w_* }_{C^2(\confD_+^0)} + \abs{\kappa - \kappa_*} + \abs{\beta - \beta_*} < \varepsilon,
    \]
    then $(u,\beta)$ is also strictly monotone.
\end{lemma}

\begin{proof}
    This follows from continuity and the asymptotic conformal monotonicity established in~\cref{asymptotic conformal monotonicity lemma}, by a standard argument involving the Serrin corner-point lemma. See, for example, the proof of~\cite[Lemma 4.21]{chen2018existence}.
\end{proof}

We remark that the proof of \cref{open property lemma} is in fact the only place where having $C^3$ regularity for the conformal mapping is actually required.

\begin{lemma}[Closed property]
    \label{closed property lemma}
    Let $\brac{ (u_n, \beta_n) } \subset \Open \cap \F^{-1}(0)$ be given and suppose that $(u_n, \beta_n) \to (u,\beta) \in \Open \cap \F^{-1}(0)$ in $\X \times \R$. If $F^2 > 1$ and each $(u_n,\beta_n)$ is strictly monotone, then the limit $(u,\beta)$ is either strictly monotone or the trivial solution $(0,0)$.
\end{lemma}

\begin{proof}
    Let the sequence $\brac{(u_n,\beta_n)}$ and its limit $(u,\beta)$ be as in the hypothesis. If $\beta = 0$, then the limiting wave is irrotational by~\eqref{definition gamma}. For supercritical Froude numbers, all nontrivial irrotational solitary waves are strictly monotone~\cite{craig1988symmetry}, so the result holds. Suppose therefore that $\beta \neq 0$, and denote the corresponding limiting conformal vertical velocity by $V$. Continuity and \cref{monotonicity equivalent lemma} ensure that
    \[
        V \leq 0 \qquad \textrm{on } \ol{\confD_+^0} \setminus \brac{i\abs{\beta}},
    \]
    and that $V$ vanishes at infinity. By \cref{sign V near vortex lemma}, we have $V < 0$ on $\partial B_r(i\abs{\beta}) \cap \confD_+^0$ for all $0 < r \ll 1$. The strong maximum principle shows that $V < 0$ on $\confD_+^0 \setminus B_r(i\abs{\beta})$ for all $0 < r \ll 1$, and therefore on $\confD_+^0$. It remains only to prove that $V < 0$ also on $\confS^0$.

    Differentiating the dynamic boundary condition~\eqref{Bernoulli condition} with respect to $\xi$ gives
    \begin{equation}
        \label{differentiated Bernoulli condition}
        V V_\xi-U V_\eta + \frac{1}{F^2} w_\xi = 0 \qquad \text{on } \confS^0,
    \end{equation}
    where $U = \re{(W_\zeta/f_\zeta)}$ is the conformal horizontal velocity, and we have used the Cauchy--Riemann equations. Seeking a contradiction, suppose that $V = 0$ at some $\zeta^* \in \confS^0$. Recalling~\eqref{V on Gamma}, we find that~\eqref{differentiated Bernoulli condition} reduces to
    \[
        U(\zeta^*)V_\eta(\zeta^*) = 0
    \]
    at this point, where $V_\eta(\zeta^*) > 0$ by the Hopf boundary-point lemma~\cref{max principle}~\ref{hopf lemma}. Therefore $\zeta^*$ is a stagnation point, where $U(\zeta^*) = V(\zeta^*) = 0$, which contradicts membership in $\Open$.
\end{proof}

\begin{corollary}[Monotonicity]
    \label{nodal property corollary}
    If $\mc{K}$ is any connected subset of $\Open \cap \F^{-1}(0)$ that contains $\Curve_\loc$, then every $(u,\beta) \in \mc{K} \setminus \brac{(0,0)}$ is strictly monotone.
\end{corollary}

\begin{proof}
    We know that the solutions on $\Curve_\loc \setminus \brac{(0,0)}$ are strictly monotone by \cref{local conformal monotonicity lemma,monotonicity equivalent lemma}, and \cref{small-amplitude theorem}\ref{local uniqueness part} tells us that these are the only solutions near $(0,0)$. Combined with \cref{open property lemma,closed property lemma}, we conclude that the set of $(u,\beta) \in \mc{K}$ that are either strictly monotone or the trivial solution $(0,0)$ is both open and closed in $\mc{K}$. Since $\mc{K}$ is connected, the result follows.
\end{proof}

\subsection{Precompactness and Fredholmness}

The next lemma adapts the ideas in~\cite[Lemma~6.3]{chen2018existence} to the present setting. In particular, it will allow us to rule out the loss of compactness alternative~\ref{loss of compactness alternative} in~\cref{global ift}.

\begin{lemma}[Compactness dichotomy]
    \label{compactness or front lemma}
    Suppose that $\brac{ (u_n,\beta_n) } \subset \F^{-1}(0) \cap \Open$ satisfies
    \[
        \sup_{n\geq 1}\parn*{ \norm{w_n}_{\W} + \frac{1}{\dist((u_n,\beta_n), \partial \Open)} } < \infty,
    \]
    and that each $(u_n, \beta_n)$ is strictly monotone. Then, either
    \begin{enumerate}[label=\textup{(\roman*)}]
        \item \label{compactness alternative} \textup{(Compactness)} $\brac{ (u_n,\beta_n)}$ has a convergent subsequence in $\X \times \R$; or
        \item \label{front alternative} \textup{(Irrotational front)} there exists a sequence of translations $\brac{\xi_n} \subset \R$, with $\xi_n \to \infty$, such that
              \[
                  w_n(\placeholder+\xi_n) \rightarrow \tilde w \in C^{k+3+\alpha}(\ol{\confD}) \qquad \text{in } C_\loc^{k+3}(\confD)
              \]
              after possibly moving to a subsequence. The limit is a nontrivial \emph{monotone irrotational front}, in that $\tilde w_\xi \leq 0$ in $\confD_+ \cup \confS$, $\tilde w$ is odd in $\eta$, and solves
              \begin{equation}
                  \label{irrotational problem}
                  \lbrac*{
                      \begin{aligned}
                          \Delta \tilde w                                                                     & = 0           & \qquad & \text{in } \confD  \\
                          \frac{1}{2} \frac{1}{\tilde w_\xi^2 + (1+\tilde w_\eta)^2} + \frac{1}{F^2} \tilde w & = \frac{1}{2} &        & \text{on } \confS.
                      \end{aligned}
                  }
              \end{equation}
    \end{enumerate}
\end{lemma}
\begin{proof}
    Let $\brac{(u_n,\beta_n)}$ be as in the hypothesis of the lemma. Possibly passing to a subsequence, we may arrange that $(\kappa_n, \beta_n)$ converges to some $(\kappa,\beta)$ with $\cos(\pi\beta) > \kappa \sin(\pi\beta)$ and $\abs{\beta} < 1$. If we first suppose that $\brac{w_n}$ is equidecaying, in the sense that
    \[
        \adjustlimits\lim_{M \to \infty}\sup_{n \geq 1}{\norm{ w_n }_{C^2(\confD_+^M)}} =0,
    \]
    then a standard argument confirms that $\brac{w_n}$ has a convergent subsequence in $\W$. Thus, in this case, the compactness alternative holds.

    Assume instead then that $\brac{w_n}$ is \emph{not} equidecaying, which we will demonstrate leads to the front alternative. Possibly moving to a subsequence again, there must then exist $\varepsilon > 0$ and a sequence $\brac{\xi_n}$ with $\xi_n \to \infty$ such that
    \begin{equation}
        \label{tilde w lower bound}
        \norm{\tilde{w}_n}_{C^2(\confD_+^0)} > \varepsilon \qquad \text{for all } n \geq 1,
    \end{equation}
    for the translates $\tilde{w}_n \ceq w_n(\placeholder + \xi_n)$. Physically, these correspond to waves with vortices at $-\xi_n \pm i\beta_n$ in the conformal domain. As $\brac{\tilde w_n}$ is bounded in $C^{k+3+\alpha}(\ol{\confD})$, we can extract a convergent subsequence $\tilde w_n \to \tilde w \in C^{k+3+\alpha}(\ol{\confD})$ in $C^{k+3}_\loc(\confD)$, and this limit is nontrivial due to~\eqref{tilde w lower bound}. Note that because each $\tilde w_n$ is odd in $\eta$, so too is $\tilde w$. Local convergence is enough to guarantee both that $\tilde w$ solves the irrotational problem~\eqref{irrotational problem} and that $\tilde w_\xi \leq 0$ on $\confD_+ \cup \confS$ due to the assumed monotonicity.
\end{proof}

As nontrivial irrotational fronts do not exist for gravity waves beneath air~\cite{rayleigh1914theory,lamb1993hydrodynamics}, we therefore have the following immediate corollary.

\begin{corollary}[Precompactness]
    \label{precompactness corollary}
    The loss of compactness alternative~\ref{loss of compactness alternative} cannot occur.
\end{corollary}

\begin{proof}
    Since $\Curve_\loc \subset \Curve$,~\cref{nodal property corollary} ensures that all the waves on $\Curve \setminus \brac{(0,0)}$ are strictly monotone. If the loss of compactness alternative~\ref{loss of compactness alternative} did occur, it would give rise to a nontrivial irrotational front by~\cref{compactness or front lemma}\ref{front alternative}.
\end{proof}

We next show that also loss of Fredholmness can be eliminated from \cref{global ift}.

\begin{lemma}[Global Fredholmness]
    \label{global fredholm lemma}
    The loss of Fredholmness alternative~\ref{loss of fredholmness alternative} cannot occur.
\end{lemma}

\begin{proof}
    Suppose that we have a sequence $\brac{s_n}$ such that $s_n \to \infty$, while the quantity~\eqref{blowup} satisfies $\sup_n N(s_n) < \infty$, and $(u(s_n), \beta(s_n)) \to (u_*,\beta_*)$ in $\X \times \R$. We claim that $\F_u(u_*,\beta_*)$ is necessarily Fredholm of index $0$. A trivial variant of the argument in~\cite[Appendix A.3]{wheeler2013solitary} shows that $\F_{1w}(u_*,\beta_*)$ is semi-Fredholm provided that the kernel of the associated constant-coefficient elliptic boundary-value problem in the limit $\xi \to \infty$ is trivial. Because $w$ vanishes as $\xi\to \infty$, this problem is in fact the same as for the operator $\F_{1w}(0,0) \colon \W \to \Y_1$, which we recall from the proof of \cref{small-amplitude theorem} is invertible. Thus $\F_{1w}(u,\beta)$ is semi-Fredholm along $\Curve$, and at $(u_*,\beta_*)$. By the Fredholm bordering result \cref{bordering lemma} and the continuity of the index, we conclude that the full operator $\F_u(u_*,\beta_*)$ is Fredholm of index $0$ as a map $\X \to \Y$.
\end{proof}

\subsection{On the closed loop alternative}

Lastly, we consider the possibility that the global curve $\Curve$ is a closed loop. Since $s \beta(s) > 0$ for $0 < s\ll 1$ by \cref{small-amplitude theorem} and \cref{global ift}\ref{symmetry}, this can only occur if $\beta$ changes sign. Hence there must be a nontrivial wave $(w_0,\kappa_0,0) \in \Curve$, for which the symmetry properties of $w$ imply that~\eqref{vortex equation} forces $\kappa_0 = 0$ as well.

To emphasize that a closed loop is in principle a very real possibility, we show that a point vortex can be injected into any non-degenerate irrotational solitary wave. By non-degenerate, we here mean the essentially generic condition that $\F_{1w}$ is invertible at the solitary wave. This holds for example along any distinguished arc of a global bifurcation curve of irrotational solitary waves.

\begin{theorem}[Vortex injection]
    \label{vortex injection theorem}
    Fix a supercritical Froude number $F^2 > 1$ and suppose that
    \[
        (u_0,0) = (w_0,0,0) \in \F^{-1}(0) \cap \Open
    \]
    is an irrotational solitary wave at which $\F_{1w}(u_0,0)$ is an isomorphism $\W \to \Y_1$. Then there exists a global $C^0$ curve
    \[
        \Curve_{w_0} \ceq \brac*{ (u(s), \beta(s)) : s \in \R } \subset \F^{-1}(0) \cap \Open
    \]
    with $(u(0),\beta(0)) = (u_0,0)$, satisfying the following:
    \begin{enumerate}[label=\textup{(\alph*)}]
        \item \textup{(Symmetry)} The curve enjoys the symmetries of~\cref{global ift}\ref{symmetry}
              \label{vi symmetry part}
        \item \textup{(Monotonicity)}
              Each wave on $\Curve_{w_0} \setminus \{(0,0)\}$ is strictly monotone.
              \label{vi monotonicity part}
        \item \textup{(Loop or blowup)} Either the blowup~\ref{blowup alternative} or closed loop alternative~\ref{loop alternative} must occur.
              \label{vi alternatives part}
        \item \textup{(Loop criterion)}
              If $\Curve_{w_0}$ is a closed loop, then it must contain an irrotational solitary wave distinct from $(u_0,0)$. The converse is true if this wave is non-degenerate.
              \label{vi closed loop part}
        \item \textup{(Analyticity and maximality)}
              The curve $\Curve_{w_0}$ is locally real analytic, and maximal in the sense of~\cref{global ift}\ref{maximal part}.
              \label{vi well-behaved part}
    \end{enumerate}
\end{theorem}

\begin{proof}
    As in the proof of~\cref{small-amplitude theorem}, we compute that
    \[
        \F_u(u_0,0) =
        \begin{pmatrix}
            \F_{1w}(u_0,0) & 0  \\
            \F_{2w}(u_0,0) & -1
        \end{pmatrix},
    \]
    where the upper left entry is an isomorphism $\W \to \Y_1$ by hypothesis. Since the lower right entry is nonzero, and therefore invertible, the full operator $\F_u(u_0,0)$ is an isomorphism $\X \to \Y$. We need only apply the global implicit function theorem~\cite[Theorem B.1]{chen2023global} to see that a global curve $\Curve_{w_0}$ exists satisfying~\crefrange{symmetry}{maximal part} of~\cref{global ift}. In particular, this immediately gives~\cref{vi symmetry part,vi well-behaved part} of the present theorem.

    If $(u_0,0)$ is the trivial solution, then all the waves on $\Curve_{w_0} \setminus \{(0,0)\}$ are strictly monotone by \cref{local conformal monotonicity lemma}. If instead $(u_0,0)$ is a nontrivial irrotational solitary wave, then necessarily it is strictly monotone by the classical work of~\textcite{craig1988symmetry}, and the same sequence of results ensures that this monotonicity persists globally. This proves~\cref{vi monotonicity part}. Likewise, we may apply~\cref{precompactness corollary,global fredholm lemma} to exclude the potential loss of compactness and Fredholmness. Thus, as claimed in~\cref{vi alternatives part}, we are left only with the possibility of blowup or a closed loop.

    It remains to prove~\cref{vi closed loop part} of the theorem. The forward implication is true for the same reasons discussed at the beginning of this subsection, so we focus on the converse. Suppose that $(u(s_1),\beta(s_1)) = (u_1,0)$ is a non-degenerate irrotational wave distinct from $(u_0,0)$, and denote by $\msc{A}$ the part of $\Curve_{w_0}$ corresponding to $s \in [-s_1,s_1]$. Due to the symmetry in \cref{vi symmetry part}, $\msc{A}$ is a closed curve. Moreover, $\Open \cap \F^{-1}(0)$ is a real-analytic curve in a neighborhood of $(u_1,0)$ by non-degeneracy and the local implicit function theorem. Thus $\msc{A} \subset \Curve_{w_0}$ is a closed locally real-analytic curve, and therefore equal to $\Curve_{w_0}$.
\end{proof}

Clearly,~\cref{small-amplitude theorem} could have been phrased as a special case of the above result. Also, note that the hypothesis that $F^2 > 1$ is superfluous whenever $(u_0,0)$ is \emph{nontrivial}, by the sharp lower bound on the Froude number in~\cite{kozlov2021subcritical}. On the other hand, any nontrivial irrotational solitary wave satisfies the upper bound $F^2 < 2$; see~\cite{starr1947momentum,keady1974bounds}. As an immediate corollary of~\cref{vortex injection theorem}, we can therefore rule out the closed loop alternative for the curve $\Curve$ when the Froude number is larger than this.

\begin{corollary}[Blowup]
    \label{no closed loop corollary}
    Suppose that $F^2 \geq 2$. Then the global curve $\Curve$ furnished by~\cref{global ift} is not a closed loop, and the blowup alternative~\ref{blowup alternative} must occur.
\end{corollary}

\begin{proof}
    By uniqueness and maximality, the same curve $\Curve$ can be obtained by applying~\cref{vortex injection theorem} to the trivial solution $(0,0)$ for fixed $F^2 \geq 2$. Because there exist no nontrivial solitary waves with this Froude number,~\cref{vi closed loop part} of the same theorem ensures that $\Curve$ cannot be a closed loop, and hence blowup must occur according to~\cref{vi alternatives part}.
\end{proof}

\subsection{Uniform regularity}

Next, we study the blowup alternative~\ref{blowup alternative} in more detail, in the hopes of better classifying the types of singularities that can occur in the limit along $\Curve$. The main result is as follows.

\begin{lemma}[Uniform regularity]
    \label{uniform regularity lemma}
    For all $K > 0$ and $F^2 > 1$, there exists a constant $C = C(K,F^2) > 0$ such that, if $(u,\beta) \in \F^{-1}(0) \cap \Open$ obeys the bound
    \begin{equation}
        \label{C1 bound assumption}
        \sup_{\confS}\parn[\bigg]{ \abs{\nabla w } + \frac{1}{w_\xi^2 + (1+w_\eta)^2} } + \abs{\gamma(\kappa,\beta)} < K,
    \end{equation}
    the corresponding quantity $N$ given by~\eqref{blowup} satisfies $N < C$.
\end{lemma}

\begin{proof}
    Throughout the proof, we use $C$ as a generic positive constant depending only on $K$ and $F^2$. First, observe that $w$, $w_\xi$, and $w_\eta$ are all harmonic in $\confD$ and decay at infinity. By the maximum principle and symmetry, they are all maximized along $\confS$. Since $w$ vanishes along $\{ \eta = 0 \}$, moreover,
    \[
        \sup_{\confD}{\abs{w}} \leq \sup_{\confD}{\abs{w_\eta}} \leq K.
    \]
    The Bernoulli condition~\eqref{Bernoulli condition} then allows us to control
    \[
        \sup_{\confS}{a^2} = \sup_{\confS}{\parn*{ \parn*{ \frac{1}{2} - \frac{1}{F^2}w }\parn*{ w_\xi^2+(1+w_\eta)^2 } }} \leq C.
    \]
    On the other hand, from this, the bound~\eqref{C1 bound assumption}, the lower bound on $a$~\eqref{a inequality}, and the nonlinear estimate~\cite[Theorem 3]{lieberman1987nonlinear}, we find that
    \[
        \norm{ w }_{C^2(\confD)} \leq C (1+ \norm{ \nabla w}_{C^0(\confD)}) \leq C (1+K).
    \]
    Standard Schauder estimates for $w_\xi$ and $w_\eta$ allow us to upgrade this to control of $w$ in $C^{k+3+\alpha}(\confD)$, so that we have control of $\norm{w}_{\W} < C$.

    Now, because $w+\eta$ is harmonic in $\confD$ and $\nabla (w+\eta)$ is non-vanishing, we have that $\log{\abs{\nabla(w+\eta)}}$ is harmonic in $\confD$. Hence, by the maximum principle
    \[
        \sup_{\confD}{\frac{1}{w_\xi^2 +(1+w_\eta)^2}} = \sup_{\confS}{\frac{1}{w_\xi^2 +(1+w_\eta)^2}} < K.
    \]
    This gives uniform control of the conformality in terms of $K$. It also furnishes a bound on $\kappa$: recall from~\eqref{definition F} that
    \[
        \kappa^2 = \frac{1}{\pi^2} \rest[\bigg]{\frac{w_{\xi\xi}^2}{(1+w_\eta)^2}}_{\zeta = i\beta} = \frac{1}{\pi^2} \rest[\bigg]{\frac{w_{\xi\xi}^2}{w_\xi^2+ (1+w_\eta)^2}}_{\zeta = i\beta} \leq C.
    \]
    Then, since $(u,\eta) \in \Open$, we have
    \[
        \cos{(\pi \beta)} + C \abs{ \sin{(\pi\beta)} } \geq \cos{(\pi\beta)} - \kappa \sin{(\pi\beta)} > 0,
    \]
    and thus $\abs{\beta}$ is bounded away from $1$ in terms of $C$. Likewise, the boundedness of $\kappa$ implies there exists $\varepsilon = \varepsilon(C) > 0$ such that $\cos{(\pi\beta)} - \kappa \sin{(\pi\beta)} > 1/2$ for $\abs{\beta} < \varepsilon$. On the other hand, from the definition of $\gamma$ in~\eqref{definition gamma} we have
    \[
        \frac{1}{\cos{(\pi\beta)} - \kappa \sin{(\pi\beta)}} = -\frac{\gamma}{4\sin{(\pi\beta)}} < C \qquad \textrm{for } \varepsilon \leq \abs{\beta} < 1.
    \]
    Thus, recalling the definition of $\Open$ in~\eqref{definition of O nbhd},
    \[
        \frac{1}{\dist((u,\beta), \partial\Open)} \lesssim \sup_{\confD}{\frac{1}{w_\xi^2 +(1+w_\eta)^2}} + \frac{1}{1-\abs{\beta}} + \frac{1}{\cos{(\pi\beta)}-\kappa\sin{(\pi\beta)}} < C.
    \]

    Combining all of the above estimates, we conclude that
    \[
        N = \norm{ w }_{\W} + \abs{\kappa} + \frac{1}{\dist((u,\beta), \partial\Open)} < C. \qedhere
    \]
\end{proof}

\subsection{Proof of \texorpdfstring{\cref{intro point vortex theorem}}{Theorem~\ref{intro point vortex theorem}}}

Assembling the results of the previous sections, we can now prove the main theorem on the existence of wave-borne point vortices.

\begin{proof}[Proof of~\cref{intro point vortex theorem}]
    Let $F^2 > 1$ be given. Then by~\cref{small-amplitude theorem,global ift}, there exists a curve $\Curve$ of solitary gravity wave with submerged point vortices that bifurcates from the trivial (irrotational) flow. Moreover, $\Curve$ admits a global $C^0$ parameterization that is locally real analytic. Each solution on $\Curve$ exhibits the claimed symmetry properties by the choice of function spaces, and their monotonicity was proved in~\cref{nodal property corollary}.

    We must also verify that each of the solutions on $\Curve$ is physical in that the corresponding mapping $f(s)$ is injective on $\ol{\confD}$ and $\inf_{\confD} \abs{\partial_\zeta f(s)} > 0$. But a positive lower bound on $\abs{\partial_\zeta f(s)}$ follows by construction, as $\Curve \subset \Open$. To prove injectivity, by the Darboux--Picard theorem~\cite[Corollary 9.16]{burckel1979introduction}, it suffices to show that $\rest{f(s)}_{\partial\confD}$ is injective. In fact, this is a consequence of the monotonicity established in~\cref{nodal property corollary}; see the proof of~\cite[Theorem 1.1]{haziot2023large}.

    Consider next the limiting behavior at the extreme of the curve. By~\cref{precompactness corollary} and~\cref{global fredholm lemma}, neither the loss of compactness~\ref{loss of compactness alternative} nor loss of Fredholmness~\ref{loss of fredholmness alternative} alternatives are possible, so~\cref{global ift} tells us that either we have blowup~\ref{blowup alternative} or $\Curve$ is a closed loop. The uniform bounds furnished by~\cref{uniform regularity lemma} allow us to conclude that, were blowup~\ref{blowup alternative} to happen, so must~\eqref{intro theorem blowup alternative}. If instead the closed loop alternative occurs, then by~\cref{vortex injection theorem}\ref{vi closed loop part}, we know $\Curve$ must contain an irrotational solitary wave distinct from the starting trivial one. Finally, we saw in~\cref{no closed loop corollary} that it is impossible for $\Curve$ to be a closed loop if $F^2 > 2$. This completes the proof.
\end{proof}

\section{Periodic wave-borne point vortices}
\label{numerics section}

The main purpose of this section is to numerically investigate the periodic analogue of the solitary wave-borne point vortex problem~\eqref{intro Euler-Kirchhoff-Helmholtz}. Physically, this corresponds to a traveling periodic wave carrying a single point vortex per period. The conformal description~\eqref{conformal domain definition} and~\eqref{symmetry f} stays the same, while $2\lambda$-periodicity of $w$ in $\xi$ and the normalizing condition
\[
    w(\lambda + i) = 0
\]
replace the asymptotic condition~\eqref{asymptotics f}, for some $\lambda > 0$. The preimages of the physical point vortices are now the equally spaced points
\[
    \brac[\big]{ 2 n \lambda + i \beta : n \in \Z} \subset \confD_+.
\]
As long as $\lambda$ is sufficiently large, we expect the periodic solutions to be good approximations of the solitary ones.

Recasting the governing equations in terms of the unknowns $(w,\kappa,\beta)$ requires replacing $W$ in~\eqref{definition complex potential} with the complex potential for the $2\lambda$-periodic strip. Towards that end, it is convenient to introduce two special functions, with which we may give an explicit representation formula for $W$.

\begin{definition}
    Given a lattice $\Lambda \subset \C$, the entire function
    \[
        \wsigma(z) = \wsigma(z;\Lambda) \ceq z \prod_{\ell \in \Lambda \setminus \brac{0}} \parn*{1 - \frac{z}{\ell}}\exp\parn*{\frac{z^2}{2\ell^2} + \frac{z}{\ell}}
    \]
    and its logarithmic derivative $\wzeta \ceq \wsigma_z / \wsigma$, are known as the \emph{Weierstrass sigma} and \emph{zeta functions}, respectively.
\end{definition}
\begin{proposition}[Periodic complex potential]
    For each half-period $\lambda > 0$, let
    \begin{equation}
        \label{periodic complex potential}
        W= W(\zeta; \gamma, \beta, \lambda) \ceq \frac{\gamma}{2\pi i}\parn*{
            \log\parn*{
                \frac{\wsigma(\zeta - i \beta)}{\wsigma(\zeta + i \beta)}} + 2\beta \wzeta(i)\zeta
        } + \zeta,
    \end{equation}
    with $\wsigma = \wsigma(\placeholder; \Lambda)$ for the lattice $\Lambda \ceq 2\lambda \Z \times 2 \Z$. Then $W$ satisfies the following.
    \begin{enumerate}[label=\textup{(\roman*)}]
        \item $W_\zeta$ is meromorphic, with simple poles at $\pm i \beta + 2\lambda \Z$ in $\confD$ when $\beta \neq 0$, having residues $\pm\gamma$;
        \item \label{periodic flux} $\im{W}=1$ on $\confS$ and $\im{W}=0$ on $\R$;
        \item $W_\zeta$ is $2\lambda$-periodic in $\xi$, and real on both axes.
    \end{enumerate}
\end{proposition}
\begin{proof}
    Without loss of generality, we assume that $\beta \in (0,1)$. We first map the period rectangle $\brac{\zeta \in \C : -\lambda < \xi \leq \lambda, 0 < \eta < 1 }$ to the annulus $\brac{\zeta \in \C : \exp(-\pi / \lambda) < \abs{\zeta} < 1}$ using $\zeta \mapsto \exp(i \pi \zeta / \lambda)$, taking the point vortex to $\exp(-\pi \beta / \lambda)$. On this annulus, we can use the tools of~\cite{crowdy2020solving}. Specifically, the so-called modified Green's function can be computed by combining~\cite[(14.82) and (14.92)]{crowdy2020solving}. This is almost the desired result, except that the constant value of $\im{W}$ on $\confS$ is not $1$; see~\cite[(5.20)]{crowdy2020solving}. By compensating for this, and using~\cite[(6.2.8) and (6.2.9)]{lawden1989elliptic} to express the result in terms of $\wsigma$, we arrive at~\eqref{periodic complex potential}.
\end{proof}

As before, the velocity field computed from $W$ above will automatically satisfy the (periodic) Euler equations as well as the kinematic conditions on the bed and surface. The vortex advection equation~\eqref{intro f vortex advection equation} simply becomes
\begin{subequations}
    \label{periodic wave borne point vortex problem}
    \begin{equation}
        \label{periodic vortex advection}
        \frac{w_{\xi\xi}}{1+w_\eta} = \pi i \kappa \qquad \textrm{at } \zeta = i\beta + 2\lambda \Z,
    \end{equation}
    but need only be imposed at a single vortex by periodicity. However the relationship between $\gamma$, $\kappa$, and $\beta$ will be slightly different from what we had in~\eqref{definition gamma}. Around $\zeta= i \beta$, we find
    \[
        W_\zeta(\zeta) = \frac{\gamma}{2\pi i} \frac{1}{\zeta - i \beta} + 1 + \gamma \frac{2\beta \wzeta(i) - \wzeta(2 \beta i)}{2\pi i} + O(\zeta - i \beta),
    \]
    where the first term in
    \[
        \frac{2\beta \wzeta(i) - \wzeta(2\beta i)}{2\pi i} = \frac{1}{4}\cot(\pi \beta) + \sum_{k=1}^\infty \frac{\sin(2 k \pi \beta)}{e^{2k\pi\lambda}-1}
    \]
    is the same as for the solitary case~\eqref{solitary W_zeta expansion}, while the series is a correction term for finite $\lambda$. This formula follows from the Fourier expansion of $\wzeta$, which can be found, for example, in~\cite[(8.6.19)]{lawden1989elliptic}. Performing an analogous expansion to~\eqref{solitary laurent expansion} reveals that for the periodic case, $\gamma$ must be given by
    \begin{equation}
        \label{definition gamma periodic}
        \gamma = \gamma(\kappa,\beta,\lambda) \ceq 4\parn*{
            \kappa - 4 \frac{2\beta \wzeta(i) - \wzeta(2\beta i)}{2\pi i}
        }^{-1},
    \end{equation}
    interpreted as vanishing at $\beta = 0$.

    Finally, the Bernoulli condition on $\confS$ must hold. The trace of $W_\zeta$ on $\confS$, written as a Fourier series, is
    \[
        a(\xi; \kappa, \beta) \ceq W_\zeta(\xi + i;\kappa, \beta, \lambda) = 1 - \frac{\gamma}{2\lambda}\parn*{\beta + 2 \sum_{k=1}^\infty \frac{\sinh(k \pi \beta / \lambda)}{\sinh(k \pi / \lambda)}\cos(k \pi \xi/\lambda)}
    \]
    by~\cite[\nopp 3.I.1.89]{oberhettinger1973fourier} and~\cite[(8.6.19)]{lawden1989elliptic}; compare to~\eqref{definition a}. Using this definition of $a$, the Bernoulli condition takes the form
    \begin{equation}
        \label{Bernoulli condition periodic}
        \frac{1}{2} \frac{a^2}{w_\xi^2 + (1+w_\eta)^2} + \frac{1}{F^2} w = \frac{1}{2} + Q \qquad \text{on } \confS,
    \end{equation}
\end{subequations}
where we, unlike~\eqref{Bernoulli condition}, must allow a nonzero $Q \in \R$. This Bernoulli constant will be part of the solution, but is expected to be small when $\lambda$ is sufficiently large.

\begin{figure}[htb]
    \centering
    \tikzsetnextfilename{pseudo_arc_length_continuation}
    \input{figures/pseudo_arc_length_continuation.tikz}
    \caption{We use pseudo arc-length continuation to numerically continue the bifurcation curve. At each step, a number of previously computed solutions, and their corresponding tangents, are used to generate an initial approximate solution $\tilde{\Theta}_k^M \in \R^{M+4}$ and approximate tangent $\tilde{n}_k^{M}$. A final approximate solution $\Theta_k^M$ is then sought in the affine hyperplane $\tilde{\Theta}_k^{M} + \lspan\brac{\tilde{n}_k^{M}}^\perp$.}
    \label{pseudo arc length continuation figure}
\end{figure}
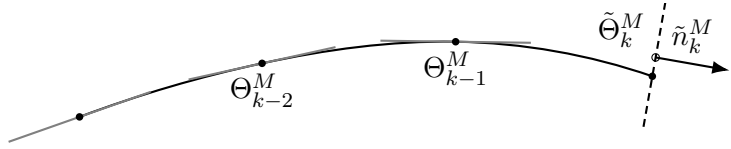

At this point, we could pursue writing down a full periodic analog of \Cref{intro point vortex theorem}, but choose not to do so. Instead, we let $M \in \N$ and look for \emph{approximate} solutions $\Theta^M=(\hat{w}^M,Q^M,\kappa^M,\beta^M)$ to~\eqref{periodic wave borne point vortex problem} with
\[
    w^M(\xi + i) = \sum_{m = 0}^M \hat{w}_m^M \cos\parn*{m \pi \xi / \lambda}
\]
by using a pseudo-spectral method and numerical pseudo arc-length continuation; see \Cref{pseudo arc length continuation figure}.

\begin{figure}[htb]
    \begin{subfigure}[t]{.65\linewidth}
        \centering
        \tikzsetnextfilename{F2_highest.tikz}
        \input{figures/F2_highest.tikz}
        \caption{The surface profile. We conjecture that the curve approaches a limiting highest wave with an interior angle of exactly \SI{120}{\degree}, like in Stokes' conjecture.}
        \label{F2 highest figure}
    \end{subfigure}%
    \begin{subfigure}[t]{.35\linewidth}
        \centering
        \tikzsetnextfilename{F2_beta_kappa_curve.tikz}
        \input{figures/F2_beta_kappa_curve.tikz}
        \caption{The projection of the bifurcation curve into the $(\beta,\kappa)$-plane. We see two turning points.}
        \label{F2 beta kappa curve figure}
    \end{subfigure}%
    \caption{The furthest we are able to resolve numerically when $F=2$, using $\lambda = 5$ and $M = 2^{10}$. The solution on the left corresponds to the small circle on the right.}
    \label{F2 figure}
\end{figure}
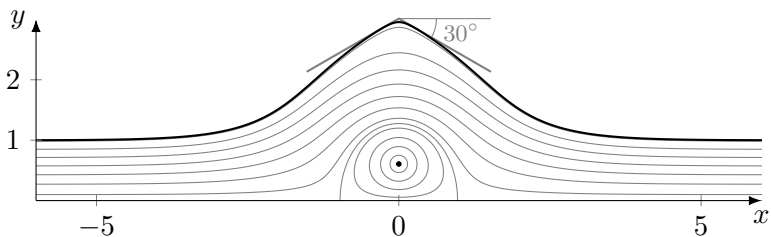
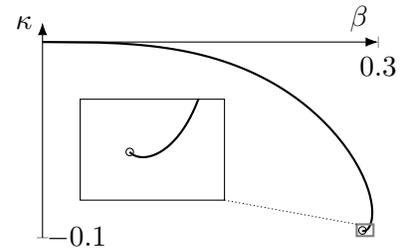

\Cref{F2 figure} shows a typical large solution when the Froude number is small enough for gravity to play a significant role. The surface appears to have an incipient corner at the crest, as is seen in the irrotational case. As the Froude number increases, however, we see from \Cref{F5 F8 figure} that the waves start overturning along the bifurcation curve. In all figures, we include the streamlines passing through equally spaced points between $i \beta$ and $i$ on $\confD_+$, along with the streamline intersecting the bottom that separates the closed streamlines from those that are not. Similar closed streamlines have been observed, for instance, for capillary-gravity waves with point vortices~\cite{varholm2016solitary} and gravity waves with constant vorticity~\cite{wahlen2009critical,kozlov2020stagnation}.

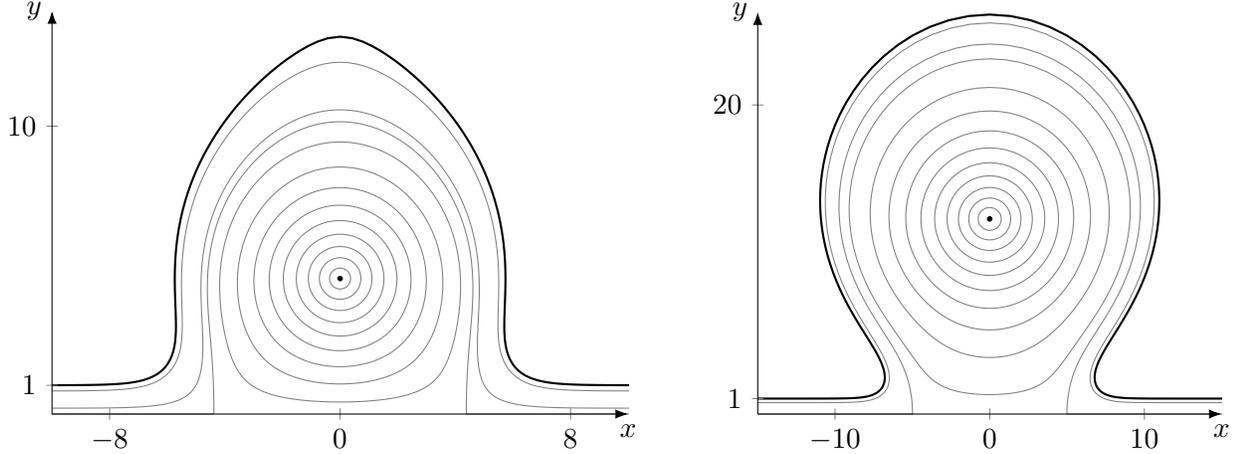
\begin{figure}[htb]
    \centering
    \begin{subfigure}[t]{.535\linewidth}
        \centering
        \tikzsetnextfilename{F5_highest.tikz}
        \input{figures/F5_highest.tikz}
        \caption{When $F = 5$, gravity has decreased sufficiently to allow the wave to barely overturn. Like in \Cref{F2 highest figure}, we display the ``highest'' resolvable solution ($\beta \approx 0.66$). Unlike \Cref{F2 beta kappa curve figure}, we detect no turning point in $(\beta,\kappa)$.}
        \label{F5 highest figure}
    \end{subfigure}%
    \begin{subfigure}[t]{.465\linewidth}
        \centering
        \tikzsetnextfilename{F8.tikz}
        \input{figures/F8.tikz}
        \caption{Moving to $F = 8$ permits waves with appreciable overhang. For this solution, $\beta \approx 0.88$. Going further requires even more points for adequate resolution at the top of the ``bulb''.}
        \label{F8 figure}
    \end{subfigure}%
    \caption{Two different numerical solutions with $M = 2^{11}$; and $\lambda = 8$ and $\lambda = 10$, respectively. Note the differing length scales.}
    \label{F5 F8 figure}
\end{figure}

\subsection{The zero-gravity limit}

In the zero-gravity limit $F = \infty$, an \emph{exact} explicit solitary\footnote{The exact zero-gravity periodic waves are also given in the same article, but the expressions for them are much more involved.} wave is available, due to~\cite{crowdy2023exact}. This is particularly useful as a sanity-check of the numerical method, at least in this, admittedly special, case. The below result translates the solution to our variables. We omit the proof, which is relatively straightforward.

\begin{proposition}[Exact zero-gravity solitary wave]
    \label{exact zero gravity wave proposition}
    When $F = \infty$, the global curve from \Cref{intro point vortex theorem} can be parameterized as $\beta \mapsto (f(\placeholder;\beta),\kappa(\beta),\beta)$ with
    \[
        f(\zeta;\beta) = \zeta + \frac{8}{\pi}\tan^2\parn*{\frac{\pi}{2}\beta} \frac{\sinh\parn*{\frac{\pi}{2}\zeta}}{\cos\parn*{\frac{\pi}{2}\beta}+\cosh\parn*{\frac{\pi}{2}\zeta}} \quad \text{and} \quad \kappa(\beta) = -\frac{1}{2}\sin^3\parn*{\frac{\pi}{2}\beta}\sec\parn*{\frac{\pi}{2}\beta}
    \]
    for $\beta \in (-1,1)$. Their corresponding point vortex strength and physical altitude are
    \[
        \gamma(\beta) = -8\sin\parn*{\frac{\pi}{2}\beta}\sec^3\parn*{\frac{\pi}{2}\beta} \quad \text{and} \quad b(\beta) = \beta + \frac{4}{\pi} \tan^3\parn*{\frac{\pi}{2}\beta},
    \]
    respectively. Overturning along the curve happens exactly at $\beta = \arccos(1-\sqrt{2})/\pi$

\end{proposition}

It is easy to check that the leading order terms in \cref{leading-order part} of \cref{small-amplitude theorem}, in the limit $F \to \infty$, are consistent with the expressions from \cref{exact zero gravity wave proposition}.

From \Cref{numerical convergence figure}, we see that numerical method finds approximate solutions that are in agreement with the explicit solutions from \Cref{exact zero gravity wave proposition}. It is able to capture the surface profile very well, even for quite large waves. Another natural question that comes to mind, and which is hinted at in~\cite{crowdy2023exact}, is: Is the overhanging part of the zero-gravity surface profile asymptotically circular? We are able to answer this in the affirmative.

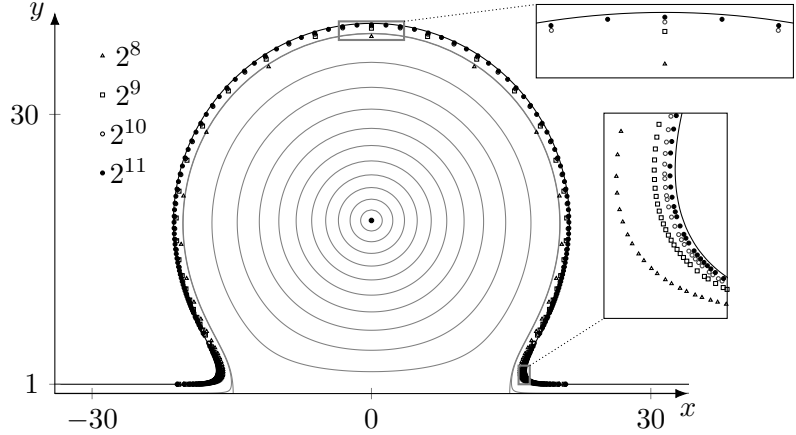
\begin{figure}[htb]
    \centering
    \tikzsetnextfilename{numerical_convergence}
    \input{figures/numerical_convergence.tikz}
    \caption{Comparison between numerical periodic surface profiles for different values of $M$, and the corresponding \emph{exact} solitary solution (solid line), when $F=\infty$, $\beta=0.75$ and $\lambda = 20$. To reduce crowding, only a subset of computed points are included (\cite{douglas1973algorithms} with $\varepsilon=10^{-3}$). The numerical streamlines and vortex position are displayed as computed for $M=2^{11}$, and the physical period is approximately $ 2 \cdot 35$.}
    \label{numerical convergence figure}
\end{figure}

\begin{theorem}[Zero-gravity asymptotics]
    If $f(i)/i \eqc 1 + 2h(\beta)$, the surface profile from \Cref{exact zero gravity wave proposition} normalized as $\tilde{\fluidS}\ceq h(\beta)^{-1}\fluidS$ approaches a circle resting on a line as $\beta \to 1$. Specifically,

    \[
        \dist\parn[\big]{\tilde{\fluidS},\partial B_1(i) \cup \R} \to 0
    \]
    as $\beta \to 1$.
\end{theorem}
\begin{proof}
    By symmetry, it is sufficient to consider the right half of the waves. From \Cref{exact zero gravity wave proposition} we have
    \[
        h(\beta)^{-1}\rest{(f(\zeta) - i)}_{\zeta = \xi + i}= X(\xi,\beta) + iY(\xi,\beta)
    \]
    with
    \[
        Y(\xi,\beta) \ceq 2\cos^2\parn*{\frac{\pi}{2} \beta} \frac{\cosh(\frac{\pi}{2}\xi)}{\cos^2(\frac{\pi}{2}\beta) +\sinh^2(\frac{\pi}{2}\xi)} \quad \text{and} \quad X(\xi,\beta) \ceq h(\beta)^{-1}\xi + \frac{\sinh\parn*{\frac{\pi}{2}\xi}}{\cos(\frac{\pi}{2}\beta)}Y(\xi,\beta),
    \]
    for all $\beta \in (0,1)$ and $\xi \geq 0$. Observe that $Y(\placeholder,\beta)$ is decreasing, as expected from \Cref{nodal property corollary}, and has image $(0,2]$ for every $\beta \in (0,1)$. A small computation shows that the inverse $\xi_0(Y_0,\beta) \ceq Y(\placeholder,\beta)^{-1}(Y_0)$ satisfies
    \[
        \sinh\parn*{\frac{\pi}{2} \xi_0(Y_0,\beta)} = \frac{\cos(\frac{\pi}{2}\beta)}{Y_0} \hat{X}(Y_0,\beta)
    \]
    with
    \[
        \hat{X}_0(Y_0,\beta) \ceq \sqrt{2\parn[\Bigg]{\sqrt{\cos^4\parn[\Big]{\frac{\pi}{2}\beta} + \sin^2\parn[\Big]{\frac{\pi}{2}\beta} Y_0^2}+\cos^2\parn[\Big]{\frac{\pi}{2}\beta}} - Y_0^2}
    \]
    for every $\beta \in (0,1)$ and $Y_0 \in (0,2]$.

    The corresponding real part
    \[
        \tilde{X}_0(Y_0,\beta) \ceq X\parn[\big]{\xi_0(Y_0,\beta),\beta} = h(\beta)^{-1}\xi_0(Y_0,\beta) + \hat{X}_0(Y_0,\beta)
    \]
    converges to
    \[
        X_0(Y_0) \ceq \hat{X}_0(Y_0,1) = \sqrt{2Y_0 - Y_0^2}
    \]
    as $\beta \to 1$, locally uniformly in $Y_0$. Notice that $Y_0 \mapsto X_0(Y_0) + iY_0$ parameterizes the right-half of $\partial_1B(i)$, excluding the origin. The remaining ingredient is that the aforementioned monotonicity implies that
    \[
        \brac{X(\xi,\beta) + i Y(\xi,\beta) : \xi > \xi_0(Y_0,\beta)} \subset (0,\infty) \times (0,Y_0)
    \]
    for every $\beta \in (0,1)$ and $Y_0 \in (0,2]$.
\end{proof}

\section{Solitary wave-borne hollow vortices}
\label{hollow vortex section}

In this section, we prove the existence of solitary gravity water waves with submerged hollow vortices using a vortex desingularization technique and the wave-borne point vortices constructed in \cref{local bifurcation section,global point vortex section} above. The technique is adapted from the recent work~\cite{chen2023desingularization}, which carries out a vortex desingularization procedure for planar point vortex configurations. For the water wave case, the presence of the upper free surface presents many new challenges.

\subsection{Nonlocal formulation}
\label{hollow vortex nonlocal formulation section}

Recalling \cref{hollow vortex intro section}, let us begin by assuming that we have a solitary gravity wave carrying a counter-rotating vortex pair centered at $z=\pm i b$ in the physical domain $\fluidD$. We will again view this as the image of a canonical domain under an a priori unknown conformal mapping $f$. For hollow vortices, the natural choice is to take
\[
    \confD_\rho \ceq
    \begin{dcases}
        \confD \setminus \brac{ i\beta, -i\beta}
         & \rho = 0             \\
        \confD \setminus \ol{B_{\abs{\rho}}(i\beta) \cup B_{\abs{\rho}}(-i\beta)}
         & 0 < \abs \rho \ll 1,
    \end{dcases}
\]
as illustrated in~\cref{hollow vortex configuration figure}. As before, we are doubling the domain, with the line $\brac{\eta = 0}$ corresponding to the bed and $\confSs \ceq \brac{\eta = 1}$ being the pre-image of the free surface. We call $\rho$ the conformal radius of the hollow vortex; in the desingularization procedure, it will serve as a bifurcation parameter. The boundary of the vortex will be given by $f(\confSv)$, where $\confSv = \partial B_{\abs\rho}(i\beta)$. Because we will use the implicit function theorem at $\rho = 0$, the abstract operator equation we ultimately formulate must allow $\rho$ to be negative.

\subsubsection*{Conformal mapping and layer-potentials}
We impose the ansatz
\begin{equation}
    \label{f ansatz}
    f = f^0+ \surf{f} + \rho^2 \mc{Z}^\rho[\mu],
\end{equation}
where $f^0$ is the conformal mapping for the given wave-borne point vortex. The new unknowns are now $\surf{f} \in C^{3+\alpha}(\ol{\confD}; \C)$, which is holomorphic in the strip, and a real-valued density $\mu \in C^{3+\alpha}(\T)$. Here, $\mc{Z}^\rho$ is the layer-potential operator given by
\begin{equation}
    \label{wave Z formula}
    \mc{Z}^\rho[\mu](\zeta) \ceq \frac{1}{2\pi i} \int_{\T} \parn*{ \frac{\mu(\sigma)}{\rho \sigma + i \beta - \zeta} + \frac{\mu(\ol{\sigma})}{\rho \sigma - i \beta - \zeta} } \rho \dee\sigma.
\end{equation}
By Privalov's theorem, for $\abs{\rho} > 0$ and $\ell \geq 0$, $\mc{Z}^\rho$ is bounded $C^{\ell+\alpha}(\T) \to C^{\ell+\alpha}(\ol{\confD_\rho})$. Moreover, for $\mu$ as above, $\mc{Z}^\rho[\mu]$ is a single-valued holomorphic mapping vanishing in the limit $\xi \to \pm\infty$. When there is no risk of confusion, we will abbreviate it by $\mc{Z} \mu$.

As in~\eqref{symmetry f}, the conformal mapping $f$ must be real on real and imaginary on imaginary, and so we require that $\surf{f}$ and the layer potential $\mc{Z}\mu$ each share these symmetries. Thus $\surf{w} \ceq \im{\surf{f}}$ will lie in the same space $\W$ defined in~\eqref{definition W space} and used in constructing wave-borne point vortices. Because the conformal vortex center lies on $i\R \setminus \brac{0}$, we know by~\cite[Section 3.4]{chen2023desingularization} that $\mc{Z}\mu$ has the desired symmetries provided the density $\mu$ lies in the subspace
\begin{align*}
    C_\imim^{3+\alpha}(\T) & \ceq \brac[\big]{\varphi \in C^{3+\alpha}(\T) : i^m \wh{\varphi}_m \in i\R \text{ for all } m \in \Z }. \\
    \intertext{In fact, we will make the further normalizing assumption that $\wh{\mu}_0 = 0$, hence $\mu \in \mo{C}_\imim^{3+\alpha}(\T)$. We also define}
    C_\reim^{3+\alpha}(\T) & \ceq \brac[\big]{ \varphi \in C^{3+\alpha}(\T) : i^m \wh{\varphi}_m \in \R \text{ for all } m \in \Z }.
\end{align*}

Let
\begin{alignat*}{-1}
    \surf{\mc{Z}}^\rho[\mu](\xi)  & \ceq \mc{Z}^\rho[\mu](\xi + i)             & \qquad & \text{for } \xi \in \R   \\
    \vort{\mc{Z}}^\rho[\mu](\tau) & \ceq \mc{Z}^\rho[\mu](\rho \tau + i \beta) &        & \text{for } \tau \in \T,
\end{alignat*}
be the operators found by taking the restriction of the layer-potential $\mc{Z}^\rho \mu$ to $\confSs$ and $\confSv$, respectively. The first of these is nonsingular and higher-order as long as $\confSs \cap \confSv = \varnothing$. Indeed, from the definition of the layer-potential operator~\eqref{wave Z formula}, we see that
\begin{align*}
    \norm{ \surf{\mc{Z}}^\rho \mu}_{C^0(\R)}                                                                                 & \lesssim \abs{\rho} \norm{\mu}_{C^0(\T)}, \\
    \norm{ \partial_\zeta \mc{Z}^\rho \mu }_{C^0(\confSs)} = \frac{1}{\abs{\rho}} \norm{ \surf{\mc{Z}}^\rho \mu' }_{C^0(\R)} & \lesssim \norm{ \mu' }_{C^0(\T)},
\end{align*}
where the constants above depend only on a lower bound for $\dist(\confSs, \confSv)$. We have explicitly, moreover, that
\begin{equation}
    \label{Z1 limit}
    \begin{aligned}
        \frac{1}{\rho} \surf{\mc{Z}}^\rho[\mu](\xi)
         & = \frac{1}{2\pi i} \int_{\T} \parn*{ \frac{\mu(\sigma)}{\rho \sigma + i \beta - i -\xi} + \frac{\mu(\ol{\sigma})}{\rho \sigma - i \beta - i -\xi} } \dee\sigma \\
         & \to \frac{1}{2\pi i} \int_{\T} \parn*{ \frac{\mu(\sigma)}{i \beta - i -\xi} + \frac{\mu(\ol{\sigma})}{- i \beta - i -\xi} } \dee\sigma                         \\
         & = \frac{\ol{\wh{\mu}_1}}{i\beta - i -\xi} - \frac{\wh{\mu}_1}{i\beta+ i + \xi} \qquad \text{as } \rho \searrow 0.
    \end{aligned}
\end{equation}
Conversely, via the Sokhotski--Plemelj formula, the trace of the layer-potential on $\confSv$ becomes
\begin{equation}
    \label{Z2 limit}
    \begin{aligned}
        \vort{\mc{Z}}^\rho[\mu](\tau)
         & = \frac{1}{2\pi i} \int_{\T} \frac{\mu(\sigma) - \mu(\tau)}{\sigma-\tau} \dee\sigma + \frac{1}{2\pi i} \int_{\T} \frac{\mu(\ol{\sigma})}{\rho (\sigma-\tau) - 2i \beta} \rho \dee\sigma \\
         & = \mc{C}\mu(\tau) - \rho \frac{\wh{\mu}_1}{2 i \beta} + O(\rho^2),
    \end{aligned}
\end{equation}
where $\mc{C}$ is the Cauchy-type integral operator represented by the first term on the right-hand side of the first equality. Indeed, it is easy to verify that for all $\ell \geq 0$ and $0 < \rho_0 \ll 1$, the mapping $(\mu, \rho) \mapsto \vort{\mc{Z}}^\rho \mu$ is real analytic $C^{\ell+\alpha}(\T) \times (-\rho_0, \rho_0) \to C^{\ell+\alpha}(\T)$, while $(\mu, \rho) \mapsto \rho^{-1}\surf{\mc{Z}}^\rho \mu$ is real analytic $C^{\ell+\alpha}(\T) \times (-\rho_0, \rho_0) \to C_0^{\ell+\alpha}(\R)$.

\subsubsection*{Complex relative potential}

We redefine $W = W(\zeta)$ to be the complex relative velocity potential for the wave-borne hollow vortex problem, which thus has domain $\confD_\rho$, and denote the complex relative potential~\eqref{definition complex potential} for the point vortex problem by
\[
    W_0 = W_0(\zeta;\gamma) \ceq \frac{\gamma}{2\pi i} \log\parn[\bigg]{ \frac{\sinh(\frac{\pi}{2} (\zeta-i\beta))}{\sinh(\frac{\pi}{2}(\zeta+i\beta))} } + \zeta.
\]
Note that here $W_0$ is viewed as a function of $\gamma$. When studying wave-borne point vortices, we required that $\gamma$, $\kappa$, and $\beta$ were related by~\eqref{definition gamma}. However, for hollow vortices $\kappa$ does not have any obvious relevance, so we will untether the variables, fixing $\beta$ to $\beta^0$, and allowing $\gamma$ to vary. Thus $\kappa$ is simply a function of $\gamma$.

Throughout this section, we use the convention that a superscript $0$ indicates the function is being evaluated with all parameters at their values for the starting point vortex configuration. For example, $W^0 \ceq W_0(\placeholder; \gamma^0)$. We expect that $W^0$ is the leading-order part of the complex potential, but in order to understand how the Bernoulli condition behaves in the limit $\rho \searrow 0$, it will be necessary to have a more precise description. That is the objective of the next lemma.

\begin{lemma}[Complex potential]
    \label{complex potential lemma}
    Fix $\gamma^0 \in \R$ and $\beta \in (0,1)$. Then there exist analytic mappings
    \[
        (\gamma,\rho) \mapsto \nu(\placeholder;\gamma,\rho) \in \mo{C}^{k+3+\alpha}_\imim(\T),
        \qquad
        (\gamma,\rho) \mapsto \surf W(\placeholder;\gamma,\rho) \in C^{k+3+\alpha}_0(\ol\confD;\C),
    \]
    defined in a neighborhood of $(\gamma^0,0)$ in $\R^3$ such that the complex potential
    \begin{equation}
        \label{potential ansatz}
        W = W_0 + \rho\mc{Z}^\rho \nu + \rho^2 \surf{W},
    \end{equation}
    satisfies, for $\rho > 0$,
    \begin{enumerate}[label=\textup{(\roman*)}]
        \item $W_\zeta$ is holomorphic on $\confD_\rho$ with $W_\zeta \in C^{k+2+\alpha}(\ol{\confD_\rho})$;
        \item $W-W^0$ is single-valued and $\im{W}$ is constant on each component of $\partial\confD_\rho$;
        \item $W_\zeta(\zeta) \to 1$ as $\zeta \to \infty$ in $\confD_\rho$; and
        \item \label{potential symmetry} $W_\zeta$ real on both the real and imaginary axes.
    \end{enumerate}
    Moreover, at $\rho = 0$ the density $\nu$ and its derivative $\nu_\rho$ are given explicitly by
    \begin{align}
        \label{definition nu0}
        \nu(\tau;0,\gamma)              & = \nu_0(\tau;\gamma)
        \ceq -2\re\parn*{ \parn*{ 1 + \frac{\gamma}{4} \cot(\pi \beta) } \tau } \\
        \label{definition nu1}
        \partial_\rho\nu(\tau;0,\gamma) & = \nu_1(\tau;\gamma)
        \ceq -\re\parn*{ \frac{\gamma \pi }{24i} \parn*{ 1 - 3 \csc^2(\pi \beta) } \tau^2 }.
    \end{align}
\end{lemma}
\begin{proof}
    We will reduce the problem to a nonlocal equation for $\nu$ and then apply the implicit function theorem. For $\nu \in \mo{C}^{k+3+\alpha}_\imim(\T)$, let $\mc{L}^\rho \nu \in C^{k+3+\alpha}_0(\confD;\C)$ denote the unique holomorphic function satisfying the boundary condition
    \[
        \im{\mc{L}^\rho \nu(\xi+i)} = -\im{\frac{1}{2\pi i} \int_{\T} \parn*{ \frac{\nu(\sigma)}{ \rho \sigma + i \beta - i - \xi} + \frac{\nu(\ol{\sigma})}{\rho \sigma - i \beta - i - \xi} } \dee\sigma }
        \qquad \text{for } \xi \in \R
    \]
    on the upper boundary $\confSs$ and enjoying the even symmetry claimed in~\cref{potential symmetry}. For $\rho \ne 0$, the right had side is simply $-\rho^{-1}\im{ \surf{\mc{Z}}\nu}$; see \eqref{Z1 limit}. It is easy to check that $\mc{L}^\rho$ is a bounded operator between these two spaces, depending analytically on $\rho$ in a neighborhood of zero. Setting
    \[
        \surf W = \mc{L}^\rho \nu
    \]
    in~\eqref{potential ansatz}, and using the fact that $\im W_0$ is constant on $\confSs$, we discover that
    \[
        \im W = \im W_0 + \rho\im\parn*{ \surf{\mc{Z}}^\rho \nu + \rho \mc{L}^\rho \nu }
    \]
    is also constant on $\confSs$. Differentiating the requirement that $\im W$ be constant along $\confSv$ then yields a nonlocal equation
    \begin{equation}
        \label{potential integral equation}
        \mc{F}(\nu;\gamma,\rho)(\tau) \ceq \re\parn*{ \tau \parn*{ \partial_\zeta W_0(i\beta + \rho \tau) + \vort{\mc{Z}}^\rho \nu' + \rho^2 \partial_\zeta\mc{L}^\rho \nu(i\beta + \rho \tau) } } = 0 \qquad \text{for } \tau \in \T
    \end{equation}
    for $\nu$ alone. While the explicit term $\partial_\zeta W_0$ is singular as $\rho \to 0$, using~\eqref{solitary W_zeta expansion} we find that the combination
    \[
        \re\parn*{\tau\partial_\zeta W_0(i\beta + \rho\tau)}
        = \re\parn*{\parn*{1 + \frac{\gamma}{4} \cot(\pi \beta) } \tau + \rho \frac{\gamma \pi}{24 i} \parn*{ 1 - 3 \csc^2(\pi \beta) } \tau^2 }+ O(\rho^2)
    \]
    is analytic in $\rho$, and so the mapping $\mc{F}$ is also analytic, with values in $\mo C^{k+2+\alpha}_\imim(\T)$. Here we have used the analyticity of the operators $\vort{\mc{Z}}^\rho$ discussed after \eqref{Z2 limit}, as well as some straightforward symmetry considerations. In particular, $\mc{F}$ is well-defined at $\rho=0$, where it is given by
    \begin{equation}
        \label{potential integral equation rho=0}
        \mc{F}(\nu;\gamma,0)(\tau) = \re\parn*{ \tau
            \parn*{\parn*{1 + \frac{\gamma}{4} \cot(\pi \beta) }
                + \mc C \nu'}}.
    \end{equation}
    From~\eqref{potential integral equation rho=0} it is easy to see that $\nu_0$ defined in~\eqref{definition nu0} solves~\eqref{potential integral equation} when $\rho=0$. Moreover, the linearized operator
    \[
        \mc{F}_\nu(\nu_0(\gamma^0);\gamma^0,0) = \re ( \tau \mc{C} \partial_\tau \placeholder ) \colon \mo{C}^{k+3+\alpha}_\imim(\T) \to \mo{C}^{k+2+\alpha}_\imim(\T)
    \]
    is invertible, and so we can then solve~\eqref{potential integral equation} for $\nu = \nu(\placeholder;\rho,\gamma)$ in a neighborhood of $\nu_0$ using the implicit function theorem. The formula~\eqref{definition nu1} for $\nu_1$ can be obtained in a similar way. The remaining claims are now straightforward to verify.
\end{proof}

\subsubsection*{Governing equations}

Let us next reformulate the wave-borne hollow vortex problem~\eqref{intro hollow vortex problem} in terms of these new unknowns. Pulled back to the conformal domain, the dynamic boundary conditions \eqref{intro dynamic hollow surface}, \eqref{intro dynamic hollow vortex} take the form
\begin{alignat*}{-1}
    \frac{1}{2} \frac{\abs{W_\zeta}^2}{\abs{f_\zeta}^2} + \frac{1}{F^2} \im{f} & = \frac{1}{2} + \frac{1}{F^2} & \qquad & \text{on } \confSs, \\
    \frac{1}{2} \frac{\abs{W_\zeta}^2}{\abs{f_\zeta}^2} + \frac{1}{F^2} \im{f} & = q                           &        & \text{on } \confSv,
\end{alignat*}
where we recall that $\confSs \ceq \brac{ \eta = 1 }$ is the pre-image of the upper boundary, $\confSv \ceq \partial B_\rho(i\beta)$, and $q \in \R$ is the Bernoulli constant. The complex potential $W$ is given by \cref{complex potential lemma}, and is completely determined by the parameters $\rho,\gamma,\beta$. Note that we have used the limiting behavior at infinity to explicitly evaluate the Bernoulli constant on $\confSs$. We can simplify the condition on $\confSs$ somewhat further by noting that the starting solitary wave with submerged point vortex itself satisfies
\[
    \frac{1}{2} \frac{\abs{W_\zeta^0}^2}{\abs{f_\zeta^0}^2} + \frac{1}{F^2} \im{f^0} = \frac{1}{2} + \frac{1}{F^2} \qquad \text{on } \confSs.
\]
Thus, inserting the ansatz for $f$ from~\eqref{f ansatz}, the dynamic condition on the upper free surface becomes
\[
    \frac{1}{2} \parn*{ \frac{\abs{W_\zeta}^2}{\abs{f_\zeta}^2} -\frac{\abs{W_\zeta^0}^2}{\abs{f_\zeta^0}^2} } + \frac{1}{F^2} \im(\surf{f} + \rho^2 \mc{Z}^\rho\mu) = 0 \qquad \text{on } \confSs.
\]

On the hollow vortex boundary, however, we must contend with the unboundedness of the point vortex velocity field in any neighborhood of $\zeta = i \beta$. We therefore normalize the condition on $\confSv$ by multiplying by $\rho^2$ to ameliorate the singularity. Regrouping terms, this results in
\[
    \frac{1}{2} \frac{\rho^2 \abs{W_\zeta}^2}{\abs{f_\zeta}^2} + \rho^2 \frac{1}{F^2} \im\parn*{ f^0 + \surf{f} + \rho^2 \mc{Z}^\rho \mu } = \tilde{Q} \qquad \text{on } \confSv,
\]
where we think of $\tilde Q \in \R$ as a normalized Bernoulli constant on $\confSv$. We will make one final renormalization in the next subsection.

Stated in terms of the layer-potential operator $\mc{Z}$, the dynamic condition on the upper interface and vortex boundary take the form
\begin{alignat}{-1}
    \label{Bernoulli surface}
    \frac{1}{2} \parn*{ \frac{\surf{a}^2}{\abs{f_\zeta^0+ \partial_\zeta \surf{f} + \rho \surf{\mc{Z}}^\rho \mu' }^2} -\frac{(\surf{a}^0)^2}{\abs{f_\zeta^0}^2} } + \frac{1}{F^2} \im(\surf{f} + \rho^2 \surf{\mc{Z}}^\rho\mu) & = 0         & \qquad & \text{on } \R  \\
    \shortintertext{and}
    \label{Bernoulli vortex}
    \frac{1}{2} \frac{ \vort{a}^2 }{\abs{f_\zeta^0+ \partial_\zeta \surf{f} + \rho \vort{\mc{Z}}^\rho \mu'}^2} + \rho^2 \frac{1}{F^2} \im\parn*{ f^0 + \surf{f} + \rho^2 \vort{\mc{Z}}^\rho \mu }                              & = \tilde{Q} &        & \text{on } \T,
\end{alignat}
where $f^0$, $\surf{f}$ are evaluated at $\xi +i$ in~\eqref{Bernoulli surface} and at $i\beta + \rho \tau$ in~\eqref{Bernoulli vortex}.
Here, similarly to~\eqref{definition a},
\begin{alignat*}{-1}
    \surf{a} & = \surf{a}(\xi; \gamma, \rho) \ceq \abs{W_\zeta(\xi + i; \gamma, \rho)}                  & \qquad & \text{for all } \xi \in \R,    \\
    \vort{a} & =\vort{a}(\tau; \gamma, \rho) \ceq \rho \abs{W_\zeta(\rho \tau + i \beta; \gamma, \rho)} &        & \text{for all } \sigma \in \T,
\end{alignat*}
and $\surf{a}^0 = \surf{a}^0(\xi) \ceq \abs{W_\zeta^0(\xi+i)}$.

\subsection{Spaces and the abstract operator equation}

Let us now formulate the wave-borne hollow vortex problem~\eqref{Bernoulli surface}--\eqref{Bernoulli vortex} as an abstract operator equation. For the remainder of this section, we rename the domain and codomain of the abstract operator $\F$ given by~\eqref{definition F} and corresponding to the submerged point vortex problem as $\Open^0 \subset \X^0$ and $\Y^0$, respectively. Then, writing
\[
    u \ceq (\surf{w}, \mu, \gamma, Q),
\]
the wave-borne hollow vortex problem can be equivalently written as
\[
    \G(u,\rho) = 0
\]
for the real-analytic mapping
\[
    \G = (\G_1,\G_2) \colon \Open \subset (\X \times \R^2) \to \Y
\]
with domain and codomain being the (newly redefined) spaces
\[
    \X \ceq \W \times \mo{C}_\imim^{k+3+\alpha}(\T),
    \qquad
    \Y \ceq C_{0,\even}^{k+2+\alpha}(\confSs) \times C_\reim^{k+2+\alpha}(\T),
\]
and given explicitly by
\begin{equation}
    \label{definition G}
    \begin{aligned}
        \G_1(u,\rho) & \ceq \frac{1}{2} \parn[\bigg]{ \frac{\surf{a}^2}{\abs{f_\zeta^0+ \partial_\zeta \surf{f} + \rho \surf{\mc{Z}}^\rho \mu' }^2} -\frac{(\surf{a}^0)^2}{\abs{f_\zeta^0}^2} } + \frac{1}{F^2} \im(\surf{f} + \rho^2 \surf{\mc{Z}}^\rho\mu)         \\
        \G_2(u,\rho) & \ceq \frac{1}{2 \rho} \parn[\bigg]{ \frac{ \vort{a}^2 }{\abs{f_\zeta^0+ \partial_\zeta \surf{f} + \rho \vort{\mc{Z}}^\rho \mu'}^2} - \frac{\gamma^2}{4\pi^2} \frac{1}{\abs{\partial_\zeta f^0(i\beta) + \partial_\zeta \surf{f}(i\beta)}^2} } \\
                     & \qquad + \rho \frac{1}{F^2} \im\parn[\big]{ f^0 + \surf{f} + \rho^2 \vort{\mc{Z}}^\rho \mu } - Q.
    \end{aligned}
\end{equation}
Here, and in what follows, we often write $\surf{f}$ for the holomorphic function on $\confD$ with imaginary part $\surf{w}$. Note that the equation $\G_2 = 0$ is slightly different from the
phrasing of the dynamic condition in~\eqref{Bernoulli vortex} as a term is being subtracted in the parenthesis, and then we are dividing by $\rho$. The new Bernoulli constant $Q$ can be written explicitly in terms of $\tilde Q$, $\rho$, $\gamma$, and $\partial_\zeta \surf f(i\beta)$. We will see that this form of the dynamic condition is considerably better suited for the asymptotic expansions carried out in the next stage of the argument. The lemma below verifies that $\G$ thus defined is indeed real analytic and its zero-set corresponds to solutions of the wave-borne hollow vortex problem.

\begin{lemma}
    \label{G well-defined lemma}
    The mapping $\G$ defined by~\eqref{definition G} is real analytic $\Open \to \Y$, and if $(w^0, \kappa^0, \beta) \in \F^{-1}(0)$ represents a wave-borne point vortex, then $u^0 \ceq (0,0,\gamma^0,0) \in \Open$ satisfies $\G(u^0,0) = 0$.
\end{lemma}
\begin{proof}
    Clearly $\G_1(u^0,0) = 0$. We have already discussed the analyticity of $\vort{\mc{Z}}^\rho,\surf{\mc{Z}}^\rho$ in $\rho$ in the paragraph below \eqref{Z2 limit}, and the analyticity of the coefficients $\surf{a},\vort{a}$ in $(\gamma,\rho)$ follows from \cref{complex potential lemma}. Combining this with some routine symmetry arguments, we conclude that $\G$ is an analytic mapping $\Open \to \Y$ except for a potential singularity in $\G_2$ when $\rho=0$, which we now rule out through a careful expansion.

    Using the asymptotics for $W$ from~\cref{complex potential lemma}, we expand $\vort{a}^2$ to obtain
    \begin{alignat*}{-1}
        \vort{a}^2 & = \rho^2 \abs{W_{0\zeta}(\rho\tau + i\beta) + \vort{\mc{Z}}^\rho \nu'+ \rho^2 \partial_\zeta \surf{W}(\rho\tau+i\beta)}^2                                                                                                                  & \qquad &                                \\
                   & = \abs{\rho W_{0\zeta}(\rho\tau+i\beta)}^2 + 2 \rho \re \parn*{ \ol{\rho W_{0\zeta}(\rho\tau+i\beta)} \vort{\mc{Z}}^\rho \nu' } + \abs{\rho \vort{\mc{Z}}^\rho \nu'}^2 + O(\rho^4)                                                         &        &                                \\
                   & = \frac{\gamma^2}{4\pi^2} - 2 \rho \re \parn*{ \frac{\gamma}{2\pi i} \parn*{ \frac{\gamma \kappa}{4} + \vort{\mc{Z}}^\rho \nu' } \tau }                                                                                                                                              \\
                   & \quad + \rho^2 \brak*{ \frac{\kappa^2\gamma^2}{16} + \abs{\vort{\mc{Z}}^\rho \nu' }^2 + 2 \re \parn[\bigg]{ \frac{\kappa \gamma}{4} \vort{\mc{Z}}^\rho \nu' + \frac{\gamma^2 \tau^2}{48} \parn*{ 1 - 3 \csc^2(\pi \beta) } } } + O(\rho^3) &        &                                \\
                   & = \frac{\gamma^2}{4\pi^2} - 2\rho \re\parn*{ \frac{\gamma}{2\pi i} \parn*{ \frac{\gamma \kappa}{4} + \mc{C} \nu_0' } \tau } + \rho^2 \parn*{ \frac{\kappa^2\gamma^2}{16} + \abs{\mc{C} \nu_0' }^2 }                                                                                  \\
                   & \quad + 2 \rho^2 \re \parn*{ -\frac{\gamma \wh{(\nu_0')}_1 }{4 \pi \beta} - \frac{\gamma}{2\pi i} \tau \mc{C} \nu_1' + \frac{\kappa \gamma}{4} \mc{C} \nu_0' + \frac{\gamma^2 \tau^2}{48} \parn*{ 1 - 3 \csc^2(\pi \beta) } } + O(\rho^3)  &        & \text{in } C^{k+2+\alpha}(\T),
    \end{alignat*}
    where we have used~\eqref{Z2 limit} in deriving the last equality. From the forms of $\nu_0$ and $\nu_1$ given in~\eqref{definition nu0} and~\eqref{definition nu1}, we have explicitly that
    \[
        \tau \mc{C}\nu_0'(\tau) = -\parn*{ 1 + \frac{\gamma}{4} \cot(\pi \beta) } \frac{1}{\tau} = -\frac{\gamma \kappa}{4} \frac{1}{\tau}, \quad \proj_1 \nu_0' = 0, \quad \tau \mc{C} \nu_1' (\tau) = \frac{\gamma \pi}{24 i} \parn*{1 - 3 \csc^2(\pi \beta)} \frac{1}{\tau^2}.
    \]
    Therefore the above expansion for $\vort{a}^2$ simplifies to
    \begin{alignat*}{-1}
        \vort{a}^2 & = \frac{\gamma^2}{4\pi^2} - \rho \re\parn*{ \frac{\gamma^2 \kappa}{2\pi i} \tau } + \rho^2 \frac{\kappa^2\gamma^2}{8}                       & \qquad &                                \\
                   & \quad+ \rho^2 \re \parn*{ - \frac{\kappa^2 \gamma^2 \tau^2}{8} + \frac{\gamma^2 \tau^2}{12} \parn*{ 1 - 3 \csc^2(\pi \beta) } } + O(\rho^3) &        & \text{in } C^{k+2+\alpha}(\T).
    \end{alignat*}

    To ease the presentation, denote $g \ceq f_\zeta^0 + \partial_\zeta \surf{f}$. Then, looking at the parenthetical term in the definition of $\G_2$, and expanding in $\rho$ leads to
    \begin{align*}
        \frac{\vort{a}^2 }{\abs{f_\zeta^0+ \partial_\zeta \surf{f} + \rho \vort{\mc{Z}}^\rho \mu'}^2}
         & =
        \frac{\gamma^2}{4\pi^2 \abs{g(i\beta)}^2} - \rho \brak*{ \frac{ \re \parn*{ \frac{\gamma^2 \kappa}{2 \pi i} \tau} }{\abs{g(i\beta)}^2} + \frac{\gamma^2 \re \parn*{ \ol{g(i\beta)} \parn*{ g_\zeta(i\beta) \tau + \mc{C} \mu' } } }{2 \pi^2 \abs{g(i\beta)}^4} } \\
         & \qquad
        + \rho^2 \frac{ \frac{\kappa^2\gamma^2}{8} + \re \parn*{ - \frac{\kappa^2 \gamma^2 \tau^2}{8} + \frac{\gamma^2 \tau^2}{12} \parn*{ 1 - 3 \csc^2{(\pi \beta)} } } }{\abs{g(i\beta)}^2}                                                                            \\
         & \qquad
        + \rho^2 \frac{ 2 \re\parn*{ \frac{\gamma^2 \kappa}{2 \pi i} \tau } \re \parn*{ \ol{g(i\beta)} ( g_\zeta(i\beta) \tau + \mc{C} \mu' ) } }{ \abs{g(i\beta)}^4 }                                                                                                   \\
         & \qquad
        - \rho^2 \frac{\gamma^2}{4\pi^2} \frac{\abs{ g_\zeta(i\beta) \tau + \mc{C} \mu' }^2 + 2 \re \parn*{ \ol{g(i\beta)} \parn*{ \frac12 g_{\zeta\zeta}(i\beta) \tau^2 - \frac{\wh{(\mu')}_1 }{2 i \beta} } } }{\abs{g(i\beta)}^4}                                     \\
         & \qquad
        + \rho^2 \frac{\gamma^2}{4\pi^2} \frac{ 4 \brak*{ \re \parn*{ \parn*{ \ol{g(i\beta)} \parn*{ g_\zeta(i\beta) \tau + \mc{C} \mu' } } } }^2 }{\abs{g(i\beta)}^6} + \msc{R}                                                                                         \\
         & \eqc \frac{\gamma^2}{4\pi^2 \abs{g(i\beta)}^2} + \rho \mc{B}_1 + \rho^2 \mc{B}_2 + \msc{R},
    \end{align*}
    where the remainder term $\msc{R}$ is a real-analytic mapping $\Open \to C^{k+2+\alpha}(\T)$ that obeys
    \[
        \msc{R} = O\parn*{\rho^3 + \rho^2 \norm{ (\surf{f},\mu) }_{\X} + \rho \norm{ (\surf{f},\mu) }_{\X}^2 } \qquad \text{in } C^{k+2+\alpha}(\T).
    \]
    Now, using the fact that $f^0$ satisfies~\eqref{intro f vortex advection equation}, we expand the $O(\rho)$ terms $\mc{B}_1$ on the right-hand side and then rearrange to find
    \begin{equation}
        \label{parenthetical expansion G2}
        \begin{aligned}
            \frac{1}{2\rho} & \parn*{ \frac{ \vort{a}^2 }{\abs{f_\zeta^0+ \partial_\zeta \surf{f} + \rho \vort{\mc{Z}}^\rho \mu'}^2} - \frac{\gamma^2}{4\pi^2} \frac{1}{\abs{f^0_\zeta(i\beta) + \partial_\zeta \surf{f}(i\beta)}^2} }                                                                                                                                                                                                                      \\
            = \             & \frac{ \abs{ f^0_\zeta(i\beta)}^2 \re\parn*{ \frac{\gamma^2 \kappa^0}{4\pi i} \tau } }{\abs{f^0_\zeta(i\beta) + \partial_\zeta \surf{f}(i\beta)}^4} - \frac{ \re\parn*{ \frac{\gamma^2 \kappa}{4\pi i} \tau } }{\abs{f^0_\zeta(i\beta) + \partial_\zeta \surf{f}(i\beta)}^2} - \frac{\gamma^2}{4 \pi^2} \frac{ \re\parn*{ \ol{f_\zeta^0(i\beta)} \mc{C} \mu' } }{\abs{f^0_\zeta(i\beta) + \partial_\zeta \surf{f}(i\beta)}^4} \\
                            & - \frac{\gamma^2}{4 \pi^2} \frac{ \re\parn*{ \parn*{ \ol{f_\zeta^0(i\beta)} \partial_\zeta^2 \surf{f}(i\beta) + \ol{\partial_\zeta \surf{f}(i\beta)} f^0_{\zeta\zeta}(i\beta) }\tau } }{\abs{f^0_\zeta(i\beta) + \partial_\zeta \surf{f}(i\beta)}^4} + \rho \frac{\mc{B}_2}{2} + \msc{Q},
        \end{aligned}
    \end{equation}
    with $\msc{Q}$ being a new real-analytic mapping $\Open \to C^{k+2+\alpha}(\T)$ satisfying the quadratic bound
    \[
        \msc{Q} = O\parn*{\rho^2 + \norm{ (\surf{f},\mu) }_{\X}^2 } \qquad \text{in } C^{k+2+\alpha}(\T).
    \]
    The right-hand side of~\eqref{parenthetical expansion G2} therefore constitutes a real-analytic mapping $\Open \to C^{k+2+\alpha}(\T)$ that vanishes when $(u,\rho) = (u^0,0)$, which in turn implies the same is true of $\G_2$. The proof of the lemma is complete.
\end{proof}

\subsection{Proof of \texorpdfstring{\cref{intro hollow vortex theorem}}{Theorem~\ref{intro hollow vortex theorem}}}

We are now ready to prove our main result on the existence of wave-borne hollow vortices. The majority of the work is done in the next lemma, which establishes that the invertibility of the linearized wave-borne hollow vortex system at $\rho = 0$ is equivalent to the invertibility of the corresponding linearized wave-borne point vortex system.

\begin{lemma}[Non-degeneracy]
    \label{hollow vortex nondegeneracy lemma}
    Suppose that $(w^0,\kappa^0, \beta) \in \Open^0$ represents a wave-borne point vortex and let $u^0 = (0,0,\gamma^0,0) \in \Open$. If $D_{(w,\kappa)} \F(w^0,\kappa^0,\beta)$ is an isomorphism $\X^0 \to \Y^0$, then the linearized operator $D_u \G(u^0,0)$ for the wave-borne hollow vortex problem is likewise an isomorphism $\X \to \Y$.
\end{lemma}
\begin{proof}
    First consider the Bernoulli condition on the upper free boundary. By~\cref{complex potential lemma}, we have the expansion
    \begin{align*}
        \surf{a}^2 & = \abs{W_{0\zeta}(\xi + i) + \surf{\mc{Z}}^\rho \nu' + \rho^2 \partial_\zeta \surf{W}(\xi+i)}^2                                                                                     \\
                   & = (\surf{a}^0)^2 + 2 (\gamma-\gamma^0) \surf{a}^0 \partial_\gamma \surf{a}^0 + 2 \re \parn*{ \ol{W_\zeta^0(\xi+i)} \surf{\mc{Z}}^\rho \nu' } + O\parn*{\rho^2+(\gamma-\gamma^0)^2}.
    \end{align*}
    in $C^{k+2+\alpha}(\R)$. Here, it is important to recall that $\surf{\mc{Z}}^\rho \nu = O(\rho)$. Expanding the parenthetical quantity in~\eqref{Bernoulli surface}, we therefore obtain
    \begin{equation}
        \label{parenthetical expansion G1}
        \begin{aligned}
            \frac{1}{2} \parn*{ \frac{ a\surf{a}^2 }{\abs{f_\zeta^0+ \partial_\zeta \surf{f} + \rho \surf{\mc{Z}}^\rho \mu'}^2} -\frac{(\surf{a}^0)^2}{\abs{f_\zeta^0}^2} } & = - \frac{ \re \parn*{ (\surf{a}^0)^2 \ol{f_\zeta^0} \partial_\zeta \surf{f} } } {\abs{f_\zeta^0}^4} + (\gamma-\gamma^0) \frac{\surf{a}^0 \partial_\gamma \surf{a}^0}{\abs{f_\zeta^0}^2} \\
                                                                                                                                                                            & \qquad - \frac{\re\parn*{ \ol{W_\zeta^0}\surf{\mc{Z}}^\rho (\nu^0)'}}{\abs{f_\zeta^0}^2} + \msc{R},
        \end{aligned}
    \end{equation}
    with remainder term $\msc{R}$ that is real analytic $\Open \to \Y_1$ and satisfies
    \[
        \msc{R} = O\parn*{\rho^2 + \norm{ \surf{w} }_{\W}^2 + (\gamma-\gamma^0)^2 } \qquad \text{in } C^{k+2+\alpha}(\R).
    \]
    But, from~\eqref{definition nu0} we see that $\proj_1 \nu_0' = 0$, and hence by~\eqref{Z1 limit} we have that
    \begin{equation}
        \label{DuG1 formula}
        D_u \G_1(u^0,0)
        \begin{pmatrix}
            \surf{\dot{w}} \\
            \dot\mu        \\
            \dot\gamma     \\
            \dot Q
        \end{pmatrix}
        = - \frac{\re \parn*{ (\surf{a}^0)^2 \ol{f_\zeta^0} \partial_\zeta \surf{\dot{f}} }}{\abs{f_\zeta^0}^4} + \dot\kappa \frac{\surf{a}^0 \partial_\kappa \surf{a}^0}{\abs{f_\zeta^0}^2}+ \frac{1}{F^2} \surf{\dot{w}}
        =
        D_{(w,\kappa)} \F_1(w^0,\gamma^0,\beta)
        \begin{pmatrix}
            \surf{\dot{w}} \\
            \dot\kappa
        \end{pmatrix}
    \end{equation}
    where we are writing $\surf{\dot{f}}$ to denote the corresponding holomorphic function on $\confD$ with imaginary part $\surf{\dot{w}}$, and
    \begin{equation}
        \label{kappa dot to gamma dot}
        \dot\kappa \ceq \rest{(\partial_\gamma\kappa)}_{\gamma^0} \dot\gamma = -\frac{(\kappa^0- \cot(\pi \beta))^2}{4} \dot\gamma
    \end{equation}
    represents the variation of $\kappa$ at $\gamma^0$ by $\dot\gamma$.

    Turning next to the Bernoulli condition on the vortex boundary, and recalling the asymptotic expansion~\eqref{parenthetical expansion G2}, we infer that the linearized operator is
    \begin{align*}
        D_u \G_2(u^0,0)
        \begin{pmatrix}
            \surf{\dot{w}} \\
            \dot\mu        \\
            \dot\gamma     \\
            \dot Q
        \end{pmatrix}
        = & - \frac{2 \re \parn*{ \frac{(\gamma^0)^2 \kappa^0}{2 \pi i} \tau } \re\parn*{ \ol{f^0_\zeta(i\beta)} \partial_\zeta \surf{\dot{f}}(i\beta) }}{\abs{f^0_\zeta(i\beta)}^4} - \frac{\re \parn*{ \frac{(\gamma^0)^2 \dot\kappa}{2\pi i} \tau }}{\abs{f^0_\zeta(i\beta)}^2} \\
          & - \frac{(\gamma^0)^2}{2 \pi^2} \frac{ \re\parn*{ \parn*{ \ol{f_\zeta^0(i\beta)} \partial_\zeta^2 \surf{\dot{f}}(i\beta) + \ol{\partial_\zeta \surf{\dot{f}}(i\beta)} f^0_{\zeta\zeta}(i\beta) }\tau } }{\abs{f^0_\zeta(i\beta)}^4}                                     \\
          & - \frac{(\gamma^0)^2}{2 \pi^2} \frac{ \re\parn*{ \ol{f_\zeta^0(i\beta)} \mc{C} \dot\mu' } }{\abs{f^0_\zeta(i\beta)}^4} -\dot Q.
    \end{align*}
    Note first that as a consequence of the symmetry assumptions, both $\partial_\zeta \surf{\dot{f}}$ and $f^0_\zeta$ are real on imaginary while $\partial_\zeta^2 \surf{\dot{f}}$ and $f^0_{\zeta\zeta}$ are imaginary on imaginary. Using this and the point vortex condition satisfied by $f^0$, we can simplify the above expression to
    \begin{equation}
        \label{DuG2 formula}
        \begin{aligned}
            D_u \G_2(u^0,0)
            \begin{pmatrix}
                \surf{\dot{w}} \\
                \dot\mu        \\
                \dot\gamma     \\
                \dot Q
            \end{pmatrix}
            = & -\frac{(\gamma^0)^2}{2 \pi^2} \frac{1}{f_\zeta^0(i\beta)^2} \re \parn*{ \parn*{ \frac{ {f_\zeta^0(i\beta)} \partial_\zeta^2 \surf{\dot{f}}(i\beta) - f^0_{\zeta\zeta}(i\beta) {\partial_\zeta \surf{\dot{f}}(i\beta)} }{f^0_\zeta(i\beta)^2} - \pi i \dot\kappa } \tau } \\
              & - \frac{(\gamma^0)^2}{2 \pi^2} \frac{1}{f^0_\zeta(i\beta)^3} \re{\mc{C}} \dot\mu' -\dot Q.
        \end{aligned}
    \end{equation}

    We recognize the coefficient of $\tau$ in the first term on the right-hand side above as $D_{(w,\kappa)}\F_2(w^0,\gamma^0,\beta)$ acting on $(\surf{\dot{w}}, \dot\kappa)$. This suggests using an alternate decomposition for the codomain
    \[
        \Y \cong \parn*{\Y_1 \times \proj_1 \Y_2} \times (1-\proj_1) \Y_2 \cong \Y^0 \times (1-\proj_1) \Y_2.
    \]
    Here the isomorphism $\Y^0 \to \Y_1 \times \proj_1 \Y_2$ is given explicitly by $\operatorname{diag}(\id,\iota_2)$, where $\iota_2$ is the natural association between $\wh{\varphi}_1 \in \R$ and $2\re(i\wh{\varphi}_1 \placeholder) \in \proj_1 \Y_2$. Combining this with~\eqref{DuG1 formula} and~\eqref{DuG2 formula} gives
    \[
        D_u \G(u^0,0)
        \begin{pmatrix}
            \surf{\dot{w}} \\
            \dot\mu        \\
            \dot\gamma     \\
            \dot Q
        \end{pmatrix}
        =
        \begin{pmatrix}
            \iota D_{(w,\kappa)} \F(w^0, \gamma^0,\beta) & 0                       \\
            0                                            & D_{(\mu,Q)} \G_2(u^0,0)
        \end{pmatrix}
        \begin{pmatrix}
            \surf{\dot{w}} \\
            \dot\kappa     \\
            \dot\mu        \\
            \dot Q
        \end{pmatrix}
        ,
    \]
    where
    \[
        \iota \ceq \operatorname{diag}\parn*{\id, \, -\frac{(\gamma^0)^2}{2 \pi^2} \frac{1}{f_\zeta^0(i\beta)^2} \iota_2}.
    \]
    The upper left block on the right-hand side above is invertible by hypothesis. The lower right block is likewise invertible $\mo{C}_\imim^{k+3+\alpha} \times \R \times \R \to C_\reim^{k+2+\alpha}$. This follows from our calculation of $D_u \G_2$ in~\eqref{DuG2 formula} and the symmetry assumptions, as
    \[
        \rest[\big]{\range{\re{\mc{C} \partial_\tau }}}_{\mo{C}_\imim^{k+3+\alpha}(\T)} = \proj_{>1} C_\reim^{k+2+\alpha}(\T), \qquad \range{D_Q \G_2(u^0,0)} = \proj_0 C_\reim^{k+2+\alpha}(\T).
    \]
    Thus, $D_u \G(u^0,0)$ is indeed an isomorphism, and the proof of the lemma is complete.
\end{proof}

\begin{proof}[Proof of~\cref{intro hollow vortex theorem}]
    Let $\Curve$ be the global bifurcation curve given by~\cref{intro point vortex theorem}. As a consequence of its construction, we know that there is discrete set $\mc{N} \subset \R$ of parameter values such that
    \[
        D_{(w,\kappa)} \F(w(s), \kappa(s), \beta(s)) \quad \text{is an isomorphism } \X^0 \to \Y^0\quad \text{for } s \not\in \mc{N}.
    \]
    In the language of analytic global bifurcation theory, the portions of the curve corresponding to parameter values in connected components of $\R \setminus \mc{N}$ are called \emph{distinguished arcs}~\cite{buffoni2003analytic}. For any $s_0 \in \R \setminus \mc{N}$, the existence of a local (real-analytic) curve of wave-borne hollow vortices $\Kurve_{s_0}$ then follows directly from~\cref{G well-defined lemma,hollow vortex nondegeneracy lemma} and the (real-analytic) implicit function theorem.

    It remains to calculate the leading-order asymptotics~\eqref{hollow vortex leading order}. Writing $u^\rho = u^0 + \rho \dot u + O(\rho^2)$, we see that
    \[
        D_u \G_1(u^0,0) \dot u = - D_\rho \G_1(u^0,0), \qquad D_u \G_2(u^0,0) \dot u = - D_\rho \G_2(u^0,0).
    \]
    The formulas for $D_u \G_1(u^0,0)$ and $D_u \G_2(u^0,0)$ are given in~\eqref{DuG1 formula} and~\eqref{DuG2 formula} respectively. From~\eqref{parenthetical expansion G1} and~\eqref{definition nu0} it follows that $D_\rho \G_1(u^0,0) = 0$. Therefore we obtain
    \begin{equation}
        \label{linearized eqn1}
        - \frac{\re \parn*{ (a^0)^2 \ol{f_\zeta^0} \partial_\zeta \surf{\dot{f}} }}{\abs{f_\zeta^0}^4} + \dot\kappa \frac{a^0 \partial_\kappa a^0}{\abs{f_\zeta^0}^2}+ \frac{1}{F^2} \surf{\dot{w}} = 0.
    \end{equation}

    To compute $D_\rho \G_2(u^0,0)$ we will use the expansion~\eqref{parenthetical expansion G2}, and the fact that $f^0, f^0_{\zeta\zeta}$ are imaginary on imaginary axis, while $f^0_\zeta$ and $f^0_{\zeta\zeta\zeta}$ are real on imaginary axis. Direct computation yields that
    \[
        D_\rho \G_2(u^0,0) = - \frac{(\kappa^0)^2 (\gamma^0)^2}{16 f^0_\zeta(i\beta)^2} + (\gamma^0)^2 \brak*{ \frac{ \parn*{ 1 - 3 \csc^2{(\pi \beta)} } }{24 f^0_\zeta(i\beta)^2} - \frac{(\kappa^0)^2 (\gamma^0)^2 }{16 f^0_\zeta(i\beta)^2} - \frac{ f^0_{\zeta\zeta\zeta}(i\beta)}{8\pi^2 f^0_\zeta(i\beta)^3} } \re { \tau^2 }.
    \]
    Comparing this with~\eqref{DuG2 formula} yields
    \begin{equation}
        \label{linearized eqn2}
        \begin{aligned}
            \dot Q                                                                 & = -\frac{(\kappa^0)^2 (\gamma^0)^2}{16 f^0_\zeta(i\beta)^2},                                                                                                                                                                                  \\
            \pi i \dot\kappa                                                       & = \frac{ {f_\zeta^0(i\beta)} \partial_\zeta^2 \surf{\dot{f}}(i\beta) - f^0_{\zeta\zeta}(i\beta) {\partial_\zeta \surf{\dot{f}}(i\beta)} }{f^0_\zeta(i\beta)^2},                                                                               \\
            \frac{(\gamma^0)^2}{2 \pi^2 f^0_\zeta(i\beta)^3} \re{ \mc{C} \dot\mu'} & = (\gamma^0)^2 \brak*{ \frac{ \parn*{ 1 - 3 \csc^2{(\pi \beta)} } }{24 f^0_\zeta(i\beta)^2} - \frac{(\kappa^0)^2 (\gamma^0)^2 }{16 f^0_\zeta(i\beta)^2} - \frac{ f^0_{\zeta\zeta\zeta}(i\beta)}{8\pi^2 f^0_\zeta(i\beta)^3} } \re { \tau^2 }.
        \end{aligned}
    \end{equation}

    Local uniqueness of solutions to~\eqref{linearized eqn1} and the second equation of~\eqref{linearized eqn2} implies that $\surf{\dot{f}} = 0$ and $\dot \kappa = 0$, hence $\dot\gamma = 0$ due to~\eqref{kappa dot to gamma dot}. Together with the first equation of~\eqref{linearized eqn2} we obtain~\eqref{hollow vortex leading order parameters}. Finally, solving the third equation of~\eqref{linearized eqn2} we have
    \[
        \dot \mu = \brak*{ \pi^2 f^0_\zeta(i\beta) \parn*{ \frac{ 1 - 3 \csc^2{(\pi \beta)} }{6} - \frac{(\kappa^0)^2}{4} } - \frac{ f^0_{\zeta\zeta\zeta}(i\beta)}{2}} \re{ \tau },
    \]
    which yields the second equality in~\eqref{hollow vortex leading order f}.
\end{proof}

\section*{Acknowledgments}

The research of RMC is supported in part by the NSF through DMS-1907584 and DMS-2205910. The research of SW is supported in part by the NSF through DMS-1812436 and DMS-2306243, and the Simons Foundation through award 960210. The authors are also grateful to the Institut Mittag-Leffler, where a portion of this research was undertaken while KV, SW, and MHW were in residence as participants in the program ``Order and Randomness in Partial Differential Equations.''

\appendix

\section{Quoted results}
\label{quoted results appendix}

First, let us recall the maximum principle and Hopf boundary-point lemma, which we use extensively in studying the monotonicity properties of the waves in \cref{local monotonicity section,global monotonicity section}. In particular, note that we are using the version that allows for an adverse sign of the zeroth order term provided that the sign of the solution is known; see, for example,~\cites{fraenkel2000introduction}[Lemma~S]{gidas1979symmetry}[Lemma~1]{serrin1971symmetry}.

\begin{theorem}
    \label{max principle}
    Let $\Omega \subset \R^n$ be a connected, open set (possibly unbounded), and consider the second-order operator $L$ given by
    \[
        L \ceq \sum_{i,j = 1}^n a_{ij}(x) \partial_i \partial_j + \sum_{i=1}^n b_i(x) \partial_i + c(x)
    \]
    where $\partial_i \ceq \partial_{x_i}$ and the coefficients $a_{ij}, b_i, c \in C^0(\ol{\Omega})$. We assume that $L$ is uniformly elliptic in the sense that there exists $\lambda > 0$ with
    \[
        \sum_{ij} a_{ij}(x) \xi_i \xi_j \geq \lambda \abs{\xi}^2, \qquad \textup{for all } \xi \in \R^n, x \in \ol{\Omega},
    \]
    and that $a_{ij}$ is symmetric. Let $u \in C^2(\Omega) \cap C^1(\ol{\Omega})$ be a classical solution of $Lu = 0$ in $\Omega$.
    \begin{enumerate}[label=\textup{(\roman*)}]
        \item \textup{(Strong maximum principle)} Suppose $u$ attains its maximum value on $\ol{\Omega}$ at a point in the interior of $\Omega$. If $c \leq 0$ in $\Omega$ and $\sup_\Omega u \ge 0$, or if $\sup_{\Omega} u = 0$, then $u$ is a constant function.
              \label{strong max principle}

        \item \textup{(Hopf boundary-point lemma)} Suppose that $u$ attains its maximum value on $\ol{\Omega}$ at a point $x_0 \in \partial \Omega$ for which there exists an open ball $B \subset \Omega$ with $\ol{B} \cap \partial\Omega = \brac{ x_0 }$. Assume that either $c \leq 0$ in $\Omega$, or else $\sup_B u = 0$. Then $u$ is a constant function or
              \[
                  \nu \cdot \nabla u(x_0) > 0,
              \]
              where $\nu$ is the outward unit normal to $\Omega$ at $x_0$.
              \label{hopf lemma}
    \end{enumerate}
\end{theorem}

In order to extend the local curve of wave-borne point vortices to the large-amplitude regime, we use analytic global bifurcation theory, which was introduced by~\textcite{dancer1973bifurcation,dancer1973globalstructure} in the late 1970s, and then refined and popularized, particularly in the water waves community, by~\textcite{buffoni2003analytic}. These results were developed with an eye towards problems on bounded domains, and thus took as hypotheses that the nonlinear operator is Fredholm index $0$ and that the solution set is locally compact. For solitary waves, neither of these facts can be taken for granted. We therefore prefer to use the below version, which follows the philosophy in~\cite{chen2018existence}: for problems posed on unbounded domains, the loss of compactness should be treated as an alternative. For the present purposes, it is also more convenient to phrase the results as a global implicit function theorem. A proof can be found in~\cite[Theorem B.1]{chen2023global}.

\begin{theorem}
    \label{homoclinic global ift}
    Let $\msc{W}$ and $\msc{Z}$ be Banach spaces, $\msc{U} \subset \msc{W} \times \R$ an open set containing a point $(w_0, \lambda_0)$. Suppose that $\msc{G} \colon \msc{U} \to \msc{Z}$ is real analytic and satisfies
    \[
        \msc{G}(w_0, \lambda_0) = 0, \qquad \msc{G}_w(w_0, \lambda_0) \colon \msc{W} \to \msc{Z} \text{ is an isomorphism}.
    \]
    Then there exist a curve $\Kurve$ that admits the global $C^0$ parameterization
    \[
        \Kurve \ceq \brac*{ (w(s), \lambda(s)) : s \in \R } \subset \msc{G}^{-1}(0) \cap \msc{U},
    \]
    and satisfies the following.
    \begin{enumerate}[label=\textup{(\alph*)}]
        \item At each $s \in \R$, the linearized operator $\msc{G}_w(w(s), \lambda(s)) \colon \msc{W} \to \msc{Z}$ is Fredholm index $0$.
              \label{K well behaved}
        \item One of the following alternatives holds as $s \to \infty$ and $s \to -\infty$.
              \begin{enumerate}[label=\textup{(A\arabic*$'$)}]
                  \item \textup{(Blowup)} The quantity
                        \[
                            N(s)\ceq \norm{w(s)}_{\msc{W}} + \abs{\lambda(s)} +\frac{1}{\dist\parn[\big]{(w(s),\lambda(s)), \partial \msc{U}}}
                        \]
                        is unbounded.
                        \label{K blowup alternative}
                  \item \textup{(Loss of compactness)} There exists a sequence $s_n \to \pm\infty$ with $\sup_n{N(s_n)} < \infty$, but $( w(s_n), \lambda(s_n) )$ has no convergent subsequence in $\msc{W} \times \R$.
                        \label{K loss of compactness alternative}
                  \item \textup{(Loss of Fredholmness)} There exists a sequence $s_n \to \pm\infty$
                        with $\sup_{n}{N(s_n)} < \infty$ and so that $(w(s_n), \lambda(s_n)) \to (w_*, \lambda_*) \in \msc{U}$ in $\msc{W} \times \R$, however $\msc{G}_w(w_*, \lambda_*)$ is not Fredholm index $0$.
                        \label{K loss of fredholmness alternative}
                  \item \textup{(Closed loop)} There exists $T > 0$ such that $(w(s+T), \lambda(s+T)) = (w(s), \lambda(s))$ for all $s \in (0,\infty)$.
                        \label{K loop}
              \end{enumerate}
              \label{K alternatives}
        \item Near each point $(w(s_0),\lambda(s_0)) \in \Kurve$, we can locally reparametrize $\Kurve$ so that $s\mapsto (w(s),\lambda(s))$ is real analytic.
              \label{K reparam}
        \item The curve $\Kurve$ is maximal in the sense that, if $\msc{J} \subset \msc{G}^{-1}(0) \cap \msc{U}$ is a locally real-analytic curve containing $(w_0,\lambda_0)$ and along which $\msc G_w$ is Fredholm index $0$, then $\msc{J} \subset \Kurve$.
              \label{K maximal part}
    \end{enumerate}
\end{theorem}

Finally, we make use of the following standard bordering lemma. A proof can be found, for example, in~\cite[Lemma 2.3]{beyn1990numerical}.

\begin{lemma}[Fredholm bordering]
    \label{bordering lemma} Let $\msc{W}$ and $\msc{Z}$ be Banach spaces and suppose that $\msc{A} \colon \msc{W} \to \msc{Z}$, $\msc{B} \colon \R^m \to \msc{Z}$, $\mc{C} \colon \msc{W} \to \R^n$, and $\mc{D} \colon \R^m \to \R^n$ are bounded linear mappings. If, in addition, $\msc{A}$ is Fredholm index $k$, then the operator matrix
    \[
        \begin{pmatrix}
            \msc{A} & \msc{B} \\ \mc{C} & \mc{D}
        \end{pmatrix} \colon \msc{W} \times \R^m \to \msc{Z} \times \R^n
    \]
    is Fredholm with index $k + m - n$.
\end{lemma}

\printbibliography[title=References]

\end{document}

%% file: figures/cyclopean_wave.tikz
\begin{tikzpicture}[use Hobby shortcut,scale=0.7 * \textwidth / 10cm]]
    \pgfmathsetmacro{\xmax}{5}


    \coordinate (L) at (-\xmax,1);
    \coordinate (LM) at (-0.5*\xmax,1.09);
    \coordinate (CLL) at (-1.5,1.75);
    \coordinate (CL) at (-1.75,2.5);
    \coordinate (CT) at (0,4);
    \coordinate (CR) at (1.75,2.5);
    \coordinate (CLR) at (1.5,1.75);
    \coordinate (RM) at (0.5*\xmax,1.08);
    \coordinate (R) at (\xmax,1);

    \coordinate (P) at (0,2.75);
    \coordinate (PL) at ($(P)+(-0.75,0)$);
    \coordinate (PR) at ($(P)+(0.75,0)$);
    \coordinate (PT) at ($(P)+(0,0.5)$);
    \coordinate (PB) at ($(P)+(0,-0.75)$);

    \fill[black!10, even odd rule] ([out angle=0, in angle=180]L) .. (LM).. (CLL) .. (CL) .. (CT) ..(CR) .. (CLR) .. (RM) .. (R) -- (\xmax,0) -- (-\xmax,0) -- cycle
    ([closed]PT) .. (PR) .. (PB) .. (PL);

    \draw[thick] ([out angle=0, in angle=180]L) .. (LM).. (CLL) .. (CL) .. (CT) ..(CR) .. (CLR) .. (RM) .. (R);
    \node[above left] at (R) {$p=0$};

    \draw[thick] ([closed]PT) .. (PR) .. (PB) .. (PL);
    \node at ($(P) + (0,-0.05)$) {$p \text{ const.}$};

    \draw[densely dotted, postaction={decorate,decoration={markings,mark=between positions 0 and 0.75 step 0.25 with {\arrow{Stealth}}}}] ([closed]$1.4*(PT) - 0.4*(P)$) .. ($1.4*(PR) - 0.4*(P)$) .. ($1.4*(PB) - 0.4*(P)$) .. ($1.4*(PL) - 0.4*(P)$);

    \draw[style={postaction={draw,decorate,decoration={border,angle=-45,
                                amplitude=5,segment length=4}}}] (-\xmax,0) -- (\xmax,0);


    \draw[-Latex] (-0.6*\xmax,3) -- (-0.6*\xmax,2) node[midway, left] {$g$};
\end{tikzpicture}

%% file: figures/point_vortex_wave.tikz
\begin{tikzpicture}[use Hobby shortcut,scale=0.95*\textwidth / 14cm]
    \pgfmathsetmacro{\posconf}{0.4}
    \pgfmathsetmacro{\posphys}{0.8}
    \pgfmathsetmacro{\sep}{8}
    \pgfmathsetmacro{\xmax}{3}

    \coordinate (Oc) at (- \sep / 2,0); 
    \coordinate (Pc) at ($(Oc) + (0,\posconf)$); 
    \coordinate (Qc) at ($(Oc) - (0,\posconf)$); 

    \fill[black!10] ($(Oc) + (-\xmax,0)$) -- ($(Oc) + (\xmax,0)$) -- ($(Oc) + (\xmax,1)$) -- ($(Oc) + (-\xmax,1)$) -- cycle;

    \draw[ densely dotted,-Stealth] ($(Oc) + (-\xmax,0)$) -- ($(Oc) + (\xmax + 0.15,0)$) node[below left] {$\xi$};

    \draw[densely dotted,-Stealth] ($(Oc) + (0,-1)$) -- ($(Oc) + (0,1.25)$) node[right] {$\eta$};

    \draw[thick] ($(Oc) + (-\xmax,1)$) node[above right] {$\Gamma$} -- ($(Oc) + (\xmax,1)$) node[above left]{$\eta = 1$};
    \draw ($(Oc) + (-\xmax,-1)$) -- ($(Oc) + (\xmax,-1)$);

    \filldraw (Pc) circle[shift only, radius=1.25pt] node[right] {$i \beta$};
    \filldraw (Qc) circle[shift only, radius=1.25pt];

    \node at ($(Oc) + (-\xmax/2,0.5)$) {$\Omega_+$};

    \coordinate (Op) at (\sep / 2,0);
    \coordinate (Pp) at ($(Op) + (0,\posphys)$); 
    \coordinate (Qp) at ($(Op) - (0,\posphys)$); 

    \fill[black!10] ([out angle=0, in angle=180]$(Op) + (-\xmax,1)$) node[above right] {$\mathscr{S}$}..($(Op) + (-0.5*\xmax,1.1)$) .. ($(Op) + (0,1.5)$) .. ($(Op) + (0.5*\xmax,1.1)$).. ($(Op) + (\xmax,1)$) -- ($(Op) + (\xmax,0)$) -- ($(Op) + (-\xmax,0)$) -- cycle;

    \draw[densely dotted,-Stealth] ($(Op) + (-\xmax,0)$) node[below right] {$\mathscr{B}$} -- ($(Op) + (\xmax + 0.15,0)$) node[below left] {$x$};

    \draw[densely dotted,-Stealth] ($(Op) + (0,-1.5)$) -- ($(Op) + (0,1.75)$) node[right] {$y$};

    \draw[thick] ([out angle=0, in angle=180]$(Op) + (-\xmax,1)$) node[above right] {$\mathscr{S}$}..($(Op) + (-0.5*\xmax,1.1)$) .. ($(Op) + (0,1.5)$) .. ($(Op) + (0.5*\xmax,1.1)$).. ($(Op) + (\xmax,1)$);

    \draw ([out angle=0, in angle=180]$(Op) + (-\xmax,-1)$)..($(Op) + (-0.5*\xmax,-1.1)$) .. ($(Op) + (0,-1.5)$) .. ($(Op) + (0.5*\xmax,-1.1)$).. ($(Op) + (\xmax,-1)$);

    \filldraw (Pp) circle[shift only, radius=1.25pt] node[right] {$i b$};
    \filldraw (Qp) circle[shift only, radius=1.25pt];

    \node at ($(Op) + (-\xmax/2,0.5)$) {$\mathscr{D}$};

    \draw[thick,->] ($(Oc) + (\xmax,0) + (0.25,0.5)$) [out angle=30] .. node[midway,above] {$f$} ($(Op) + (-\xmax,0) + (-0.25,0.5)$);

\end{tikzpicture}

%% file: figures/hollow_vortex_wave.tikz
\begin{tikzpicture}[use Hobby shortcut,scale=0.95*\textwidth / 14cm]
    \pgfmathsetmacro{\sep}{8}
    \pgfmathsetmacro{\xmax}{3}

    \coordinate (Oc) at (- \sep / 2,0); 
    \coordinate (Pc) at ($(Oc) + (0,0.55)$); 

    \fill[black!10, even odd rule] ($(Oc) + (-\xmax,0)$) -- ($(Oc) + (\xmax,0)$) -- ($(Oc) + (\xmax,1)$) -- ($(Oc) + (-\xmax,1)$) -- cycle
    (Pc) circle (0.25);

    \node at ($(Oc) + (-\xmax/2,0.5)$) {$\Omega_{\rho,+}$};

    \draw[-Stealth] ($(Oc) + (-\xmax,0)$) -- ($(Oc) + (\xmax + 0.15,0)$) node[below left] {$\xi$};

    \draw[densely dotted] (Oc) -- ($(Pc)+ (0,-0.25)$);
    \draw[densely dotted,-Stealth] ($(Pc)+ (0,0.25)$) -- ($(Oc) + (0,1.25)$) node[right] {$\eta$};

    \draw[thick] ($(Oc) + (-\xmax,1)$) node[above right] {$\confSs$} -- ($(Oc) + (\xmax,1)$) node[above left]{$\eta = 1$};

    \draw[thick] (Pc) circle (0.25);
    \node[left] at ($(Pc)+ (-0.25,0)$) {$\confSv$};
    \draw[-Latex] ([in angle=30]$(Pc) + (1,1)$) node[above] {$B_\rho(i \beta)$} .. (Pc);

    \coordinate (Op) at ( \sep / 2,0); 
    \coordinate (Pp) at ($(Op) + (0,1.1)$); 

    \fill[black!10, even odd rule] ([out angle=0, in angle=180]$(Op) + (-\xmax,1)$) ..($(Op) + (-0.5*\xmax,1.2)$) .. ($(Op) + (0,2)$) .. ($(Op) + (0.5*\xmax,1.2)$).. ($(Op) + (\xmax,1)$) -- ($(Op) + (\xmax,0)$) -- ($(Op) + (-\xmax,0)$) -- cycle
    ([closed]$(Pp) + (0,0.35)$) .. ($(Pp) + (0.5,0)$) .. ($(Pp) + (0,-0.4)$) .. ($(Pp) + (-0.5,0)$);

    \draw[-Stealth] ($(Op) + (-\xmax,0)$) node[below right] {$\mathscr{B}$} -- ($(Op) + (\xmax + 0.15,0)$) node[below left] {$x$};

    \draw[densely dotted] (Op) -- ($(Pp) + (0,-0.4)$);
    \draw[densely dotted,-Stealth] ($(Pp) + (0,0.4)$) -- ($(Op) + (0,2.25)$) node[right] {$y$};

    \draw[thick] ([out angle=0, in angle=180]$(Op) + (-\xmax,1)$) node[above right] {$\mathscr{S}$}..($(Op) + (-0.5*\xmax,1.2)$) .. ($(Op) + (0,2)$) .. ($(Op) + (0.5*\xmax,1.2)$).. ($(Op) + (\xmax,1)$);

    \node at (Pp) {$\mathscr{V}$};
    \draw[thick] ([closed]$(Pp) + (0,0.35)$) .. ($(Pp) + (0.5,0)$) .. ($(Pp) + (0,-0.4)$) .. ($(Pp) + (-0.5,0)$);

    \node at ($(Op) + (-\xmax/2,0.5)$) {$\mathscr{D}$};

    \draw[thick,->] ($(Oc) + (\xmax,0) + (0.25,0.5)$) [out angle=30] .. node[midway,above] {$f$} ($(Op) + (-\xmax,0) + (-0.25,0.5)$);

\end{tikzpicture}

%% file: figures/conformal_monotonicity.tikz
\begin{tikzpicture}[use Hobby shortcut,scale=1.5]
    \pgfmathsetmacro{\vpos}{0.4}
    \pgfmathsetmacro{\xmax}{5}
    \pgfmathsetmacro{\Mval}{1}

    \fill[black!10] ($(\Mval,0)$) rectangle ($(\xmax,1)$) node[midway,black] {$\Omega_+^M$};

    \draw[-Stealth] (-1,0) -- ($(\xmax,0)$) node[below] {$\xi$};

    \node [below] at ($(\Mval,0)$) {$M$};
    \draw ($(\Mval,-2pt)$) --($(\Mval,2pt)$);

    \draw[ densely dotted,-Stealth] (0,0) -- (0, 1.25) node[left] {$\eta$};

    \draw (-1,1) -- ($(\Mval,1)$);

    \draw[thick] ($(\Mval,0)$) -- ($(\xmax,0)$) node[midway,below] {$B^M$};

    \draw[thick, densely dashed] ($(\Mval,0)$) -- ($(\Mval,1)$) node[midway,left] {$L^M$};

    \draw[thick, densely dotted] ($(\Mval,1)$) -- ($(\xmax,1)$) node[midway,above] {$\Gamma^M$};

    \filldraw ($(0,\vpos)$) circle[shift only, radius=1.25pt];

\end{tikzpicture}

%% file: figures/pseudo_arc_length_continuation.tikz
\begin{tikzpicture}[use Hobby shortcut]
    \pgfmathsetmacro{\psize}{0.05}

    \coordinate (A) at (0,0);
    \coordinate (Atl) at (-0.943,-0.332);
    \coordinate (Atr) at (0.943,0.332);
    \coordinate (B) at (2.42672,0.71252);
    \coordinate (Btl) at (1.45,0.497);
    \coordinate (Btr) at (3.403,0.928);
    \coordinate (C) at (5,1);
    \coordinate (Ctl) at (4,1.014);
    \coordinate (Ctr) at (6,0.986);
    \coordinate (Dp) at (7.659,0.791);
    \coordinate (Dptl) at (6.674,0.964);
    \coordinate (Dptr) at (8.644,0.618);
    \coordinate (D) at (7.615,0.54);

    \draw[thick] (A) .. (B) .. (C) .. (D);
    \draw[thick, densely dashed, shorten >= -.75cm, shorten <= -.75cm] (Dp) -- (D);

    \node[below] at (B) {$\Theta_{k-2}^M$};
    \node[below] at (C) {$\Theta_{k-1}^M$};

    \draw[thick,black!50] (Atl) -- (Atr);
    \draw[thick,black!50] (Btl) -- (Btr);
    \draw[thick,black!50] (Ctl) -- (Ctr);
    \draw[thick,-Latex] (Dp) -- (Dptr) node[midway,above] {$\tilde{n}_k^M$};

    \fill (A) circle (\psize);
    \fill (B) circle (\psize);
    \fill (C) circle (\psize);
    \draw (Dp) circle (\psize) node [above left] {$\tilde{\Theta}_k^M$};
    \fill (D) circle (\psize);

\end{tikzpicture}

%% file: figures/F2_highest.tikz
\begin{tikzpicture}[trim axis left, trim axis right]
    \tracinglostchars=0\relax 
    \begin{axis}[
            axis on top,
            width=\textwidth,
            axis lines=left,
            axis line style={-Latex},
            x label style={at={(axis description cs:1,0)},anchor= north},
            y label style={at={(axis description cs:0,1)},anchor= east,rotate=-90},
            xlabel={$x$},
            xtick = {-5,0,5},
            ylabel={$y$},
            ytick={1,2},
            xmin=-6,
            xmax=6,
            ymin=0,
            ymax=3,
            enlarge y limits=false,
            axis equal image,
            colormap/blackwhite,]
        \addplot[contour prepared, contour prepared format=standard, contour/labels=false, contour/draw color=black!50] table {figures/l_5_F_2_b_286_N_2_10_contours.txt};
        \addplot[thick,mark=none] table{figures/l_5_F_2_b_286_N_2_10_interface.txt};
        \filldraw (0,0.6063649674380741) circle (0.03cm);
        \coordinate (top) at (axis cs: 0,3.012);
    \end{axis}
    \begin{scope}[on background layer]
        \draw[black!50] ($(top) + (-30:0.5)$) arc (-30:0:0.5);
        \node[black!50] at ($(top) + (-12:0.85)$) {\footnotesize$\SI{30}{\degree}$};
        \draw[thick,black!50] ($(top)+(-150:{sqrt(2)})$) -- (top) -- ($(top)+(-30:{sqrt(2)})$);
        \draw[black!50] (top) -- ($(top)+(0:{sqrt(3)/sqrt(2)})$);
    \end{scope}
\end{tikzpicture}

%% file: figures/F2_beta_kappa_curve.tikz
\begin{tikzpicture}[trim left=(full.south west),trim right=(full.south east)]
    \tracinglostchars=0\relax 
    \begin{axis}[
            axis on top,
            unit vector ratio = 1 1.75,
            width=\textwidth,
            axis lines=center,
            axis line style={-Latex},
            y label style={at={(axis description cs:0,1)},anchor= east,},
            yticklabel style={anchor = west},
            xlabel={$\beta$},
            ylabel={$\kappa$},
            xmin=0,
            xmax=0.3,
            xtick={0.3},
            ytick={-0.1},
            ymin=-0.1,
            ymax=0.01,
            enlarge y limits=false,
            colormap/blackwhite,
            name=full]
        \addplot[thick,mark=none] table{figures/l_5_F_2_b_286_N_2_10_beta_kappa_curve.txt} coordinate[pos=1] (end);
        \draw (end) circle (0.05cm);
        \coordinate (spy_end_ur) at (axis cs:0.2957,-0.09314);
        \coordinate (spy_end_ll) at (axis cs:0.2807,-0.09914);
        \coordinate (spy_end_ul) at (axis cs:0.2807,-0.09314);
    \end{axis}
    \draw[black!50,thick] (spy_end_ll) rectangle (spy_end_ur);
    \begin{scope}[xshift=0.5cm,yshift=.5cm]
        \begin{axis}[
                unit vector ratio = 1 1.75,
                width=3.5cm,
                ticks=none,
                xmin=0.2807,
                xmax=0.2957,
                ymin=-0.09914,
                ymax=-0.09314,
                colormap/blackwhite,
                name=magnified
            ]
            \addplot[thick,mark=none] table{figures/l_5_F_2_b_286_N_2_10_beta_kappa_curve.txt} coordinate[pos=1] (end_magnified);
            \draw (end_magnified) circle (0.05cm);
            \coordinate (magnifying_glass_ll) at (axis cs:0.2957,-0.09914);
        \end{axis}
    \end{scope}
    \draw[densely dotted] (magnifying_glass_ll) -- (spy_end_ul);
\end{tikzpicture}

%% file: figures/F5_highest.tikz
\begin{tikzpicture}[trim axis left,trim axis right]
    \tracinglostchars=0\relax 
    \begin{axis}[
            axis on top,
            width=\textwidth,
            axis lines=left,
            axis line style={-Latex},
            x label style={at={(axis description cs:1,0)},anchor= north},
            y label style={at={(axis description cs:0,1)},anchor= east,rotate=-90},
            xlabel={$x$},
            xtick = {-8,0,8},
            ylabel={$y$},
            ytick={1,10},
            xmin=-10,
            xmax=10,
            ymin=0,
            ymax=14,
            enlarge y limits=false,
            axis equal image,
            colormap/blackwhite,]
        \addplot[contour prepared, contour prepared format=standard, contour/labels=false, contour/draw color=black!50] table {figures/l_8_F_5_b_656_N_2_11_contours.txt};
        \addplot[thick,mark=none] table{figures/l_8_F_5_b_656_N_2_11_interface.txt};
        \filldraw (0,4.708318699952704) circle (0.03cm);
    \end{axis}
\end{tikzpicture}

%% file: figures/F8.tikz
\begin{tikzpicture}[trim axis left, trim axis right]
    \tracinglostchars=0\relax 
    \begin{axis}[
            axis on top,
            width=\textwidth,
            axis lines=left,
            axis line style={-Latex},
            x label style={at={(axis description cs:1,0)},anchor= north},
            y label style={at={(axis description cs:0,1)},anchor= east,rotate=-90},
            xlabel={$x$},
            xtick = {-10,0,10},
            ylabel={$y$},
            ytick={1,20},
            xmin=-15,
            xmax=15,
            ymin=0,
            ymax=26,
            enlarge y limits=false,
            axis equal image,
            colormap/blackwhite,]
        \addplot[contour prepared, contour prepared format=standard, contour/labels=false, contour/draw color=black!50] table {figures/l_10_F_8_b_884_N_2_11_contours.txt};
        \addplot[thick,mark=none] table{figures/l_10_F_8_b_884_N_2_11_interface.txt};
        \filldraw (0,12.62477469960786) circle (0.03cm);
    \end{axis}
\end{tikzpicture}

%% file: figures/numerical_convergence.tikz
\begin{tikzpicture}
    \tracinglostchars=0\relax 
    \begin{axis}[
            axis on top,
            width=10cm,
            axis lines=left,
            axis line style={-Latex},
            x label style={at={(axis description cs:1,0)},anchor= north},
            y label style={at={(axis description cs:0,1)},anchor= east,rotate=-90},
            xlabel={$x$},
            xtick = {-30,0,30},
            ylabel={$y$},
            ytick={1,30},
            xmin=-34,
            xmax=34,
            ymin==0,
            ymax=41,
            enlarge y limits=false,
            axis equal image,
            colormap/blackwhite,
            legend entries={$2^8$,$2^9$,$2^{10}$,$2^{11}$},
            legend style={at={(axis cs:-30,30)},anchor=west,draw=none,fill=none}]
        \addplot[forget plot,contour prepared, contour prepared format=standard, contour/labels=false, contour/draw color=black!50] table {figures/l_20_F_inf_b_75_N_2_11_contours.txt};
        \addplot[mark=triangle,mark size=.75pt,only marks] table{figures/l_20_F_inf_b_75_N_2_8_interface.txt};
        \addplot[mark=square,mark size=.75pt,only marks] table{figures/l_20_F_inf_b_75_N_2_9_interface.txt};
        \addplot[mark=o,mark size=.75pt,only marks] table{figures/l_20_F_inf_b_75_N_2_10_interface.txt};
        \addplot[mark=*,mark size=.75pt,only marks] table{figures/l_20_F_inf_b_75_N_2_11_interface.txt};
        \addplot[mark=none] table{figures/F_inf_b_75_N_2_11_exact_interface.txt};
        \filldraw (0,18.603683128409692) circle (0.03cm); 
        \coordinate (spy_overturn_ll) at (axis cs:15.8,1);
        \coordinate (spy_overturn_ur) at (axis cs:17,3);
        \coordinate (spy_top_ll) at (axis cs:-3.5,38);
        \coordinate (spy_top_ur) at (axis cs:3.5,40);
    \end{axis}
    \draw[black!50,thick] (spy_overturn_ll) rectangle (spy_overturn_ur);
    \begin{scope}[xshift=7.3cm,yshift=1cm]
        \begin{axis}[
                width=5cm,
                ticks=none,
                xmin=15.8,
                xmax=17,
                ymin=1,
                ymax=3,
                axis equal image,
                colormap/blackwhite]
            \addplot[mark=triangle,mark size=.75pt,draw=none] table{figures/l_20_F_inf_b_75_N_2_8_interface.txt};
            \addplot[mark=square,mark size=.75pt,draw=none] table{figures/l_20_F_inf_b_75_N_2_9_interface.txt};
            \addplot[mark=o,black!75,mark size=.75pt,draw=none] table{figures/l_20_F_inf_b_75_N_2_10_interface.txt};
            \addplot[mark=*,mark size=.75pt,draw=none] table{figures/l_20_F_inf_b_75_N_2_11_interface.txt};
            \addplot[mark=none,smooth] table{figures/F_inf_b_75_N_2_11_exact_interface.txt};
            \coordinate (magnifying_glass_overturn_ll) at (axis cs:15.8,1);
        \end{axis}
    \end{scope}
    \draw[densely dotted] (spy_overturn_ur) -- (magnifying_glass_overturn_ll);

    \draw[black!50,thick] (spy_top_ll) rectangle (spy_top_ur);
    \begin{scope}[xshift=6.4cm,yshift=4.2cm]
        \begin{axis}[
                width=5cm,
                ticks=none,
                xmin=-3.5,
                xmax=3.5,
                ymin=38,
                ymax=40,
                axis equal image,
                colormap/blackwhite]
            \addplot[mark=triangle,mark size=.75pt,draw=none] table{figures/l_20_F_inf_b_75_N_2_8_interface.txt};
            \addplot[mark=square,mark size=.75pt,draw=none] table{figures/l_20_F_inf_b_75_N_2_9_interface.txt};
            \addplot[mark=o,black!75,mark size=.75pt,draw=none] table{figures/l_20_F_inf_b_75_N_2_10_interface.txt};
            \addplot[mark=*,mark size=.75pt,draw=none] table{figures/l_20_F_inf_b_75_N_2_11_interface.txt};
            \addplot[mark=none,smooth] table{figures/F_inf_b_75_N_2_11_exact_interface.txt};
            \coordinate (magnifying_glass_top_ul) at (axis cs:-3.5,40);
        \end{axis}
    \end{scope}
    \draw[densely dotted] (spy_top_ur) -- (magnifying_glass_top_ul);
\end{tikzpicture}